\documentclass[10pt]{article}

\usepackage[utf8]{inputenc}
\usepackage[T1]{fontenc}
\usepackage[english]{babel}

\usepackage{pgf,tikz}
\usepackage{caption}
\usepackage[outline]{contour}
\contourlength{1.2pt}
\usepackage{amsmath}
\usepackage{amssymb}
\usepackage{amsfonts}
\usepackage{mathtools}
\usepackage{xcolor}
\usepackage{amsthm}
\usepackage{hyperref}
\usepackage{xurl}
\hypersetup{breaklinks=true}

\theoremstyle{plain}
\newtheorem{theorem}{Theorem}

\newtheorem{lemma}[theorem]{Lemma}

\theoremstyle{definition}
\newtheorem{definition}[theorem]{Definition}
\newtheorem{remark}[theorem]{Remark}

\def\vtx{v}
\def\teegy{t_1}
\def\teketto{t_2}

\def\teegyvesszo{\theta_1}
\def\tekettovesszo{\theta_2}
\def\teivesszo{\theta_i}
\DeclareMathOperator{\curv}{\text{curv}}

\title{A lower bound on the number of colours needed to nicely colour a sphere}

\author{Péter Ágoston\thanks{MTA-ELTE Lendület Combinatorial Geometry (CoGe) Research Group, Eötvös Loránd University, Budapest, Hungary and Alfréd Rényi Institute of Mathematics, supported by the European Union, co-financed by the European Social Fund (EFOP-3.6.3-VEKOP-16-2017- 00002) and by the Ministry of Innovation and Technology NRDI Office within the framework of the Artificial Intelligence National Laboratory.  (RRF-2.3.1-21-2022-00004), e-mail address: agostonp@renyi.hu.}}

\index{Agoston, Peter}

\begin{document}
\thispagestyle{empty}

\maketitle

\begin{abstract}
The Hadwiger--Nelson problem is about determining the chromatic number of the plane (CNP), defined as the minimum number of colours needed to colour the plane so that no two points of distance $1$ have the same colour. In this paper we investigate a related problem for spheres and we use a few natural restrictions on the colouring. Thomassen showed that with these restrictions, the chromatic number of all manifolds satisfying certain properties (including the plane and all spheres with a large enough radius) is at least $7$. We prove that with these restrictions, the chromatic number of any sphere with a large enough radius is at least $8$. This also gives a new lower bound for the minimum colours needed for colouring the $3$-dimensional space with the same restrictions.
\end{abstract}

\section{Introduction}

\subsection{Colourings of the plane}

\begin{center}
\begin{minipage}{.24\textwidth}
	\centering
	\usetikzlibrary{arrows}
%\pagestyle{empty}
%\begin{document}
\begin{tikzpicture}[line cap=round,line join=round]%,>=triangle 45,x=1.0cm,y=1.0cm]
\clip(-1,-0.2) rectangle (1,1.9);
\draw (0,1.66)-- (-0.73,0.97);
\draw (-0.73,0.97)-- (-0.5,0);
\draw (-0.5,0)-- (0.23,0.68);
\draw (0.23,0.68)-- (0,1.66);
\draw (-0.73,0.97)-- (0.23,0.68);
\draw (0,1.66)-- (-0.23,0.68);
\draw (-0.23,0.68)-- (0.73,0.97);
\draw (0.73,0.97)-- (0,1.66);
\draw (-0.23,0.68)-- (0.5,0);
\draw (0.73,0.97)-- (0.5,0);
\draw (-0.5,0)-- (0.5,0);
\begin{scriptsize}
\fill [color=black] (0,1.66) circle (1.5pt);
%\draw[color=black] (0.03,1.85) node {$A$};
\fill [color=black] (-0.5,0) circle (1.5pt);
%\draw[color=black] (-0.72,0.04) node {$D$};
\fill [color=black] (0.5,0) circle (1.5pt);
%\draw[color=black] (0.73,0.04) node {$G$};
\fill [color=black] (-0.73,0.97) circle (1.5pt);
%\draw[color=black] (-0.9,1.02) node {$B$};
\fill [color=black] (0.23,0.68) circle (1.5pt);
%\draw[color=black] (0.46,0.63) node {$C$};
\fill [color=black] (-0.23,0.68) circle (1.5pt);
%\draw[color=black] (-0.4,0.63) node {$E$};
\fill [color=black] (0.73,0.97) circle (1.5pt);
%\draw[color=black] (0.95,1.02) node {$F$};
\end{scriptsize}
\end{tikzpicture}
%\end{document}
	\captionof{figure}{}
	\label{fig:Moserspindle}
\end{minipage}
\begin{minipage}{.68\textwidth}
	\centering
	\input{Figures/Isbell}
	\captionof{figure}{}
	\label{fig:Isbell}
\end{minipage}
\end{center}

The Hadwiger--Nelson problem is a well-known problem in combinatorial geometry. It asks to determine the chromatic number of the plane (CNP), i.e., the minimum number of colours needed to colour the plane so that no two points of distance $1$ have the same colour. Alternatively, it is the chromatic number of the graph of unit distances on the plane. Since 1950 it has been known that $4\le CNP\le7$. The lower bound was obtained by Nelson (1950), but it can be most easily proven by using a graph called the Moser spindle (Figure \ref{fig:Moserspindle}) (Moser, Moser (1961)~\cite{mm}), while the upper bound was given by Isbell (1950), using the colouring in Figure \ref{fig:Isbell}. Since 2018, it is also known that $CNP\ge5$ (de Grey \cite{dg}, Exoo, Ismailescu \cite{ei}).

\subsection{Tilings}

The problem has some variations, in which we require some additional conditions from the colour classes:

A natural restriction is if we require from the colour classes to be measurable, the best known lower bound for the number of colours needed is also $5$ (Falconer (1981) \cite{f}) and the best known upper bound is also $7$ (also from Figure \ref{fig:Isbell}).

A stronger condition is if we require the colouring to be generated in the following way:

Divide the plane into regions using Jordan curves and let the colour classes be the unions of such regions. We call the regions in such a partition of the plane {\bf tiles} and call such a colouring a {\bf tiling}. This definition itself has many variants and ambiguities:

1) Can we use infinitely many Jordan curves inside a bounded region?

2) Can we use infinitely many regions inside a bounded region?

3) Can regions be not simply connected? If so, do we allow the complement of a region to have infinitely many connected components?

4) What about the borders? It is natural to require all of their points to have the same colour as one of the bordering tiles, but it still can be a question whether there should be at least some regularity of their colourings or not.

The best known lower bound for the number of colours needed in a tiling of the plane is $6$ (Townsend (2005) \cite{t}) and the best known upper bound is also $7$. Note that the definition of tiling used in Townsend's proof uses the most generous answer for all of the above questions.

We now define the notion we will mean by tiling.

\begin{definition}
Take a family $\mathcal{T}$ of sets $T_i$ ($i\in I$ for some index set $I$), with the following properties:

1) Every $T_i$ has the union of finitely many disjoint simple closed Jordan curves as its boundary.

2) $\dot{\bigcup\limits_{i\in I}}{T_i}=\mathbb{R}^2$.

Also take a function $f: I\rightarrow\left\lbrace 1,2,...,n\right\rbrace$ for some $n\in\mathbb{N}$ satisfying the property that for any pair $\left(T_i,T_j\right)$ of tiles with points $p\in T_i$, $p'\in T_j$, $\left\lvert pp'\right\rvert=1$, $f(i)\neq f(j)$.

Then the pair $\left(\mathcal{T},f\right)$ is called an $n$-tiling, the $T_i$ are called the tiles in this tiling and $f$ is called the colouring.
\end{definition}

Thomassen also defined a type of tiling:

A colouring of a surface with a metric is \emph{nice}, if it is a tiling, all tiles have diameter less than $1$, all pairs of tiles with the same colour have distance more than $1$ and all tiles are simply connected.

He also proved the following theorem:

\begin{theorem}[Thomassen \cite{t}]
\label{thm:thomassen}
Suppose a surface $S$ satisfies the following three conditions for some natural number $k$:

1. Every noncontractible simple closed curve has diameter at least $2$.

2. If $C$ is a simple closed curve of diameter less than $2$, then the area of $int(C)$ is at most $k$.

3. The diameter of $S$ is at least $12k+30$.

Then every nice tiling contains at least $7$ colours.
\end{theorem}

Since the plane satisfies the conditions, the theorem proves that every nice tiling of the plane contains at least $7$ colours.

In this paper, we will use a weaker condition for a \emph{nice colouring} (or by an alternative name, a \emph{nice tiling}):

A colouring of a surface with a metric is \emph{nice}, if it is a tiling, all tiles have diameter less than $1$ and all pairs of tiles with the same colour have distance more than $1$.

Note that this definition contains two restrictions compared to a general tiling:

\begin{center}
\begin{minipage}{.48\textwidth}
	\centering
	\usetikzlibrary{arrows}
%\pagestyle{empty}
%\begin{document}
\definecolor{ffqqqq}{rgb}{1,0,0}
\definecolor{wwffff}{rgb}{0.4,1,1}
\begin{tikzpicture}[line cap=round,line join=round]%,>=triangle 45,x=1.0cm,y=1.0cm]
\clip(6.65,-1.57) rectangle (9.09,0.03);
\fill[line width=0pt,color=wwffff,fill=wwffff,fill opacity=1.0] (7.58,-0.82) -- (7.93,-1.09) -- (8.49,-1.24) -- cycle;
\fill[line width=0pt,color=wwffff,fill=wwffff,fill opacity=1.0] (7.58,-0.82) -- (8.5,-0.44) -- (8.02,-0.45) -- cycle;
\fill[line width=0pt,color=ffqqqq,fill=ffqqqq,fill opacity=1.0] (8.49,-1.24) -- (8.82,-1.34) -- (8.95,-0.74) -- (8.77,-0.2) -- (8.5,-0.44) -- cycle;
\fill[line width=0pt,color=ffqqqq,fill=ffqqqq,fill opacity=1.0] (7.25,-1.38) -- (7.58,-0.82) -- (7.19,-0.33) -- (6.8,-0.95) -- cycle;
\draw [shift={(7.58,-0.82)},line width=0.4pt,color=wwffff,fill=wwffff,fill opacity=1.0]  (0,0) --  plot[domain=-0.44:0.4,variable=\t]({1*1*cos(\t r)+0*1*sin(\t r)},{0*1*cos(\t r)+1*1*sin(\t r)}) -- cycle ;
\draw (7.58,-0.82)-- (7.19,-0.33);
\draw (7.19,-0.33)-- (6.8,-0.95);
\draw (6.8,-0.95)-- (7.25,-1.38);
\draw (7.25,-1.38)-- (7.58,-0.82);
\draw (7.58,-0.82)-- (7.93,-1.09);
\draw (7.93,-1.09)-- (8.49,-1.24);
\draw (8.49,-1.24)-- (8.82,-1.34);
\draw (8.82,-1.34)-- (8.95,-0.74);
\draw (8.95,-0.74)-- (8.77,-0.2);
\draw (8.77,-0.2)-- (8.5,-0.44);
\draw (8.5,-0.44)-- (8.02,-0.45);
\draw (8.02,-0.45)-- (7.58,-0.82);
\draw [shift={(7.58,-0.82)}] plot[domain=-0.44:0.4,variable=\t]({1*1*cos(\t r)+0*1*sin(\t r)},{0*1*cos(\t r)+1*1*sin(\t r)});
\draw [->] (7.58,-0.82) -- (8.58,-0.76);
\draw [->] (8.58,-0.76) -- (7.58,-0.82);
\begin{scriptsize}
\draw[color=black] (8.17,-0.69) node {$1$};
\end{scriptsize}
\end{tikzpicture}
%\end{document}
	
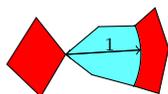
\captionof{figure}{Tiles with distance $1$}
	\label{1distancetiles}
\end{minipage}
\begin{minipage}{.48\textwidth}
	\centering
	\usetikzlibrary{arrows}
%\pagestyle{empty}
%\begin{document}
\definecolor{ffffff}{rgb}{1,1,1}
\definecolor{ffqqqq}{rgb}{1,0,0}
\begin{tikzpicture}[line cap=round,line join=round,scale=0.4]%,>=triangle 45,x=1.0cm,y=1.0cm]
\clip(3.02,-4.44) rectangle (13.16,0.8);
\fill[line width=0pt,color=ffqqqq,fill=ffqqqq,fill opacity=1.0] (3.74,-2.4) -- (5.6,-3.76) -- (7.16,-2.34) -- (6.46,0.08) -- (4.18,-0.2) -- cycle;
\fill[line width=0pt,color=ffqqqq,fill=ffqqqq,fill opacity=1.0] (9.82,-2.2) -- (11.82,-3.34) -- (12.46,-1.64) -- (10.38,-0.28) -- cycle;
\draw [shift={(5.64,-2.66)},line width=0.4pt,color=ffqqqq,fill=ffqqqq,fill opacity=1.0]  plot[domain=1.28:2.11,variable=\t]({1*2.86*cos(\t r)+0*2.86*sin(\t r)},{0*2.86*cos(\t r)+1*2.86*sin(\t r)});
\draw [shift={(5.64,-2.66)},line width=0.4pt,color=ffqqqq,fill=ffqqqq,fill opacity=1.0]  (0,0) --  plot[domain=1.28:2.11,variable=\t]({1*2.86*cos(\t r)+0*2.86*sin(\t r)},{0*2.86*cos(\t r)+1*2.86*sin(\t r)}) -- cycle ;
\draw [shift={(6.42,-0.68)},color=ffqqqq,fill=ffqqqq,fill opacity=1.0]  (0,0) --  plot[domain=3.71:4.45,variable=\t]({1*3.19*cos(\t r)+0*3.19*sin(\t r)},{0*3.19*cos(\t r)+1*3.19*sin(\t r)}) -- cycle ;
\draw [shift={(4.48,-0.96)},line width=0.4pt,color=ffqqqq,fill=ffqqqq,fill opacity=1.0]  (0,0) --  plot[domain=5.09:5.81,variable=\t]({1*3.02*cos(\t r)+0*3.02*sin(\t r)},{0*3.02*cos(\t r)+1*3.02*sin(\t r)}) -- cycle ;
\draw [shift={(9.38,-4.08)},line width=0.4pt,color=ffqqqq,fill=ffqqqq,fill opacity=1.0]  (0,0) --  plot[domain=0.67:1.31,variable=\t]({1*3.93*cos(\t r)+0*3.93*sin(\t r)},{0*3.93*cos(\t r)+1*3.93*sin(\t r)}) -- cycle ;
\draw [line width=0.4pt,color=ffffff,fill=ffffff,fill opacity=1.0] (9.47,-5.14) circle (2.96cm);
\draw (4.18,-0.2)-- (3.74,-2.4);
\draw (6.46,0.08)-- (7.16,-2.34);
\draw [shift={(6.42,-0.68)}] plot[domain=3.71:4.45,variable=\t]({1*3.19*cos(\t r)+0*3.19*sin(\t r)},{0*3.19*cos(\t r)+1*3.19*sin(\t r)});
\draw [shift={(4.48,-0.96)}] plot[domain=5.09:5.81,variable=\t]({1*3.02*cos(\t r)+0*3.02*sin(\t r)},{0*3.02*cos(\t r)+1*3.02*sin(\t r)});
\draw [shift={(5.64,-2.66)}] plot[domain=1.28:2.11,variable=\t]({1*2.86*cos(\t r)+0*2.86*sin(\t r)},{0*2.86*cos(\t r)+1*2.86*sin(\t r)});
\draw [shift={(9.47,-5.14)}] plot[domain=0.65:1.45,variable=\t]({1*2.96*cos(\t r)+0*2.96*sin(\t r)},{0*2.96*cos(\t r)+1*2.96*sin(\t r)});
\draw (11.82,-3.34)-- (12.46,-1.64);
\draw (10.38,-0.28)-- (9.82,-2.2);
\draw [shift={(9.38,-4.08)}] plot[domain=0.67:1.31,variable=\t]({1*3.93*cos(\t r)+0*3.93*sin(\t r)},{0*3.93*cos(\t r)+1*3.93*sin(\t r)});
\draw [->] (3.74,-2.4) -- (12.46,-1.64);
\draw [->] (12.46,-1.64) -- (3.74,-2.4);
\begin{scriptsize}
\draw[color=black] (8.46,-1.5) node {$\le1$};
\end{scriptsize}
\end{tikzpicture}
%\end{document}
	
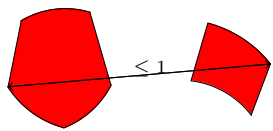
\captionof{figure}{Siamese tiles}
	\label{siamesetiles}
\end{minipage}
\end{center}

1) A tiling can contain two tiles with distance $1$, if the borders are coloured appropriately. (See Figure \ref{1distancetiles}.) A nice tiling cannot.

2) A tiling can contain two tiles with distance less than $1$ if all pairs of points of the two have distance less than $1$. We call such a pair a pair of Siamese tiles. (See Figure \ref{siamesetiles}.) \footnote{For the origin of the name, see for example \href{https://dustingmixon.wordpress.com/2018/06/24/polymath16-eighth-thread-more-upper-bounds/\# comment-5064}{https://dustingmixon.wordpress.com/2018/06/\\24/polymath16-eighth-thread-more-upper-bounds/\# comment-5064}.} A nice tiling cannot.

On the other hand, one restriction of Thomassen is not required in this definition: we do not require tiles to be simply connected, as we will deal with not simply connected tiles in the beginning (Section \ref{sec:converting}).

Note that the statement for the plane also follows from a relatively easy proof using Lemma \ref{lem:nice} and the fact that a triangulated planar graph has $3n-6$ edges.

\subsection{Colouring of spheres}

We can define the chromatic number of a sphere of radius $r$ similarly to the planar case: it is the minimum number of colours needed to colour the points of a sphere of radius $r$ such that no two points with Euclidean distance $1$ have the same colour.

Much less is known of the value of this number compared to the planar case.

It is long known that the chromatic number of a sphere of radius $r$ is at least $4$ if $r\ge\frac{1}{\sqrt{3}}$. For $r>\frac{\sqrt{3}}{2}$ Moser's spindle gives the lower bound, for smaller values, a generalized version of Moser's spindle is used. (Simmons (1976) \cite{s}). Recently, Cherkashin and Voronov \cite{cv} proved the statement for all $r>\frac{1}{2}$, proving the conjecture of Simmons. This lower bound on the radius is sharp: for $r=\frac{1}{2}$, two colours are enough, while for $r<\frac{1}{2}$, only one colour is needed.

It is also known that the chromatic number of any sphere is at most $15$, even with all of the above defined restrictions, as the $3$-dimensional space has a $15$-tiling (Radoičić, Tóth (2003)\cite{rt}), which can be used to generate such a colouring.

Recently, a $7$-colouring of large enough spheres have also been found by Tom Sirgedas, as part of the Polymath 16 project.\cite{si}

Also, the minimal number of colours needed for a nice tiling of a large enough sphere is at least $7$, which follows from Theorem \ref{thm:thomassen}.

The main result of this paper is improving this number to $8$.

Note that some sources mentioned earlier that the chromatic number of all spheres is at most $7$ \cite{bmp} \cite{hdcg}. It seems (from personal communication through Dömötör Pálvölgyi) that the authors expected that a colouring similar to that of Isbell in Figure \ref{fig:Isbell} also works for spheres. The present paper disproves this assumption, though it does not contradict to the chromatic number of spheres being at most $7$.

This result also improves the lower bound for the minimal number of colours needed for a nice tiling of the $3$-dimensional space, for which problem the best known bound was $6$ for the general colouring case (Nechushtan (2002) \cite{n}).

\section{The new result}

\begin{theorem}
\label{thm:main}
There is no nice tiling with at most $7$ colours of a large enough (radius $r\ge\frac{46.5}{\pi}=14.801...$) sphere $S$, even if we generalize the definition and allow tiles not to be simply connected.
\end{theorem}

In order to make the proof more legible, we give an outline of the main steps.

Suppose that there exists such a tiling $(\mathcal{T},f)$ of $S$. Now get rid of all the tiles that are in regions "completely surrounded" by one or two tiles. Then we define an adjacency graph $G$ on the remaining tiles such that $G$ is a triangulated planar graph.

It can be easily seen from the tiling being nice that all degrees of $G$ are at most $6$, so it has at most $12$ vertices with degree less than $6$ (exactly $12$ if counted with multiplicity given by the differences of $6$ and the degrees of these vertices). These vertices are called irregular vertices.

Also, for some subsets of $G$ that only have vertices with degree $6$, there exists a function to an infinite triangular grid such that the mapping is a local isomorphism in all vertices.

Now we have two cases.

The first case is actually made up of two subcases: either all irregular vertices are `close˙ to each other, or they can be separated into two groups of cardinality $6$ (counted with multiplicity), where the elements of the two groups have a large enough distance from each other, while inside one group, the distances are bounded. In this case, we can construct two cycles $c_2$ and $c_3$ of bounded length separating these two groups (or in case there is only one group, we divide them into two groups beforehand), which are close enough to the first and the second group, respectively. Then again we get a contradiction from the mapping of the part between the two cycles to the triangular grid: we find a cycle in this part that goes through two nearly antipodal points, but its graph length is not larger than $max(l(c_2),l(c_3))$.

The second case is when there is at least one way to divide the irregular vertices into two groups such that no two points from different groups are close to each other and the cardinality of the two groups (counted with multiplicity) is not divisible by $6$. In this case, we get a contradiction by examining the exact way to colour parts of the infinite triangular grid.

\begin{center}
\begin{minipage}{.3\textwidth}
	\centering
	\usetikzlibrary{arrows}
%\pagestyle{empty}
%\begin{document}
\begin{tikzpicture}[line cap=round,line join=round,scale=0.2]%,>=triangle 45,x=1.0cm,y=1.0cm]
\clip(-2.3,-10.42) rectangle (11.5,3.62);
\draw(4.56,-3.66) circle (6.31cm);
\begin{scriptsize}
\fill [color=black] (7.6,-4.12) circle (5pt);
\fill [color=black] (7.28,-4.5) circle (5pt);
\fill [color=black] (7.28,-4.04) circle (5pt);
\fill [color=black] (7,-4.24) circle (5pt);
\fill [color=black] (7.06,-3.72) circle (5pt);
\fill [color=black] (7.6,-3.52) circle (5pt);
\fill [color=black] (6.64,-3.84) circle (5pt);
\fill [color=black] (6.76,-3.2) circle (5pt);
\fill [color=black] (7.26,-3.04) circle (5pt);
\fill [color=black] (6.32,-3.52) circle (5pt);
\fill [color=black] (6.52,-4.4) circle (5pt);
\fill [color=black] (6.8,-4.8) circle (5pt);
\end{scriptsize}
\end{tikzpicture}
%\end{document}
	\label{Case1}
\end{minipage}
\begin{minipage}{.3\textwidth}
	\centering
	\usetikzlibrary{arrows}
%\pagestyle{empty}
%\begin{document}
\begin{tikzpicture}[line cap=round,line join=round,scale=0.2]%,>=triangle 45,x=1.0cm,y=1.0cm]
\clip(-2.3,-10.42) rectangle (11.5,3.62);
\draw(4.56,-3.66) circle (6.31cm);
\begin{scriptsize}
\fill [color=black] (7.6,-4.12) circle (5pt);
\fill [color=black] (7.28,-4.5) circle (5pt);
\fill [color=black] (7.28,-4.04) circle (5pt);
\fill [color=black] (6.9,-3.84) circle (5pt);
\fill [color=black] (2.54,-2.26) circle (5pt);
\fill [color=black] (2.1,-2.96) circle (5pt);
\fill [color=black] (1.98,-2.44) circle (5pt);
\fill [color=black] (2.9,-2.6) circle (5pt);
\fill [color=black] (2.44,-2.66) circle (5pt);
\fill [color=black] (2.62,-3.14) circle (5pt);
\fill [color=black] (6.8,-4.26) circle (5pt);
\fill [color=black] (6.8,-4.8) circle (5pt);
\end{scriptsize}
\end{tikzpicture}
%\end{document}
	\label{Case2}
\end{minipage}
\begin{minipage}{.3\textwidth}
	\centering
	\usetikzlibrary{arrows}
%\pagestyle{empty}
%\begin{document}
\begin{tikzpicture}[line cap=round,line join=round,scale=0.2]%,>=triangle 45,x=1.0cm,y=1.0cm]
\clip(-2.3,-10.42) rectangle (11.5,3.62);
\draw(4.56,-3.66) circle (6.31cm);
\begin{scriptsize}
\fill [color=black] (7.6,-4.12) circle (5pt);
\fill [color=black] (7.28,-4.5) circle (5pt);
\fill [color=black] (7.28,-4.04) circle (5pt);
\fill [color=black] (6.9,-3.84) circle (5pt);
\fill [color=black] (7.32,-3.58) circle (5pt);
\fill [color=black] (2.1,-2.96) circle (5pt);
\fill [color=black] (1.98,-2.44) circle (5pt);
\fill [color=black] (7.06,-1.46) circle (5pt);
\fill [color=black] (3.14,-0.12) circle (5pt);
\fill [color=black] (2.62,-3.14) circle (5pt);
\fill [color=black] (4.52,-7) circle (5pt);
\fill [color=black] (6.8,-4.8) circle (5pt);
\end{scriptsize}
\end{tikzpicture}
%\end{document}
	\label{Case3}
\end{minipage}

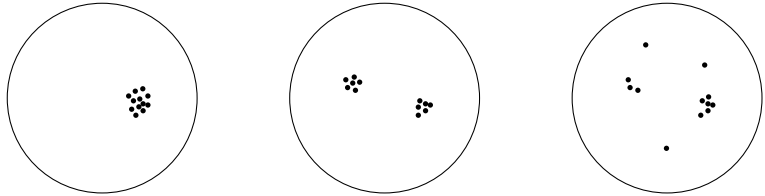
\captionof{figure}{Case 1a, Case 1b and Case 2}
\end{center}

Now we continue with a detailed proof.

First, note that although we use Euclidean distances in Theorem \ref{thm:main}, in the following we will use spherical distances. It only affects Theorem \ref{thm:main} in a way that we have to prove the statement from a smaller minimum radius ($r\ge17.9$) and by a conversion in the end, we get the original Theorem.

\subsection{Converting the problem to a graph problem}
\label{sec:converting}

In this section, we convert Theorem \ref{thm:main} into a theorem about fully triangulated graphs drawn on a sphere, which we will prove in the next section.

In the following, for any graph $X$, we denote the set of its vertices by $V(X)$ and the set of its edges by $E(X)$ and we denote the vertices of any graph $X$ by $\vtx_1(X),\vtx_2(X),...,\vtx_{\left\lvert V(X)\right\rvert}(X)$.

Also, from now on, cycles always will mean directed cycles, and if $X$ is a cycle, enumerate its vertices in the order they appear on $X$ (starting at an arbitrary vertex of $X$) with regard to the direction of $X$. In this case, use the indices $mod\lvert V(X)\rvert$.

Suppose that we have a $7$-tiling $(\mathcal{T},f)$ of $S$. Note that since all disks of diameter $1$ on $S$ can contain points of at most one tile per colour and $S$ can be covered with finitely many such disks, the family of tiles is finite, so we can write $\mathcal{T}=\left\lbrace T_1,...,T_n\right\rbrace$.

Call two tiles $T_i$ and $T_j$ adjacent if $\left(Cl\left(T_i\right)\cap T_j\right)\cup\left(T_i\cap Cl\left(T_j\right)\right)\neq\emptyset$.

\begin{lemma}
\label{lem:smallcomponent1}
If $r\ge\frac{2}{\pi}$ and $i\in\left\lbrace 1,...,n\right\rbrace$, then there is a unique connected component of $S\setminus T_i$, which contains all points of $S$ with the exception of at most an open disk of radius $1$.
\end{lemma}

\begin{proof}
If we take an open disk $D$ of radius $1$ around any point of $T_i$, it covers $T_i$, so $S\setminus D\subseteq S\setminus T_i$ and since $S\setminus D$ is connected, all of its points are in the same connected component of $S\setminus T_i$. And this component is unique as $S\setminus D$ is a closed disk of radius $r\pi-1$, so none of the other components satisfy the property described in the statement of the lemma.
\end{proof}

\begin{lemma}
\label{lem:smallcomponent2}
If $r\ge\frac{2}{\pi}$ and the adjacent tiles $T_i,T_j\in\mathcal{T}$ both have spherical diameter less than $1$, then there is a unique connected component of $S\setminus\left(T_i\cup T_j\right)$, which contains all points of $S$ with the exception of at most an open disk of radius $1$.
\end{lemma}

\begin{proof}
If we take an open disk $D$ of radius $1$ around a common border point of $T_i$ and $T_j$, it covers both $T_i$ and $T_j$, so $S\setminus D\subseteq S\setminus\left(T_i\cup T_j\right)$ and since $S\setminus D$ is connected, all of its points are in the same connected component of $S\setminus\left(T_i\cup T_j\right)$. And this component is unique as $S\setminus D$ is a closed disk of radius $r\pi-1>1$, so none of the other components satisfy the property described in the statement of the lemma.
\end{proof}

For any $T_i\in\mathcal{T}$ or $T_i,T_j\in\mathcal{T}$ (where $T_i$ and $T_j$ are adjacent), call the unique component described above the large component of $S\setminus T_i$ or $S\setminus\left(T_i\cup T_j\right)$, respectively, and all the other components the small components.

Now we will define a graph $G$ and simultaneously draw it on the surface of $S$.

First, define a subfamily $\mathcal{T'}$ of $\mathcal{T}$ that consists only of the tiles that are not included in any small component of $S\setminus T_i$ for some $i=1,...,n$, neither are they included in any small component of $S\setminus\left(T_i\cup T_j\right)$ for some $i,j=1,...,n$, $i\neq j$.

Now take an arbitrary point $\vtx_i(G)$ from each tile $T_i\in\mathcal{T}'$ and also, take a point $\pi_{i,j}=\pi_{j,i}$ on the common border of all adjacent tiles $T_i$ and $T_j$.

Next, take a topological tree $\tau_i$ fully inside $T_i$, whose leaves are the above defined $\pi_{i,j}$'s and which includes $\vtx_i(G)$. Such a tree exists, because $T_i$ is connected, thus, we can add the branches ending in the $\pi_{i,j}$'s one by one. (Note that in the optimal case, $\tau_i$ would be a topological star centered around $\vtx_i(G)$, but this is not always possible due to our very generous definition of tiles, which allows cutting points.) Now for all adjacent pairs $T_i,T_j\in\mathcal{T}'$, draw an edge between $\vtx_i(G)$ and $\vtx_j(G)$, which follows the path connecting $\vtx_i(G)$ and $\pi_{i,j}$ within $\tau_i$ and then the path connecting $\pi_{j,i}$ and $\vtx_j(G)$ within $\tau_j$. Call the resulting graph $G$ and denote this drawing of $G$ by $G_0$. The interiors of the edges of $G_0$ might overlap, but they do not cross each other at all, so with a slight movement, we can make them non-overlapping, thus creating the final drawing of $G$.\footnote{To make some disclaimer about the term "slight movement": apart from making all of the interiors of the edges of $G$ disjoint, the only requirement is to keep the statement of Theorem \ref{lem:distance} true. Technically, even a non-movement that only creates a virtual ordering among the edges would be satisfactory.}

\begin{center}
	\centering
	\input{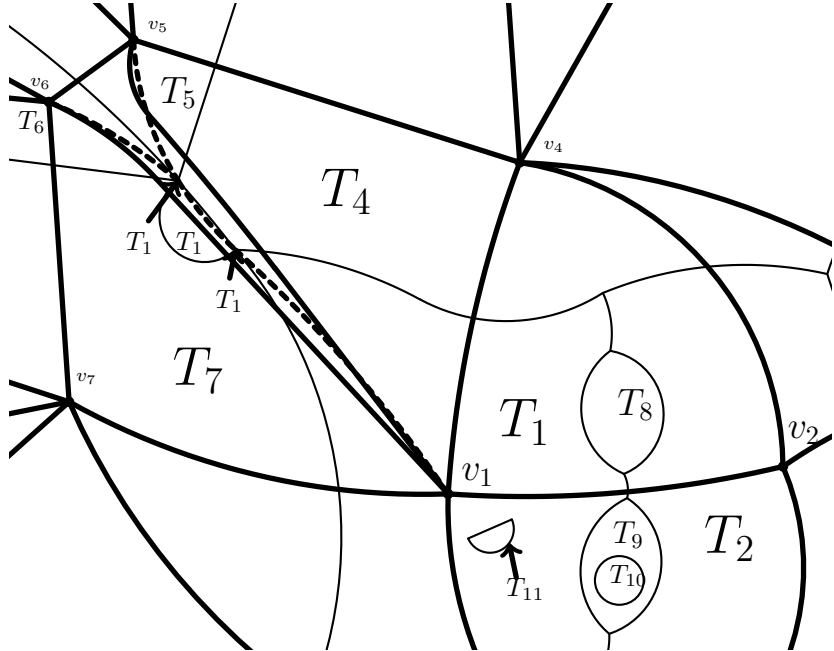}
	\captionof{figure}{A portion of $\mathcal{T}$ (the borders are denoted by the thin lines, and in case it matters, it is also denoted where the border points belong) and the corresponding section of $G$ (the edges are denoted by thick lines). The dashed line represents the edges $\left(\vtx_1(G),\vtx_5(G)\right)$ and $\left(\vtx_1(G),\vtx_6(G)\right)$ of $G_0$. $T_8$, $T_9$, $T_{10}$ and $T_{11}$ are in the small components for at least one tile or at least one pair of tiles, so no vertices belong to them.}
	\label{smallparts}
\end{center}

\begin{definition}
We call a colouring of a graph \emph{nice}, if there are no two different vertices within graph distance at most $2$, which are coloured with the same colour. Alternatively, a nice colouring of $G$ can be defined as a colouring, which is also a proper colouring for $G^2$ (the square of $G$).
\end{definition}

\begin{lemma}
\label{lem:nice}
The tiling $\left(\mathcal{T},f\right)$ of $S$ being nice means that applying the same colouring to the corresponding vertices of $G$ is a nice colouring.
\end{lemma}

\begin{proof}
All adjacent tiles in $\mathcal{T}$ have a spherical distance of $0$ per definition and all tiles in $\mathcal{T}$ that are adjacent to the same tile have a spherical distance less than $1$ per definition. Thus, no tiles with graph distance at most $2$ in $G$ can have the same colour in a nice tiling.
\end{proof}

\begin{lemma}
\label{lem:unitdisk}
Any (open or closed) disk $D$ on $S$ with radius at least $1$ contains at least one vertex from $G$.
\end{lemma}

\begin{proof}
If the center $O$ of $D$ belongs to a tile from $\mathcal{T'}$, then the vertex representing it has distance less than $1$ from it, so it is inside $D$.

If $O$ belongs to a tile $T_i$ from $\mathcal{T}\setminus\mathcal{T'}$, then either it is in the small component of $S\setminus T_j$ for some $T_j\in\mathcal{T}$ or it is in the small component of $S\setminus\left(T_j\cup T_k\right)$ for some $T_j,T_k\in\mathcal{T}$ (note that $T_i$ is not necessarily bordering any of these tiles). And there exists a line through $O$, which contains at least one point from $Cl\left(T_j\right)$ or from $Cl\left(T_k\right)$, we can assume without loss of generality that they are from $Cl\left(T_j\right)$. And since these two points are within an open unit disk around $\vtx_j(G)$, all of the shorter line segment connecting them is also in this unit disk, and thus, $O$ too. So $\vtx_j(G)$ is also within an open unit disk around $O$.
\end{proof}

\begin{lemma}
\label{lem:distance}
If $\vtx_i(G)$ and $\vtx_j(G)$ are adjacent vertices in $G$, they have a distance less than $2$ and (in case the drawing of $G$ was appropriately constructed from $G_0$) all points of the edge $\left(\vtx_i(G),\vtx_j(G)\right)$ have a distance less than $1$ from at least one $\vtx_i(G)$ and $\vtx_j(G)$.
\end{lemma}

\begin{proof}
For $G_0$, this follows per definition: the first half is true since $\vtx_i(G)\in T_i$ and $\vtx_j(G)\in T_j$, the diameter of both tiles is less than $1$ and they have a common border point, while the second half follows from the fact that all points of the edge $\left(\vtx_i(G),\vtx_j(G)\right)$ are either in $T_i$ or in $T_j$. If the changes applied when constructing the final drawing of $G$ were small enough, the second half is also true, while the first one is true per definition as the placement of vertices did not change.
\end{proof}

\begin{lemma}
\label{lem:fullytriangulated}
$G$ is a fully triangulated planar graph.
\end{lemma}

\begin{proof}
$G$ is a simple graph as it has finitely many edges, finitely many vertices, no loops and no parallel edges. It is planar per definition as it has finitely many vertices and edges, and it is drawn onto $S$ without any crossings, so by projecting $S$ minus an arbitrary point of $S\setminus G$ onto the plane, it also can be drawn onto the plane.

Also, $G$ is connected: we will take some $i,j\in\lbrace1,2,...,n\rbrace$ for which $T_i,T_j\in\mathcal{T'}$ and show that they are in the same connected component. First, take an arbitrary Jordan curve $\gamma$ connecting a point of $T_i$ and a point of $T_j$, then we can make a list of the tiles intersected by $\gamma$ in order of the first intersection (if we would list them with multiplicity, then this could be an infinite list). We will now show that any tile $T_k\in\mathcal{T}'$ on the list is adjacent to at least one tile from $\mathcal{T}'$ preceding it.

Let $p$ be the point in which $\gamma$ enters $T_k$ (the infimum of the points in $T_k\cap\gamma$, counted along $\gamma$) and let $T_{k'}$ be the tile (or a tile) from which $\gamma$ enters $T_k$ first (to be more precise, in case of $p\notin T_k$, define this as the tile $p$ belongs to, otherwise, calling the section of $\gamma$ ending at $p$, let $T_{k'}$ be any tile for which $p$ is a point of convergence within $\gamma_p$). In case $T_{k'}\in\mathcal{T}'$, $T_k$ and $T_{k'}$ are adjacent per definition, while if $T_{k'}\in\mathcal{T}\setminus\mathcal{T}'$, then $T_{k'}$ is either in the small component of $S\setminus T_k$ or in the small component of $S\setminus\left(T_k\cup T_{k''}\right)$ for some $T_{k''}\in\mathcal{T}'$. And since $T_i\in\mathcal{T}'$, it is not in the same small component, so the first case would lead to a contradiction, while the second would mean that $\gamma$ intersected $T_{k''}$ before $T_k$, and since $T_{k''}$ must be adjacent to $T_k$, this finishing the statement.

Now suppose that $G$ contains a face $G$ that is bordered by more than $3$ edges and denote the corresponding face of $G_0$ by $F_0$. We may assume without loss of generality that the vertices of $F$ are $v_1$, $v_2$, ... $v_k$ in this order and also that the $\vtx_i(G)$ ($i=1,2,...,k$) are on the border of $F_0$, otherwise we could simply move $\vtx_i(G)$ to the last point in $\tau_i$ in which the edges $\left(\vtx_i(G),\vtx_{i-1}(G)\right)$ and $\left(\vtx_i(G),\vtx_{i+1}(G)\right)$ coincide (where the indices are counted $\mod k$). Now suppose that $T_1$ and $T_3$ are not adjacent. In this case, $F_0\setminus\left(T_1\cup T_3\right)$ is connected, thus there exists a Jordan curve $\gamma$ from $\vtx_2(G)$ to $\vtx_4(G)$ within $F_0\setminus\left(T_1\cup T_3\right)$. As above, the first tile $T_j$ from $\mathcal{T'}\setminus\left\lbrace T_2\right\rbrace$ is adjacent to $T_2$ and a starting segment of $\gamma$ connects $\vtx_2(G)$ and $\vtx_j(G)$, whose interior is fully inside $F_0$. If the edge $\left(\vtx_2(G),\vtx_j(G)\right)$ would be outside $F_0$, they would separate $T_1$ from $T_3$ because of Jordan's theorem and because of the fact that $T_j$ is not $T_1$, nor $T_3$. And this is a contradiction because both $T_1$ and $T_3$ are from $\mathcal{T'}$. So the edge should be at least partly within $F_0$, which contradicts the assumption that $F$ is a face of $G$.
\end{proof}

Now, combining Lemmas \ref{lem:nice}, \ref{lem:unitdisk}, \ref{lem:distance} and \ref{lem:fullytriangulated}, we get that Theorem \ref{thm:main} is a consequence of Theorem \ref{thm:main2} applied for $d_1=2$ and $d_2=1$.

\begin{remark}
Note that at this point it is easy to prove the planar case of Thomassen's result with our more generous definition of tiling, since we can construct a similar graph $G$ on the plane, for which the average degree would tend to $6$ due to Lemma \ref{lem:fullytriangulated} by choosing a large enough section of the plane, but in a colouring with 6 colours, no vertices can have degree greater than $5$ due to Theorem \ref{lem:nice}, a contradiction.
\end{remark}

\subsection{The reworded problem}

\begin{theorem}
\label{thm:main2}
Suppose we have a sphere $S$ with radius $r$ and a fully triangulated connected planar graph $G$ drawn on the surface of $S$ without crossings, which has $n$ vertices, all of its vertices have spherical distance less than $d_1$ on $S$, all open unit disks on $S$ contain at least one vertex and all of the points of all of its edges have distance less than $d_2$ from at least one of its respective endpoints. Then for $r\ge\frac{23d_1+0.5d_2}{\pi}$, $G$ does not have a nice $7$-colouring.
\end{theorem}

\begin{proof}
First, suppose that $G$ does have a nice $7$-colouring $\sigma$. This obviously means that all of its degrees are at most $6$ as any vertex with degree at least $7$ would have at least two neighbours of the same colour by the pigeonhole principle.

\newpage

\noindent{\bf Notations and preliminary remarks}

\vskip6pt

We will continue to use the notations $V(X)$, $E(X)$ and $\vtx_i(X)$ defined in the previous section for all graphs $X$. Also, for any edge $\left(\vtx_i(G),\vtx_j(G)\right)$ of $G$, call $e\setminus\left(\vtx_i(G)\cup \vtx_j(G)\right)$ an open edge of $G$, denote the set of open edges of $G$ by $E_o(G)$. In the above definitions, we used the fixed embedding of $G$. In what follows, we will work with both the abstract graph and its embedding simultaneously.

Call the vertices of $G$ with degree $6$ {\it regular vertices} and the vertices of $G$ with degree less than $6$ {\it irregular vertices} and let their {\it multiplicity} be the difference of $6$ and their degree. Note that since $G$ is a fully triangulated planar graph, $\left\lvert E(G)\right\rvert=3\cdot\left\lvert V(G)\right\rvert-6$, thus the number of irregular vertices counted by multiplicity is exactly $12$. Also, let the multiset of irregular vertices be $I=\left\lbrace i_1,...,i_{12}\right\rbrace$.

For any two subsets $a$ and $b$ of $S$, let their {\it spherical distance} be their spherical distance on $S$ (denoted by $dist_S(a,b)$).

For any two subgraphs $G_a$ and $G_b$ of $G$, let their {\it graph distance} mean the smallest graph distance in $G$ occuring between their vertices (denoted by $dist_G(a,b)$).

For any $S'\subseteq S$, let $-\left(S'\right)$ be the antipode set of $S'$ within $S$.

For any cycle $c$ in $G$, let its sides be the connected components of $S\setminus c$. It is trivial that all cycles have two sides.

For any directed cycle $c$, define the side to its left as its {\it inside or} its {\it interior} (denoted by $int(c)$) and the side to its right as its {\it outside} or its {\it exterior} (denoted by $ext(c)$).
Also, denote the subgraphs of $G$ which only includes the edges with interiors within these regions (along with their endpoints) by $G_{int(c)}$, $G_{ext(c)}$, respectively.

For any $3$-cycle $c$ in $G$ for which $int(c)$ does not contain any vertices or open edges, call $int(c)\cup c$ a {\it triangle}, call the set of triangles of $G$ $\Delta(G)$.

For any $3$-cycle $c$ in $G$ for which $int(c)$ does not contain any vertices or open edges, call $int(c)$ an {\it open triangle}, call the set of open triangles of $G$ $\Delta_o(G)$.

Since $G$ is a fully triangulated graph, $S$ is the disjoint union of the elements of $V(G)$, $E_o(G)$ and $\Delta_o(G)$.

Let $A\left(\Delta_o(G)\right)$ be the adjacency graph of the open triangles of $G$ defined as the graph which has the elements of $\Delta_o(G)$ as its vertices and those are connected with an edge, which have a common bordering edge.

For a directed cycle $c$ in $G$, let the {\emph local curvature} of $c$ in $\vtx_i(c)$ (denoted by $lc\left(\vtx_i(c),c\right)$) be defined as $deg_{G_{ext(c)}}\left(\vtx_i(c)\right)-2$ and the {\emph curvature} of $c$ within $G$ ($\curv_G(c)$) be $\sum\limits_{i=1}^{\lvert V(c)\rvert}{lc\left(\vtx_i(c),c\right)}$.

For a subset $S'\in S$ and a cycle on $S$, let their \emph{signed spherical distance} be $$dist'_S(S',c)=dist_S\left(S',int(c)\right)-dist_S\left(S',ext(c)\right).$$ Obviously at least one of the two halves of the formula is $0$, depending on which side of $c$ is $S'$ located, and thus, $\left\lvert dist'_S(S',c)\right\rvert=dist_S(S',c)$.

\begin{lemma}\label{lem:2pointscycle}
For any two points $p_1,p_2\in S$ and any cycle $c\subset S$, \\$\left\lvert dist'_S\left(p_1,c\right)-dist'_S\left(p_2,c\right)\right\rvert\le dist_S\left(p_1,p_2\right)$.
\end{lemma}

\begin{proof}
If both $dist'_S\left(p_1,c\right)-dist'_S\left(p_2,c\right)$ are non-negative or both are non-positive, then the statement is obvious as $dist'_S\left(p_i,c\right)=dist_S\left(p_i,c\right)$ for $i=1,2$. If one is positive, while the other one is negative, then the shortest spherical segment between $p_1$ and $p_2$ intersects $c$ in at least one point $p_3$, for which the following formula is true: $$\left\lvert dist'_S\left(p_1,c\right)-dist'_S\left(p_2,c\right)\right\rvert=dist_S\left(p_1,c\right)+dist_S\left(p_2,c\right)$$
$$\le\left(p_1,p_3\right)+dist_S\left(p_2,p_3\right)=dist_S\left(p_1,p_2\right).$$
\end{proof}

\begin{lemma}\label{lem:point2cycles}
For any point $p_0\in S$ and cycles $c_1,c_2\subset S$ with $int\left(c_1\right)\subseteq int\left(c_2\right)$, $\left\lvert dist'_S\left(p_0,c_1\right)-dist'_S\left(p_0,c_2\right)\right\rvert\le\max\limits_{p_1\in c_1}{\min\limits_{p_2\in c_2}{dist_S\left(p_1,p_2\right)}}$.
\end{lemma}

\begin{proof}
In case $dist'_S\left(p_0,c_1\right)$ and $dist'_S\left(p_0,c_2\right)$ have the same sign and we assume without loss of generality that $c_1$ is closer to $p_0$, then by choosing $p_1$ as the closest point of $c_1$ to $p_0$, the statement follows from the triangle inequality as above. If they have an opposite sign, then it also follows from the triangle inequality, again from choosing $p_1$ as the closest point of $c_1$ to $p_0$.
\end{proof}

\begin{lemma}
\label{lem:ADeltaoGisconnected}
$A\left(\Delta_o(G)\right)$ is connected.
\end{lemma}

\begin{proof}
For any $e\in E_o(G)$, the two open triangles bordering $e$ are connected with an edge in $A\left(\Delta_o(G)\right)$, thus, they are in the same connected component of $A\left(\Delta_o(G)\right)$. This also means that for any $v\in V(G)$, the open triangles bordering $v$ are also in the same connected component of $A\left(\Delta_o(G)\right)$ as we can get from any open triangle bordering $v$ to any one by crossing the open edges of $G$ starting from $v$ one by one, and as we have seen it above, these are in the same component.

Also, for any two neighbouring vertices of $G$, there exists an open triangle bordering both of them. And since $G$ is connected, the subgraphs defined in the above paragraph are in the same connected component of $A\left(\Delta_o(G)\right)$. And since all open triangles have at least one bordering vertex in $G$ (in fact, exactly $3$), this proves the statement.
\end{proof}

Let the \emph{graph length} of a path or closed path $p$ in $G'$ be the number of its edges. We denote it by $l(p)$.

Let the \emph{spherical broken line} of a path or closed path $p$ in $G'$ be the curve defined by connecting neighbouring vertices of $p$ with spherical segments instead of the edges connecting them.

Let the \emph{broken line length} of a path or a closed path $p$ in $G'$ be the length of the spherical broken line belonging to $p$. We denote it by $L(p)$.

Let the \emph{$i$-neighbourhood} of a vertex $v$ of $G'$ be the subgraph of $G'$ induced by those vertices, which have graph distance at most $i$ from $v$ (denoted by $N_i(v)$).

Let the \emph{strict $i$-neighbourhood} of a vertex $v$ of $G'$ be the subgraph of $G'$ induced by those vertices, which have graph distance exactly $i$ from $v$ (denoted by $n_i(v)$).

For a vertex $v$ and a set $S$ of vertices, let $e(v,S)$ mean the number of edges starting from $v$ and ending in any of the vertices of $S$.

\begin{center}
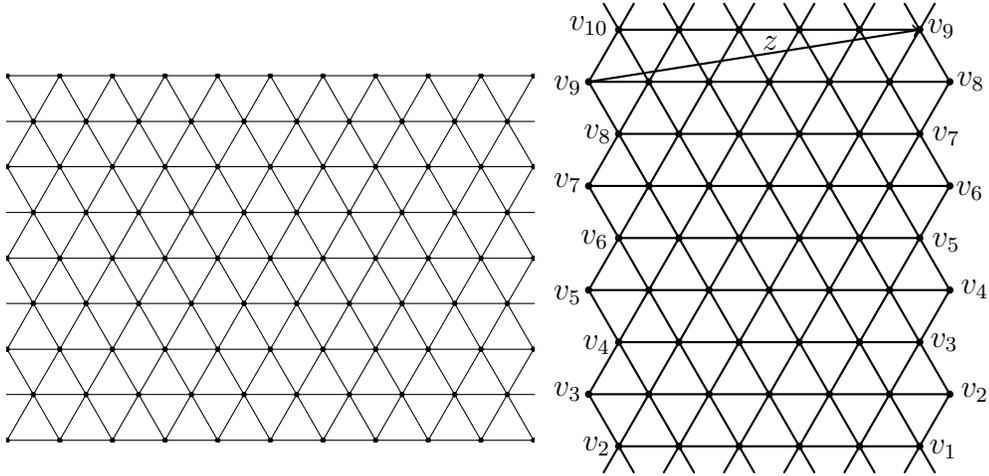

\begin{minipage}{.59\textwidth}
	\centering
	\input{Figures/haromszogracs}
\end{minipage}
\begin{minipage}{.39\textwidth}
	\centering
	\input{Figures/GTz}
\end{minipage}
	\captionof{figure}{$G_T$ on the left and $G_T(z)$ for a specific $z$ on the right}
	\label{GT}
\end{center}

Define an infinite regular triangular grid $G(T)$ drawn in the Euclidean plane as shown in Figure \ref{GT} left. Define open edges, triangles and open triangles in $G(T)$ analogously as in $G$. Also define $G_T(z)$ for any vector $z$ as the graph we obtain from $G_T$ by merging all of its vertices, whose vector is of the form $kz$ for some integer $k$ (see Figure \ref{GT} right).

Furthermore, define the following subgraphs of $G(T)$: $G(h)$ (Figure \ref{fig:triangulargrid4}), $G(h)^+$ (Figure \ref{fig:triangulargrid2}), $G(H)^-$ (Figure \ref{fig:triangulargrid3}) and $G(H)$ (Figure \ref{triangulargrid1b}).

Analogously to the colouring of the plane by Isbell, call a colouring of the vertices of the infinite triangular grid $T$ an {\it Isbell colouring} if it is constructed in the following way:

We take a vertex in the grid and colour it and its neighbours with $7$ different colours. We then tile the grid with the disjoint translates of this coloured hexagon.

Such a colouring is trivially nice and periodical, thus any Isbell colouring of $T$ can be generated using any of the vertices of $T$ as the starting vertex. Also, for all colourings of the starting hexagon, there are exactly two ways to colour $T$ depending on how we place the hexagons compared to each other. Also, all Isbell colourings can be generated with the above procedure starting from any hexagon formed by a vertex and its $6$ neighbours.

\begin{lemma}
\label{lem:Isbell1}
The $G_H$ (Figure \ref{triangulargrid1b}) can only be nicely $7$-coloured by a part of an Isbell colouring.
\end{lemma}

\begin{center}
\begin{minipage}{.32\textwidth}
	\centering
	\usetikzlibrary{arrows}
%\pagestyle{empty}
%\begin{document}
\begin{tikzpicture}[line cap=round,line join=round,scale=0.9]%,>=triangle 45,x=1.0cm,y=1.0cm]
\clip(-2.25,-2) rectangle (2.35,2.12);
\draw (-1,-1.73)-- (0,-1.73);
\draw (0,-1.73)-- (1,-1.73);
\draw (1,-1.73)-- (1.5,-0.87);
\draw (1.5,-0.87)-- (2,0);
\draw (2,0)-- (1.5,0.87);
\draw (1.5,0.87)-- (1,1.73);
\draw (1,1.73)-- (0,1.73);
\draw (0,1.73)-- (-1,1.73);
\draw (-1,1.73)-- (-1.5,0.87);
\draw (-1.5,0.87)-- (-2,0);
\draw (-2,0)-- (-1.5,-0.87);
\draw (-1.5,-0.87)-- (-1,-1.73);
\draw (1,-1.73)-- (0.5,-0.87);
\draw (0.5,-0.87)-- (1.5,-0.87);
\draw (1.5,-0.87)-- (1,0);
\draw (1,0)-- (2,0);
\draw (1,0)-- (1.5,0.87);
\draw (1.5,0.87)-- (0.5,0.87);
\draw (0.5,0.87)-- (1,1.73);
\draw (0.5,0.87)-- (0,1.73);
\draw (0,1.73)-- (-0.5,0.87);
\draw (-0.5,0.87)-- (-1,1.73);
\draw (-1.5,0.87)-- (-0.5,0.87);
\draw (-1,0)-- (-0.5,0.87);
\draw (-2,0)-- (-1,0);
\draw (-1,0)-- (-1.5,0.87);
\draw (-1.5,-0.87)-- (-1,0);
\draw (-1,0)-- (-0.5,-0.87);
\draw (-0.5,-0.87)-- (-1.5,-0.87);
\draw (-1,-1.73)-- (-0.5,-0.87);
\draw (-0.5,-0.87)-- (0,-1.73);
\draw (0,-1.73)-- (0.5,-0.87);
\draw (0.5,-0.87)-- (-0.5,-0.87);
\draw (0.5,-0.87)-- (1,0);
\draw (0.5,-0.87)-- (0,0);
\draw (1,0)-- (0,0);
\draw (-0.5,-0.87)-- (0,0);
\draw (-1,0)-- (0,0);
\draw (0,0)-- (-0.5,0.87);
\draw (-0.5,0.87)-- (0.5,0.87);
\draw (0,0)-- (0.5,0.87);
\draw (1,0)-- (0.5,0.87);
\begin{scriptsize}
\fill [color=black] (0,0) circle (2pt);
\fill [color=black] (1,0) circle (2pt);
\fill [color=black] (0.5,0.87) circle (2pt);
\fill [color=black] (-0.5,0.87) circle (2pt);
\fill [color=black] (-1,0) circle (2pt);
\fill [color=black] (-0.5,-0.87) circle (2pt);
\fill [color=black] (0.5,-0.87) circle (2pt);
\fill [color=black] (1.5,0.87) circle (2pt);
\fill [color=black] (1,1.73) circle (2pt);
\fill [color=black] (0,1.73) circle (2pt);
\fill [color=black] (-1,1.73) circle (2pt);
\fill [color=black] (-1.5,0.87) circle (2pt);
\fill [color=black] (-2,0) circle (2pt);
\fill [color=black] (-1.5,-0.87) circle (2pt);
\fill [color=black] (-1,-1.73) circle (2pt);
\fill [color=black] (0,-1.73) circle (2pt);
\fill [color=black] (1,-1.73) circle (2pt);
\fill [color=black] (1.5,-0.87) circle (2pt);
\fill [color=black] (2,0) circle (2pt);
\end{scriptsize}
\end{tikzpicture}
%\end{document}
	\captionof{figure}{}
	\label{triangulargrid1b}
\end{minipage}
\begin{minipage}{.32\textwidth}
	\centering
	\usetikzlibrary{arrows}
%\pagestyle{empty}
%\begin{document}
\begin{tikzpicture}[line cap=round,line join=round,scale=0.9]%,>=triangle 45,x=1.0cm,y=1.0cm]
\clip(-2.25,-2) rectangle (2.35,2.12);
\draw (-1,-1.73)-- (0,-1.73);
\draw (0,-1.73)-- (1,-1.73);
\draw (1,-1.73)-- (1.5,-0.87);
\draw (1.5,-0.87)-- (2,0);
\draw (2,0)-- (1.5,0.87);
\draw (1.5,0.87)-- (1,1.73);
\draw (1,1.73)-- (0,1.73);
\draw (0,1.73)-- (-1,1.73);
\draw (-1,1.73)-- (-1.5,0.87);
\draw (-1.5,0.87)-- (-2,0);
\draw (-2,0)-- (-1.5,-0.87);
\draw (-1.5,-0.87)-- (-1,-1.73);
\draw (1,-1.73)-- (0.5,-0.87);
\draw (0.5,-0.87)-- (1.5,-0.87);
\draw (1.5,-0.87)-- (1,0);
\draw (1,0)-- (2,0);
\draw (1,0)-- (1.5,0.87);
\draw (1.5,0.87)-- (0.5,0.87);
\draw (0.5,0.87)-- (1,1.73);
\draw (0.5,0.87)-- (0,1.73);
\draw (0,1.73)-- (-0.5,0.87);
\draw (-0.5,0.87)-- (-1,1.73);
\draw (-1.5,0.87)-- (-0.5,0.87);
\draw (-1,0)-- (-0.5,0.87);
\draw (-2,0)-- (-1,0);
\draw (-1,0)-- (-1.5,0.87);
\draw (-1.5,-0.87)-- (-1,0);
\draw (-1,0)-- (-0.5,-0.87);
\draw (-0.5,-0.87)-- (-1.5,-0.87);
\draw (-1,-1.73)-- (-0.5,-0.87);
\draw (-0.5,-0.87)-- (0,-1.73);
\draw (0,-1.73)-- (0.5,-0.87);
\draw (0.5,-0.87)-- (-0.5,-0.87);
\draw (0.5,-0.87)-- (1,0);
\draw (0.5,-0.87)-- (0,0);
\draw (1,0)-- (0,0);
\draw (-0.5,-0.87)-- (0,0);
\draw (-1,0)-- (0,0);
\draw (0,0)-- (-0.5,0.87);
\draw (-0.5,0.87)-- (0.5,0.87);
\draw (0,0)-- (0.5,0.87);
\draw (1,0)-- (0.5,0.87);
\begin{scriptsize}
\fill [color=black] (0,0) circle (2pt);
\draw[color=black] (0.3,0.2) node {\normalsize{1}};
\fill [color=black] (1,0) circle (2pt);
\draw[color=black] (1.3,0.2) node {\normalsize{4}};
\fill [color=black] (0.5,0.87) circle (2pt);
\draw[color=black] (0.8,1.06) node {\normalsize{5}};
\fill [color=black] (-0.5,0.87) circle (2pt);
\draw[color=black] (-0.2,1.06) node {\normalsize{6}};
\fill [color=black] (-1,0) circle (2pt);
\draw[color=black] (-0.7,0.2) node {\normalsize{7}};
\fill [color=black] (-0.5,-0.87) circle (2pt);
\draw[color=black] (-0.2,-0.66) node {\normalsize{2}};
\fill [color=black] (0.5,-0.87) circle (2pt);
\draw[color=black] (0.8,-0.66) node {\normalsize{3}};
\fill [color=black] (1.5,0.87) circle (2pt);
\fill [color=black] (1,1.73) circle (2pt);
\fill [color=black] (0,1.73) circle (2pt);
\fill [color=black] (-1,1.73) circle (2pt);
\fill [color=black] (-1.5,0.87) circle (2pt);
\fill [color=black] (-2,0) circle (2pt);
\fill [color=black] (-1.5,-0.87) circle (2pt);
\fill [color=black] (-1,-1.73) circle (2pt);
\fill [color=black] (0,-1.73) circle (2pt);
\fill [color=black] (1,-1.73) circle (2pt);
\fill [color=black] (1.5,-0.87) circle (2pt);
\fill [color=black] (2,0) circle (2pt);
\end{scriptsize}
\end{tikzpicture}
%\end{document}
	\captionof{figure}{}
	\label{triangulargrid1feliratos}
\end{minipage}
\begin{minipage}{.32\textwidth}
	\centering
	\usetikzlibrary{arrows}
%\pagestyle{empty}
%\begin{document}
\begin{tikzpicture}[line cap=round,line join=round,scale=0.9]%,>=triangle 45,x=1.0cm,y=1.0cm]
\clip(-2.25,-2) rectangle (2.35,2.12);
\draw (-1,-1.73)-- (0,-1.73);
\draw (0,-1.73)-- (1,-1.73);
\draw (1,-1.73)-- (1.5,-0.87);
\draw (1.5,-0.87)-- (2,0);
\draw (2,0)-- (1.5,0.87);
\draw (1.5,0.87)-- (1,1.73);
\draw (1,1.73)-- (0,1.73);
\draw (0,1.73)-- (-1,1.73);
\draw (-1,1.73)-- (-1.5,0.87);
\draw (-1.5,0.87)-- (-2,0);
\draw (-2,0)-- (-1.5,-0.87);
\draw (-1.5,-0.87)-- (-1,-1.73);
\draw (1,-1.73)-- (0.5,-0.87);
\draw (0.5,-0.87)-- (1.5,-0.87);
\draw (1.5,-0.87)-- (1,0);
\draw (1,0)-- (2,0);
\draw (1,0)-- (1.5,0.87);
\draw (1.5,0.87)-- (0.5,0.87);
\draw (0.5,0.87)-- (1,1.73);
\draw (0.5,0.87)-- (0,1.73);
\draw (0,1.73)-- (-0.5,0.87);
\draw (-0.5,0.87)-- (-1,1.73);
\draw (-1.5,0.87)-- (-0.5,0.87);
\draw (-1,0)-- (-0.5,0.87);
\draw (-2,0)-- (-1,0);
\draw (-1,0)-- (-1.5,0.87);
\draw (-1.5,-0.87)-- (-1,0);
\draw (-1,0)-- (-0.5,-0.87);
\draw (-0.5,-0.87)-- (-1.5,-0.87);
\draw (-1,-1.73)-- (-0.5,-0.87);
\draw (-0.5,-0.87)-- (0,-1.73);
\draw (0,-1.73)-- (0.5,-0.87);
\draw (0.5,-0.87)-- (-0.5,-0.87);
\draw (0.5,-0.87)-- (1,0);
\draw (0.5,-0.87)-- (0,0);
\draw (1,0)-- (0,0);
\draw (-0.5,-0.87)-- (0,0);
\draw (-1,0)-- (0,0);
\draw (0,0)-- (-0.5,0.87);
\draw (-0.5,0.87)-- (0.5,0.87);
\draw (0,0)-- (0.5,0.87);
\draw (1,0)-- (0.5,0.87);
\begin{scriptsize}
\fill [color=black] (0,0) circle (2pt);
\draw[color=black] (0.3,0.2) node {\normalsize{1}};
\fill [color=black] (1,0) circle (2pt);
\draw[color=black] (1.3,0.2) node {\normalsize{4}};
\fill [color=black] (0.5,0.87) circle (2pt);
\draw[color=black] (0.8,1.06) node {\normalsize{5}};
\fill [color=black] (-0.5,0.87) circle (2pt);
\draw[color=black] (-0.2,1.06) node {\normalsize{6}};
\fill [color=black] (-1,0) circle (2pt);
\draw[color=black] (-0.7,0.2) node {\normalsize{7}};
\fill [color=black] (-0.5,-0.87) circle (2pt);
\draw[color=black] (-0.2,-0.66) node {\normalsize{2}};
\fill [color=black] (0.5,-0.87) circle (2pt);
\draw[color=black] (0.8,-0.66) node {\normalsize{3}};
\fill [color=black] (1.5,0.87) circle (2pt);
%\draw[color=black] (1.7,1.06) node {\normalsize{7}};
\fill [color=black] (1,1.73) circle (2pt);
%\draw[color=black] (1.2,1.94) node {\normalsize{3}};
\fill [color=black] (0,1.73) circle (2pt);
\draw[color=black] (0.3,1.94) node {\normalsize{2}};
\fill [color=black] (-1,1.73) circle (2pt);
%\draw[color=black] (-0.7,1.94) node {\normalsize{4}};
\fill [color=black] (-1.5,0.87) circle (2pt);
%\draw[color=black] (-1.2,1.06) node {\normalsize{3}};
\fill [color=black] (-2,0) circle (2pt);
%\draw[color=black] (-1.7,0.2) node {\normalsize{5}};
\fill [color=black] (-1.5,-0.87) circle (2pt);
%\draw[color=black] (-1.2,-0.66) node {\normalsize{4}};
\fill [color=black] (-1,-1.73) circle (2pt);
%\draw[color=black] (-0.7,-1.54) node {\normalsize{6}};
\fill [color=black] (0,-1.73) circle (2pt);
%\draw[color=black] (0.3,-1.54) node {\normalsize{5}};
\fill [color=black] (1,-1.73) circle (2pt);
%\draw[color=black] (1.3,-1.54) node {\normalsize{7}};
\fill [color=black] (1.5,-0.87) circle (2pt);
%\draw[color=black] (1.8,-0.66) node {\normalsize{6}};
\fill [color=black] (2,0) circle (2pt);
\draw[color=black] (2.2,0.2) node {\normalsize{2}};
\end{scriptsize}
\end{tikzpicture}
%\end{document}
	\captionof{figure}{}
	\label{triangulargrid1feliratos1b}
\end{minipage}%
\end{center}
\begin{center}
\begin{minipage}{.32\textwidth}
	\centering
	\usetikzlibrary{arrows}
%\pagestyle{empty}
%\begin{document}
\begin{tikzpicture}[line cap=round,line join=round,scale=0.9]%,>=triangle 45,x=1.0cm,y=1.0cm]
\clip(-2.25,-2) rectangle (2.35,2.12);
\draw (-1,-1.73)-- (0,-1.73);
\draw (0,-1.73)-- (1,-1.73);
\draw (1,-1.73)-- (1.5,-0.87);
\draw (1.5,-0.87)-- (2,0);
\draw (2,0)-- (1.5,0.87);
\draw (1.5,0.87)-- (1,1.73);
\draw (1,1.73)-- (0,1.73);
\draw (0,1.73)-- (-1,1.73);
\draw (-1,1.73)-- (-1.5,0.87);
\draw (-1.5,0.87)-- (-2,0);
\draw (-2,0)-- (-1.5,-0.87);
\draw (-1.5,-0.87)-- (-1,-1.73);
\draw (1,-1.73)-- (0.5,-0.87);
\draw (0.5,-0.87)-- (1.5,-0.87);
\draw (1.5,-0.87)-- (1,0);
\draw (1,0)-- (2,0);
\draw (1,0)-- (1.5,0.87);
\draw (1.5,0.87)-- (0.5,0.87);
\draw (0.5,0.87)-- (1,1.73);
\draw (0.5,0.87)-- (0,1.73);
\draw (0,1.73)-- (-0.5,0.87);
\draw (-0.5,0.87)-- (-1,1.73);
\draw (-1.5,0.87)-- (-0.5,0.87);
\draw (-1,0)-- (-0.5,0.87);
\draw (-2,0)-- (-1,0);
\draw (-1,0)-- (-1.5,0.87);
\draw (-1.5,-0.87)-- (-1,0);
\draw (-1,0)-- (-0.5,-0.87);
\draw (-0.5,-0.87)-- (-1.5,-0.87);
\draw (-1,-1.73)-- (-0.5,-0.87);
\draw (-0.5,-0.87)-- (0,-1.73);
\draw (0,-1.73)-- (0.5,-0.87);
\draw (0.5,-0.87)-- (-0.5,-0.87);
\draw (0.5,-0.87)-- (1,0);
\draw (0.5,-0.87)-- (0,0);
\draw (1,0)-- (0,0);
\draw (-0.5,-0.87)-- (0,0);
\draw (-1,0)-- (0,0);
\draw (0,0)-- (-0.5,0.87);
\draw (-0.5,0.87)-- (0.5,0.87);
\draw (0,0)-- (0.5,0.87);
\draw (1,0)-- (0.5,0.87);
\begin{scriptsize}
\fill [color=black] (0,0) circle (2pt);
\draw[color=black] (0.3,0.2) node {\normalsize{1}};
\fill [color=black] (1,0) circle (2pt);
\draw[color=black] (1.3,0.2) node {\normalsize{4}};
\fill [color=black] (0.5,0.87) circle (2pt);
\draw[color=black] (0.8,1.06) node {\normalsize{5}};
\fill [color=black] (-0.5,0.87) circle (2pt);
\draw[color=black] (-0.2,1.06) node {\normalsize{6}};
\fill [color=black] (-1,0) circle (2pt);
\draw[color=black] (-0.7,0.2) node {\normalsize{7}};
\fill [color=black] (-0.5,-0.87) circle (2pt);
\draw[color=black] (-0.2,-0.66) node {\normalsize{2}};
\fill [color=black] (0.5,-0.87) circle (2pt);
\draw[color=black] (0.8,-0.66) node {\normalsize{3}};
\fill [color=black] (1.5,0.87) circle (2pt);
\draw[color=black] (1.7,1.06) node {\normalsize{2}};
\fill [color=black] (1,1.73) circle (2pt);
%\draw[color=black] (1.2,1.94) node {\normalsize{7}};
\fill [color=black] (0,1.73) circle (2pt);
%\draw[color=black] (0.3,1.94) node {\normalsize{3}};
\fill [color=black] (-1,1.73) circle (2pt);
\draw[color=black] (-0.7,1.94) node {\normalsize{2}};
\fill [color=black] (-1.5,0.87) circle (2pt);
%\draw[color=black] (-1.2,1.06) node {\normalsize{4}};
\fill [color=black] (-2,0) circle (2pt);
%\draw[color=black] (-1.7,0.2) node {\normalsize{3}};
\fill [color=black] (-1.5,-0.87) circle (2pt);
%\draw[color=black] (-1.2,-0.66) node {\normalsize{5}};
\fill [color=black] (-1,-1.73) circle (2pt);
%\draw[color=black] (-0.7,-1.54) node {\normalsize{4}};
\fill [color=black] (0,-1.73) circle (2pt);
%\draw[color=black] (0.3,-1.54) node {\normalsize{6}};
\fill [color=black] (1,-1.73) circle (2pt);
%\draw[color=black] (1.3,-1.54) node {\normalsize{5}};
\fill [color=black] (1.5,-0.87) circle (2pt);
%\draw[color=black] (1.8,-0.66) node {\normalsize{7}};
\fill [color=black] (2,0) circle (2pt);
%\draw[color=black] (2.2,0.2) node {\normalsize{6}};
\end{scriptsize}
\end{tikzpicture}
%\end{document}
	\captionof{figure}{}
	\label{triangulargrid1feliratos2b}
\end{minipage}
\begin{minipage}{.32\textwidth}
	\centering
	\usetikzlibrary{arrows}
%\pagestyle{empty}
%\begin{document}
\begin{tikzpicture}[line cap=round,line join=round,scale=0.9]%,>=triangle 45,x=1.0cm,y=1.0cm]
\clip(-2.25,-2) rectangle (2.35,2.12);
\draw (-1,-1.73)-- (0,-1.73);
\draw (0,-1.73)-- (1,-1.73);
\draw (1,-1.73)-- (1.5,-0.87);
\draw (1.5,-0.87)-- (2,0);
\draw (2,0)-- (1.5,0.87);
\draw (1.5,0.87)-- (1,1.73);
\draw (1,1.73)-- (0,1.73);
\draw (0,1.73)-- (-1,1.73);
\draw (-1,1.73)-- (-1.5,0.87);
\draw (-1.5,0.87)-- (-2,0);
\draw (-2,0)-- (-1.5,-0.87);
\draw (-1.5,-0.87)-- (-1,-1.73);
\draw (1,-1.73)-- (0.5,-0.87);
\draw (0.5,-0.87)-- (1.5,-0.87);
\draw (1.5,-0.87)-- (1,0);
\draw (1,0)-- (2,0);
\draw (1,0)-- (1.5,0.87);
\draw (1.5,0.87)-- (0.5,0.87);
\draw (0.5,0.87)-- (1,1.73);
\draw (0.5,0.87)-- (0,1.73);
\draw (0,1.73)-- (-0.5,0.87);
\draw (-0.5,0.87)-- (-1,1.73);
\draw (-1.5,0.87)-- (-0.5,0.87);
\draw (-1,0)-- (-0.5,0.87);
\draw (-2,0)-- (-1,0);
\draw (-1,0)-- (-1.5,0.87);
\draw (-1.5,-0.87)-- (-1,0);
\draw (-1,0)-- (-0.5,-0.87);
\draw (-0.5,-0.87)-- (-1.5,-0.87);
\draw (-1,-1.73)-- (-0.5,-0.87);
\draw (-0.5,-0.87)-- (0,-1.73);
\draw (0,-1.73)-- (0.5,-0.87);
\draw (0.5,-0.87)-- (-0.5,-0.87);
\draw (0.5,-0.87)-- (1,0);
\draw (0.5,-0.87)-- (0,0);
\draw (1,0)-- (0,0);
\draw (-0.5,-0.87)-- (0,0);
\draw (-1,0)-- (0,0);
\draw (0,0)-- (-0.5,0.87);
\draw (-0.5,0.87)-- (0.5,0.87);
\draw (0,0)-- (0.5,0.87);
\draw (1,0)-- (0.5,0.87);
\begin{scriptsize}
\fill [color=black] (0,0) circle (2pt);
\draw[color=black] (0.3,0.2) node {\normalsize{1}};
\fill [color=black] (1,0) circle (2pt);
\draw[color=black] (1.3,0.2) node {\normalsize{4}};
\fill [color=black] (0.5,0.87) circle (2pt);
\draw[color=black] (0.8,1.06) node {\normalsize{5}};
\fill [color=black] (-0.5,0.87) circle (2pt);
\draw[color=black] (-0.2,1.06) node {\normalsize{6}};
\fill [color=black] (-1,0) circle (2pt);
\draw[color=black] (-0.7,0.2) node {\normalsize{7}};
\fill [color=black] (-0.5,-0.87) circle (2pt);
\draw[color=black] (-0.2,-0.66) node {\normalsize{2}};
\fill [color=black] (0.5,-0.87) circle (2pt);
\draw[color=black] (0.8,-0.66) node {\normalsize{3}};
\fill [color=black] (1.5,0.87) circle (2pt);
\draw[color=black] (1.7,1.06) node {\normalsize{7}};
\fill [color=black] (1,1.73) circle (2pt);
\draw[color=black] (1.2,1.94) node {\normalsize{3}};
\fill [color=black] (0,1.73) circle (2pt);
\draw[color=black] (0.3,1.94) node {\normalsize{2}};
\fill [color=black] (-1,1.73) circle (2pt);
\draw[color=black] (-0.7,1.94) node {\normalsize{4}};
\fill [color=black] (-1.5,0.87) circle (2pt);
\draw[color=black] (-1.2,1.06) node {\normalsize{3}};
\fill [color=black] (-2,0) circle (2pt);
\draw[color=black] (-1.7,0.2) node {\normalsize{5}};
\fill [color=black] (-1.5,-0.87) circle (2pt);
\draw[color=black] (-1.2,-0.66) node {\normalsize{4}};
\fill [color=black] (-1,-1.73) circle (2pt);
\draw[color=black] (-0.7,-1.54) node {\normalsize{6}};
\fill [color=black] (0,-1.73) circle (2pt);
\draw[color=black] (0.3,-1.54) node {\normalsize{5}};
\fill [color=black] (1,-1.73) circle (2pt);
\draw[color=black] (1.3,-1.54) node {\normalsize{7}};
\fill [color=black] (1.5,-0.87) circle (2pt);
\draw[color=black] (1.8,-0.66) node {\normalsize{6}};
\fill [color=black] (2,0) circle (2pt);
\draw[color=black] (2.2,0.2) node {\normalsize{2}};
\end{scriptsize}
\end{tikzpicture}
%\end{document}
	\captionof{figure}{}
	\label{triangulargrid1feliratos1}
\end{minipage}
\begin{minipage}{.32\textwidth}
	\centering
	\usetikzlibrary{arrows}
%\pagestyle{empty}
%\begin{document}
\begin{tikzpicture}[line cap=round,line join=round,scale=0.9]%,>=triangle 45,x=1.0cm,y=1.0cm]
\clip(-2.25,-2) rectangle (2.35,2.12);
\draw (-1,-1.73)-- (0,-1.73);
\draw (0,-1.73)-- (1,-1.73);
\draw (1,-1.73)-- (1.5,-0.87);
\draw (1.5,-0.87)-- (2,0);
\draw (2,0)-- (1.5,0.87);
\draw (1.5,0.87)-- (1,1.73);
\draw (1,1.73)-- (0,1.73);
\draw (0,1.73)-- (-1,1.73);
\draw (-1,1.73)-- (-1.5,0.87);
\draw (-1.5,0.87)-- (-2,0);
\draw (-2,0)-- (-1.5,-0.87);
\draw (-1.5,-0.87)-- (-1,-1.73);
\draw (1,-1.73)-- (0.5,-0.87);
\draw (0.5,-0.87)-- (1.5,-0.87);
\draw (1.5,-0.87)-- (1,0);
\draw (1,0)-- (2,0);
\draw (1,0)-- (1.5,0.87);
\draw (1.5,0.87)-- (0.5,0.87);
\draw (0.5,0.87)-- (1,1.73);
\draw (0.5,0.87)-- (0,1.73);
\draw (0,1.73)-- (-0.5,0.87);
\draw (-0.5,0.87)-- (-1,1.73);
\draw (-1.5,0.87)-- (-0.5,0.87);
\draw (-1,0)-- (-0.5,0.87);
\draw (-2,0)-- (-1,0);
\draw (-1,0)-- (-1.5,0.87);
\draw (-1.5,-0.87)-- (-1,0);
\draw (-1,0)-- (-0.5,-0.87);
\draw (-0.5,-0.87)-- (-1.5,-0.87);
\draw (-1,-1.73)-- (-0.5,-0.87);
\draw (-0.5,-0.87)-- (0,-1.73);
\draw (0,-1.73)-- (0.5,-0.87);
\draw (0.5,-0.87)-- (-0.5,-0.87);
\draw (0.5,-0.87)-- (1,0);
\draw (0.5,-0.87)-- (0,0);
\draw (1,0)-- (0,0);
\draw (-0.5,-0.87)-- (0,0);
\draw (-1,0)-- (0,0);
\draw (0,0)-- (-0.5,0.87);
\draw (-0.5,0.87)-- (0.5,0.87);
\draw (0,0)-- (0.5,0.87);
\draw (1,0)-- (0.5,0.87);
\begin{scriptsize}
\fill [color=black] (0,0) circle (2pt);
\draw[color=black] (0.3,0.2) node {\normalsize{1}};
\fill [color=black] (1,0) circle (2pt);
\draw[color=black] (1.3,0.2) node {\normalsize{4}};
\fill [color=black] (0.5,0.87) circle (2pt);
\draw[color=black] (0.8,1.06) node {\normalsize{5}};
\fill [color=black] (-0.5,0.87) circle (2pt);
\draw[color=black] (-0.2,1.06) node {\normalsize{6}};
\fill [color=black] (-1,0) circle (2pt);
\draw[color=black] (-0.7,0.2) node {\normalsize{7}};
\fill [color=black] (-0.5,-0.87) circle (2pt);
\draw[color=black] (-0.2,-0.66) node {\normalsize{2}};
\fill [color=black] (0.5,-0.87) circle (2pt);
\draw[color=black] (0.8,-0.66) node {\normalsize{3}};
\fill [color=black] (1.5,0.87) circle (2pt);
\draw[color=black] (1.7,1.06) node {\normalsize{2}};
\fill [color=black] (1,1.73) circle (2pt);
\draw[color=black] (1.2,1.94) node {\normalsize{7}};
\fill [color=black] (0,1.73) circle (2pt);
\draw[color=black] (0.3,1.94) node {\normalsize{3}};
\fill [color=black] (-1,1.73) circle (2pt);
\draw[color=black] (-0.7,1.94) node {\normalsize{2}};
\fill [color=black] (-1.5,0.87) circle (2pt);
\draw[color=black] (-1.2,1.06) node {\normalsize{4}};
\fill [color=black] (-2,0) circle (2pt);
\draw[color=black] (-1.7,0.2) node {\normalsize{3}};
\fill [color=black] (-1.5,-0.87) circle (2pt);
\draw[color=black] (-1.2,-0.66) node {\normalsize{5}};
\fill [color=black] (-1,-1.73) circle (2pt);
\draw[color=black] (-0.7,-1.54) node {\normalsize{4}};
\fill [color=black] (0,-1.73) circle (2pt);
\draw[color=black] (0.3,-1.54) node {\normalsize{6}};
\fill [color=black] (1,-1.73) circle (2pt);
\draw[color=black] (1.3,-1.54) node {\normalsize{5}};
\fill [color=black] (1.5,-0.87) circle (2pt);
\draw[color=black] (1.8,-0.66) node {\normalsize{7}};
\fill [color=black] (2,0) circle (2pt);
\draw[color=black] (2.2,0.2) node {\normalsize{6}};
\end{scriptsize}
\end{tikzpicture}
%\end{document}
	\captionof{figure}{}
	\label{triangulargrid1feliratos2}
\end{minipage}%
\end{center}

\begin{proof}
We may assume without loss of generality that the central $G_h$ is coloured as in Figure \ref{triangulargrid1feliratos}. We now have to colour the remaining $12$ vertices so that all border vertices of the central $G_h$ get exactly one neighbour from all colours (except for its own colour). And since for all colours from $2$ to $7$, there are exactly $3$ coloured vertices that lack a neighbour with that colour and all of the uncoloured vertices border $1$ or $2$ of the coloured ones, we must use all six of these colours at least twice. But since there are $12$ uncoloured vertices in total, we must use all of them exactly twice. Now decide which two vertices get colour $2$. Since only the vertices coloured with $4$, $5$ and $6$ lack a neighbour with colour $2$, the only two possibilities to place these two vertices are those shown in Figure \ref{triangulargrid1feliratos1b} and Figure \ref{triangulargrid1feliratos2b}. Similarly, there are two possibilities for the placement of the vertices with colours $3$, $4$, $5$, $6$ and $7$, only with their placement being rotated. And (progressing in a counterclockwise order if we chose the first option for $2$ and in a clockwise order otherwise) we can prove for these five colours that whichever option was chosen for colour $2$, should be chosen for these colours too (otherwise there would be an overlap between neighbouring colours). So we get the colourings on Figure \ref{triangulargrid1feliratos1} and Figure \ref{triangulargrid1feliratos2}, which indeed are parts of Isbell colourings.
\end{proof}

\begin{center}
\begin{minipage}{.32\textwidth}
	\centering
	\usetikzlibrary{arrows}
%\pagestyle{empty}
%\begin{document}
\begin{tikzpicture}[line cap=round,line join=round]%,>=triangle 45,x=1.0cm,y=1.0cm]
\clip(-1.15,-1) rectangle (1.15,1);
%\draw (1,0)-- (2,0);
\draw (-1,0)-- (-0.5,0.87);
\draw (-1,0)-- (-0.5,-0.87);
\draw (0.5,-0.87)-- (-0.5,-0.87);
\draw (0.5,-0.87)-- (1,0);
\draw (0.5,-0.87)-- (0,0);
\draw (1,0)-- (0,0);
\draw (-0.5,-0.87)-- (0,0);
\draw (-1,0)-- (0,0);
\draw (0,0)-- (-0.5,0.87);
\draw (-0.5,0.87)-- (0.5,0.87);
\draw (0,0)-- (0.5,0.87);
\draw (1,0)-- (0.5,0.87);
\begin{scriptsize}
\fill [color=black] (0,0) circle (1.5pt);
\fill [color=black] (1,0) circle (1.5pt);
\fill [color=black] (0.5,0.87) circle (1.5pt);
\fill [color=black] (-0.5,0.87) circle (1.5pt);
\fill [color=black] (-1,0) circle (1.5pt);
\fill [color=black] (-0.5,-0.87) circle (1.5pt);
\fill [color=black] (0.5,-0.87) circle (1.5pt);
%\fill [color=black] (2,0) circle (1.5pt);
\end{scriptsize}
\end{tikzpicture}
%\end{document}
	\captionof{figure}{$G_h$}
	\label{fig:triangulargrid4}
\end{minipage}
\begin{minipage}{.32\textwidth}
	\centering
	\usetikzlibrary{arrows}
%\pagestyle{empty}
%\begin{document}
\begin{tikzpicture}[line cap=round,line join=round]%,>=triangle 45,x=1.0cm,y=1.0cm]
\clip(-1.15,-1) rectangle (2.15,1);
\draw (1,0)-- (2,0);
\draw (-1,0)-- (-0.5,0.87);
\draw (-1,0)-- (-0.5,-0.87);
\draw (0.5,-0.87)-- (-0.5,-0.87);
\draw (0.5,-0.87)-- (1,0);
\draw (0.5,-0.87)-- (0,0);
\draw (1,0)-- (0,0);
\draw (-0.5,-0.87)-- (0,0);
\draw (-1,0)-- (0,0);
\draw (0,0)-- (-0.5,0.87);
\draw (-0.5,0.87)-- (0.5,0.87);
\draw (0,0)-- (0.5,0.87);
\draw (1,0)-- (0.5,0.87);
\begin{scriptsize}
\fill [color=black] (0,0) circle (1.5pt);
\fill [color=black] (1,0) circle (1.5pt);
\fill [color=black] (0.5,0.87) circle (1.5pt);
\fill [color=black] (-0.5,0.87) circle (1.5pt);
\fill [color=black] (-1,0) circle (1.5pt);
\fill [color=black] (-0.5,-0.87) circle (1.5pt);
\fill [color=black] (0.5,-0.87) circle (1.5pt);
\fill [color=black] (2,0) circle (1.5pt);
\end{scriptsize}
\end{tikzpicture}
%\end{document}
	\captionof{figure}{$G_h^+$}
	\label{fig:triangulargrid2}
\end{minipage}
\begin{minipage}{.32\textwidth}
	\centering
	\usetikzlibrary{arrows}
%\pagestyle{empty}
%\begin{document}
\begin{tikzpicture}[line cap=round,line join=round,scale=0.75]%,>=triangle 45,x=1.0cm,y=1.0cm]
\clip(-2.25,-2) rectangle (1.35,2.12);
\draw (-1,-1.73)-- (0,-1.73);
\draw (0,1.73)-- (-1,1.73);
\draw (-1,1.73)-- (-1.5,0.87);
\draw (-1.5,0.87)-- (-2,0);
\draw (-2,0)-- (-1.5,-0.87);
\draw (-1.5,-0.87)-- (-1,-1.73);
\draw (0.5,0.87)-- (0,1.73);
\draw (0,1.73)-- (-0.5,0.87);
\draw (-0.5,0.87)-- (-1,1.73);
\draw (-1.5,0.87)-- (-0.5,0.87);
\draw (-1,0)-- (-0.5,0.87);
\draw (-2,0)-- (-1,0);
\draw (-1,0)-- (-1.5,0.87);
\draw (-1.5,-0.87)-- (-1,0);
\draw (-1,0)-- (-0.5,-0.87);
\draw (-0.5,-0.87)-- (-1.5,-0.87);
\draw (-1,-1.73)-- (-0.5,-0.87);
\draw (-0.5,-0.87)-- (0,-1.73);
\draw (0,-1.73)-- (0.5,-0.87);
\draw (0.5,-0.87)-- (-0.5,-0.87);
\draw (0.5,-0.87)-- (1,0);
\draw (0.5,-0.87)-- (0,0);
\draw (1,0)-- (0,0);
\draw (-0.5,-0.87)-- (0,0);
\draw (-1,0)-- (0,0);
\draw (0,0)-- (-0.5,0.87);
\draw (-0.5,0.87)-- (0.5,0.87);
\draw (0,0)-- (0.5,0.87);
\draw (1,0)-- (0.5,0.87);
\begin{scriptsize}
\fill [color=black] (0,0) circle (1.5pt);
\fill [color=black] (1,0) circle (1.5pt);
\fill [color=black] (0.5,0.87) circle (1.5pt);
\fill [color=black] (-0.5,0.87) circle (1.5pt);
\fill [color=black] (-1,0) circle (1.5pt);
\fill [color=black] (-0.5,-0.87) circle (1.5pt);
\fill [color=black] (0.5,-0.87) circle (1.5pt);
\fill [color=black] (0,1.73) circle (1.5pt);
\fill [color=black] (-1,1.73) circle (1.5pt);
\fill [color=black] (-1.5,0.87) circle (1.5pt);
\fill [color=black] (-2,0) circle (1.5pt);
\fill [color=black] (-1.5,-0.87) circle (1.5pt);
\fill [color=black] (-1,-1.73) circle (1.5pt);
\fill [color=black] (0,-1.73) circle (1.5pt);
\end{scriptsize}
\end{tikzpicture}
%\end{document}
	\captionof{figure}{$G_H^-$}
	\label{fig:triangulargrid3}
\end{minipage}
\end{center}

\begin{lemma}
\label{lem:Isbell2}
If we embed $G_h^+$ (see Figure \ref{fig:triangulargrid2}) in the infinite triangular grid, then any colouring of it is contained in at most one Isbell colouring of $T$.
\end{lemma}

\begin{proof}
As it was stated above, an Isbell colouring can be generated from the colouring of any of its $G_h$'s with only two options and the colouring of the extra vertex determines which option is used.
\end{proof}

We will use the following lemmas later in the proof:

\begin{lemma}
\label{lem:length}
For any path or closed path $p$ in $G$, $L(p)<d_1l(p)$.
\end{lemma}

\begin{proof}
Since all pairs of neighbouring vertices have a distance less than $d_1$, the inequality directly follows from summing these inequalities up.
\end{proof}

\begin{lemma}
\label{lem:curvecontraction}
Any cycle $c$ can be contracted to a cycle $c'$, for which $int(c')\cup c'$ is a triangle using the following two kind of steps:

\begin{center}
\begin{minipage}{.4\textwidth}
	\centering
	\usetikzlibrary{arrows}
%\pagestyle{empty}
%\begin{document}
\definecolor{qqqqff}{rgb}{0,0,1}
\begin{tikzpicture}[line cap=round,line join=round,scale=0.25]%,>=triangle 45,x=1.0cm,y=1.0cm]
\clip(1.03,-6.92) rectangle (13.81,6.82);
\draw [line width=2pt] (10.64,-4.36)-- (11.76,-0.74);
\draw [line width=2pt] (11.76,-0.74)-- (10.62,0.82);
\draw [line width=2pt] (9.9,2.96)-- (7.66,3.18);
\draw [line width=2pt] (7.66,3.18)-- (6.88,4.58);
\draw [line width=2pt] (6.88,4.58)-- (4.7,4.5);
\draw [line width=2pt] (4.7,4.5)-- (2.96,3.06);
\draw [line width=2pt] (2.96,3.06)-- (1.58,1.08);
\draw [line width=2pt] (1.58,1.08)-- (2.7,-1.62);
\draw [line width=2pt] (2.7,-1.62)-- (3.52,-5.14);
\draw (7.44,-2.34)-- (2.7,-1.62);
\draw (7.44,-2.34)-- (10.62,0.82);
\draw (10.64,-4.36)-- (10.62,0.82);
\draw (10.62,0.82)-- (7.22,0.84);
\draw (7.22,0.84)-- (7.44,-2.34);
\draw (7.22,0.84)-- (7.66,3.18);
\draw (4.7,0.44)-- (2.7,-1.62);
\draw (1.58,1.08)-- (4.7,0.44);
\draw (4.7,0.44)-- (7.44,-2.34);
\draw (3.52,-5.14)-- (2.62,-6.6);
\draw (3.52,-5.14)-- (8.84,-5.5);
\draw (3.52,-5.14)-- (0.08,-1.22);
\draw (10.64,-4.36)-- (8.84,-5.5);
\draw (10.64,-4.36)-- (12.58,-3.94);
\draw (12.58,-3.94)-- (13.42,-1.9);
\draw (12.58,-3.94)-- (11.76,-0.74);
\draw (13.42,-1.9)-- (11.76,-0.74);
\draw (11.76,-0.74)-- (12.18,0.48);
\draw (12.18,0.48)-- (9.9,2.96);
\draw (12.18,0.48)-- (11.04,4.7);
\draw (7.6,5.24)-- (6.88,4.58);
\draw (4.1,5.88)-- (4.7,4.5);
\draw (4.1,5.88)-- (-0.18,5.6);
\draw (-0.18,5.6)-- (2.96,3.06);
\draw (-0.18,5.6)-- (1.58,1.08);
\draw (1.58,1.08)-- (0.08,-1.22);
\draw (2.7,-1.62)-- (0.08,-1.22);
\draw (0.08,-1.22)-- (2.62,-6.6);
\draw (2.62,-6.6)-- (8.84,-5.5);
\draw (8.84,-5.5)-- (10.36,-6.56);
\draw (10.36,-6.56)-- (12.58,-3.94);
\draw (10.36,-6.56)-- (14.78,-5.7);
\draw (14.78,-5.7)-- (12.58,-3.94);
\draw (14.78,-5.7)-- (13.42,-1.9);
\draw (13.42,-1.9)-- (13.1,0.84);
\draw (13.1,0.84)-- (12.18,0.48);
\draw (13.42,-1.9)-- (15.22,-1.66);
\draw (15.22,-1.66)-- (13.1,0.84);
\draw (13.1,0.84)-- (11.94,4.88);
\draw (14.4,5.9)-- (11.94,4.88);
\draw (11.94,4.88)-- (11.04,4.7);
\draw (7.6,5.24)-- (9,6.18);
\draw (9.9,2.96)-- (9,6.18);
\draw (11.04,4.7)-- (9,6.18);
\draw (-1.86,5.2)-- (0.08,-1.22);
\draw (-0.18,5.6)-- (-1.86,5.2);
\draw (-1.86,5.2)-- (-3.72,8.7);
\draw (-4.98,-8.3)-- (1.98,-9.04);
\draw (1.98,-9.04)-- (2.62,-6.6);
\draw (1.98,-9.04)-- (10.36,-6.56);
\draw (1.98,-9.04)-- (12.16,-7.26);
\draw (12.16,-7.26)-- (10.36,-6.56);
\draw (12.16,-7.26)-- (14.78,-5.7);
\draw (12.16,-7.26)-- (15.92,-7.82);
\draw (15.92,-7.82)-- (1.98,-9.04);
\draw (15.92,-7.82)-- (14.78,-5.7);
\draw (-3.72,8.7)-- (6.5,8.88);
\draw (-0.18,5.6)-- (6.5,8.88);
\draw (-1.86,5.2)-- (1.58,1.08);
\draw (-1.76,-2.94)-- (1.98,-9.04);
\draw (-1.76,-2.94)-- (2.62,-6.6);
\draw (-1.76,-2.94)-- (0.08,-1.22);
\draw (-1.76,-2.94)-- (-4.98,-8.3);
\draw (-1.76,-2.94)-- (-1.86,5.2);
\draw (-4.98,-8.3)-- (-3.72,8.7);
\draw (-1.86,5.2)-- (6.5,8.88);
\draw (11.04,4.7)-- (13.1,0.84);
\draw (4.7,4.5)-- (-0.18,5.6);
\draw (4.1,5.88)-- (7.84,8.88);
\draw (9,6.18)-- (7.84,8.88);
\draw (7.6,5.24)-- (7.84,8.88);
\draw (6.88,4.58)-- (7.84,8.88);
\draw (6.5,8.88)-- (7.84,8.88);
\draw (-3.72,8.7)-- (-1.76,-2.94);
\draw (11.04,4.7)-- (8.98,9.52);
\draw (7.84,8.88)-- (8.98,9.52);
\draw (9,6.18)-- (8.98,9.52);
\draw (8.98,9.52)-- (6.5,8.88);
\draw (13.42,-1.9)-- (12.18,0.48);
\draw (11.94,4.88)-- (9.74,11);
\draw (11.04,4.7)-- (9.74,11);
\draw (8.98,9.52)-- (-3.72,8.7);
\draw (8.98,9.52)-- (9.74,11);
\draw (-3.72,8.7)-- (9.74,11);
\draw [shift={(-5.41,-5.59)}] plot[domain=0.33:0.69,variable=\t]({1*18.61*cos(\t r)+0*18.61*sin(\t r)},{0*18.61*cos(\t r)+1*18.61*sin(\t r)});
\draw [shift={(3.21,0.36)}] plot[domain=0.86:1.41,variable=\t]({1*5.59*cos(\t r)+0*5.59*sin(\t r)},{0*5.59*cos(\t r)+1*5.59*sin(\t r)});
\draw [shift={(8.09,1.57)}] plot[domain=4.82:5.4,variable=\t]({1*7.11*cos(\t r)+0*7.11*sin(\t r)},{0*7.11*cos(\t r)+1*7.11*sin(\t r)});
\draw [shift={(9.8,3.78)}] plot[domain=2.14:2.79,variable=\t]({1*6.07*cos(\t r)+0*6.07*sin(\t r)},{0*6.07*cos(\t r)+1*6.07*sin(\t r)});
\draw [shift={(8.57,-18.33)},line width=2pt]  plot[domain=1.42:1.94,variable=\t]({1*14.13*cos(\t r)+0*14.13*sin(\t r)},{0*14.13*cos(\t r)+1*14.13*sin(\t r)});
\draw [shift={(2.55,-3.05)}] plot[domain=0.25:0.69,variable=\t]({1*9.49*cos(\t r)+0*9.49*sin(\t r)},{0*9.49*cos(\t r)+1*9.49*sin(\t r)});
\draw [shift={(12.95,6.78)}] plot[domain=3.74:4.34,variable=\t]({1*6.4*cos(\t r)+0*6.4*sin(\t r)},{0*6.4*cos(\t r)+1*6.4*sin(\t r)});
\draw [shift={(17.64,4.37)},line width=2pt]  plot[domain=3.32:3.61,variable=\t]({1*7.86*cos(\t r)+0*7.86*sin(\t r)},{0*7.86*cos(\t r)+1*7.86*sin(\t r)});
\draw [shift={(10.68,-11.02)}] plot[domain=1.93:2.45,variable=\t]({1*9.26*cos(\t r)+0*9.26*sin(\t r)},{0*9.26*cos(\t r)+1*9.26*sin(\t r)});
\draw [shift={(5.83,-8.44)}] plot[domain=0.7:1.31,variable=\t]({1*6.31*cos(\t r)+0*6.31*sin(\t r)},{0*6.31*cos(\t r)+1*6.31*sin(\t r)});
\draw [shift={(6.32,25.58)}] plot[domain=4.6:4.84,variable=\t]({1*32.4*cos(\t r)+0*32.4*sin(\t r)},{0*32.4*cos(\t r)+1*32.4*sin(\t r)});
\draw [shift={(9.59,11.48)}] plot[domain=4.1:4.48,variable=\t]({1*8.52*cos(\t r)+0*8.52*sin(\t r)},{0*8.52*cos(\t r)+1*8.52*sin(\t r)});
\draw [shift={(0.2,-5.1)}] plot[domain=0.23:0.7,variable=\t]({1*15.41*cos(\t r)+0*15.41*sin(\t r)},{0*15.41*cos(\t r)+1*15.41*sin(\t r)});
\draw [shift={(3.92,-0.78)}] plot[domain=0.56:1.02,variable=\t]({1*7.05*cos(\t r)+0*7.05*sin(\t r)},{0*7.05*cos(\t r)+1*7.05*sin(\t r)});
\draw [shift={(1.06,4.02)}] plot[domain=-0.13:0.18,variable=\t]({1*6.65*cos(\t r)+0*6.65*sin(\t r)},{0*6.65*cos(\t r)+1*6.65*sin(\t r)});
\draw [shift={(8.62,4.93)}] plot[domain=3.46:3.99,variable=\t]({1*5.97*cos(\t r)+0*5.97*sin(\t r)},{0*5.97*cos(\t r)+1*5.97*sin(\t r)});
\draw [shift={(5.22,5.29)}] plot[domain=4.61:5.13,variable=\t]({1*4.88*cos(\t r)+0*4.88*sin(\t r)},{0*4.88*cos(\t r)+1*4.88*sin(\t r)});
\draw [shift={(9.34,10.11)}] plot[domain=3.98:4.49,variable=\t]({1*9.5*cos(\t r)+0*9.5*sin(\t r)},{0*9.5*cos(\t r)+1*9.5*sin(\t r)});
\draw [shift={(12.13,6.92)}] plot[domain=3.46:4.03,variable=\t]({1*7.81*cos(\t r)+0*7.81*sin(\t r)},{0*7.81*cos(\t r)+1*7.81*sin(\t r)});
\begin{scriptsize}
\fill [color=black] (3.52,-5.14) circle (5pt);
\fill [color=black] (10.64,-4.36) circle (5pt);
\fill [color=black] (11.76,-0.74) circle (5pt);
\fill [color=black] (7.22,0.84) circle (5pt);
\fill [color=black] (4.7,0.44) circle (5pt);
\fill [color=black] (7.44,-2.34) circle (5pt);
\fill [color=black] (10.62,0.82) circle (5pt);
\fill [color=black] (9.9,2.96) circle (5pt);
\fill [color=black] (6.88,4.58) circle (5pt);
\fill [color=black] (2.96,3.06) circle (5pt);
\fill [color=black] (2.7,-1.62) circle (5pt);
\fill [color=black] (1.58,1.08) circle (5pt);
\fill [color=black] (7.66,3.18) circle (5pt);
\fill [color=black] (4.7,4.5) circle (5pt);
\fill [color=black] (2.62,-6.6) circle (5pt);
\fill [color=black] (8.84,-5.5) circle (5pt);
\fill [color=black] (12.58,-3.94) circle (5pt);
\fill [color=black] (13.42,-1.9) circle (5pt);
\fill [color=black] (12.18,0.48) circle (5pt);
\fill [color=black] (11.04,4.7) circle (5pt);
\fill [color=black] (7.6,5.24) circle (5pt);
\fill [color=black] (4.1,5.88) circle (5pt);
\fill [color=qqqqff] (-0.18,5.6) circle (1.5pt);
\fill [color=qqqqff] (0.08,-1.22) circle (1.5pt);
\fill [color=qqqqff] (-1.86,5.2) circle (1.5pt);
\fill [color=qqqqff] (-3.72,8.7) circle (1.5pt);
\fill [color=black] (13.1,0.84) circle (5pt);
\fill [color=black] (11.94,4.88) circle (5pt);
\fill [color=black] (10.36,-6.56) circle (5pt);
\fill [color=qqqqff] (14.78,-5.7) circle (1.5pt);
\fill [color=qqqqff] (15.22,-1.66) circle (1.5pt);
\fill [color=qqqqff] (14.4,5.9) circle (1.5pt);
\fill [color=qqqqff] (7.84,8.88) circle (1.5pt);
\fill [color=black] (9,6.18) circle (5pt);
\fill [color=qqqqff] (6.5,8.88) circle (1.5pt);
\fill [color=qqqqff] (-4.98,-8.3) circle (1.5pt);
\fill [color=qqqqff] (1.98,-9.04) circle (1.5pt);
\fill [color=qqqqff] (12.16,-7.26) circle (1.5pt);
\fill [color=qqqqff] (15.92,-7.82) circle (1.5pt);
\fill [color=qqqqff] (-1.76,-2.94) circle (1.5pt);
\fill [color=qqqqff] (8.98,9.52) circle (1.5pt);
\fill [color=qqqqff] (9.74,11) circle (1.5pt);
\end{scriptsize}
\end{tikzpicture}
%\end{document}
\end{minipage}
	$\rightarrow$
\begin{minipage}{.4\textwidth}
	\centering
	\usetikzlibrary{arrows}
%\pagestyle{empty}
%\begin{document}
\definecolor{qqqqff}{rgb}{0,0,1}
\begin{tikzpicture}[line cap=round,line join=round,scale=0.25]%,>=triangle 45,x=1.0cm,y=1.0cm]
\clip(1.03,-6.92) rectangle (13.81,6.82);
\draw [line width=2pt] (10.64,-4.36)-- (11.76,-0.74);
\draw [line width=2pt] (11.76,-0.74)-- (10.62,0.82);
\draw (9.9,2.96)-- (7.66,3.18);
\draw [line width=2pt] (7.66,3.18)-- (6.88,4.58);
\draw [line width=2pt] (6.88,4.58)-- (4.7,4.5);
\draw [line width=2pt] (4.7,4.5)-- (2.96,3.06);
\draw [line width=2pt] (2.96,3.06)-- (1.58,1.08);
\draw [line width=2pt] (1.58,1.08)-- (2.7,-1.62);
\draw [line width=2pt] (2.7,-1.62)-- (3.52,-5.14);
\draw (7.44,-2.34)-- (2.7,-1.62);
\draw (7.44,-2.34)-- (10.62,0.82);
\draw (10.64,-4.36)-- (10.62,0.82);
\draw (10.62,0.82)-- (7.22,0.84);
\draw (7.22,0.84)-- (7.44,-2.34);
\draw (7.22,0.84)-- (7.66,3.18);
\draw (4.7,0.44)-- (2.7,-1.62);
\draw (1.58,1.08)-- (4.7,0.44);
\draw (4.7,0.44)-- (7.44,-2.34);
\draw (3.52,-5.14)-- (2.62,-6.6);
\draw (3.52,-5.14)-- (8.84,-5.5);
\draw (3.52,-5.14)-- (0.08,-1.22);
\draw (10.64,-4.36)-- (8.84,-5.5);
\draw (10.64,-4.36)-- (12.58,-3.94);
\draw (12.58,-3.94)-- (13.42,-1.9);
\draw (12.58,-3.94)-- (11.76,-0.74);
\draw (13.42,-1.9)-- (11.76,-0.74);
\draw (11.76,-0.74)-- (12.18,0.48);
\draw (12.18,0.48)-- (9.9,2.96);
\draw (12.18,0.48)-- (11.04,4.7);
\draw (7.6,5.24)-- (6.88,4.58);
\draw (4.1,5.88)-- (4.7,4.5);
\draw (4.1,5.88)-- (-0.18,5.6);
\draw (-0.18,5.6)-- (2.96,3.06);
\draw (-0.18,5.6)-- (1.58,1.08);
\draw (1.58,1.08)-- (0.08,-1.22);
\draw (2.7,-1.62)-- (0.08,-1.22);
\draw (0.08,-1.22)-- (2.62,-6.6);
\draw (2.62,-6.6)-- (8.84,-5.5);
\draw (8.84,-5.5)-- (10.36,-6.56);
\draw (10.36,-6.56)-- (12.58,-3.94);
\draw (10.36,-6.56)-- (14.78,-5.7);
\draw (14.78,-5.7)-- (12.58,-3.94);
\draw (14.78,-5.7)-- (13.42,-1.9);
\draw (13.42,-1.9)-- (13.1,0.84);
\draw (13.1,0.84)-- (12.18,0.48);
\draw (13.42,-1.9)-- (15.22,-1.66);
\draw (15.22,-1.66)-- (13.1,0.84);
\draw (13.1,0.84)-- (11.94,4.88);
\draw (14.4,5.9)-- (11.94,4.88);
\draw (11.94,4.88)-- (11.04,4.7);
\draw (7.6,5.24)-- (9,6.18);
\draw (9.9,2.96)-- (9,6.18);
\draw (11.04,4.7)-- (9,6.18);
\draw (-1.86,5.2)-- (0.08,-1.22);
\draw (-0.18,5.6)-- (-1.86,5.2);
\draw (-1.86,5.2)-- (-3.72,8.7);
\draw (-4.98,-8.3)-- (1.98,-9.04);
\draw (1.98,-9.04)-- (2.62,-6.6);
\draw (1.98,-9.04)-- (10.36,-6.56);
\draw (1.98,-9.04)-- (12.16,-7.26);
\draw (12.16,-7.26)-- (10.36,-6.56);
\draw (12.16,-7.26)-- (14.78,-5.7);
\draw (12.16,-7.26)-- (15.92,-7.82);
\draw (15.92,-7.82)-- (1.98,-9.04);
\draw (15.92,-7.82)-- (14.78,-5.7);
\draw (-3.72,8.7)-- (6.5,8.88);
\draw (-0.18,5.6)-- (6.5,8.88);
\draw (-1.86,5.2)-- (1.58,1.08);
\draw (-1.76,-2.94)-- (1.98,-9.04);
\draw (-1.76,-2.94)-- (2.62,-6.6);
\draw (-1.76,-2.94)-- (0.08,-1.22);
\draw (-1.76,-2.94)-- (-4.98,-8.3);
\draw (-1.76,-2.94)-- (-1.86,5.2);
\draw (-4.98,-8.3)-- (-3.72,8.7);
\draw (-1.86,5.2)-- (6.5,8.88);
\draw (11.04,4.7)-- (13.1,0.84);
\draw (4.7,4.5)-- (-0.18,5.6);
\draw (4.1,5.88)-- (7.84,8.88);
\draw (9,6.18)-- (7.84,8.88);
\draw (7.6,5.24)-- (7.84,8.88);
\draw (6.88,4.58)-- (7.84,8.88);
\draw (6.5,8.88)-- (7.84,8.88);
\draw (-3.72,8.7)-- (-1.76,-2.94);
\draw (11.04,4.7)-- (8.98,9.52);
\draw (7.84,8.88)-- (8.98,9.52);
\draw (9,6.18)-- (8.98,9.52);
\draw (8.98,9.52)-- (6.5,8.88);
\draw (13.42,-1.9)-- (12.18,0.48);
\draw (11.94,4.88)-- (9.74,11);
\draw (11.04,4.7)-- (9.74,11);
\draw (8.98,9.52)-- (-3.72,8.7);
\draw (8.98,9.52)-- (9.74,11);
\draw (-3.72,8.7)-- (9.74,11);
\draw [shift={(-5.41,-5.59)}] plot[domain=0.33:0.69,variable=\t]({1*18.61*cos(\t r)+0*18.61*sin(\t r)},{0*18.61*cos(\t r)+1*18.61*sin(\t r)});
\draw [shift={(3.21,0.36)}] plot[domain=0.86:1.41,variable=\t]({1*5.59*cos(\t r)+0*5.59*sin(\t r)},{0*5.59*cos(\t r)+1*5.59*sin(\t r)});
\draw [shift={(8.09,1.57)}] plot[domain=4.82:5.4,variable=\t]({1*7.11*cos(\t r)+0*7.11*sin(\t r)},{0*7.11*cos(\t r)+1*7.11*sin(\t r)});
\draw [shift={(9.8,3.78)}] plot[domain=2.14:2.79,variable=\t]({1*6.07*cos(\t r)+0*6.07*sin(\t r)},{0*6.07*cos(\t r)+1*6.07*sin(\t r)});
\draw [shift={(8.57,-18.33)},line width=2pt]  plot[domain=1.42:1.94,variable=\t]({1*14.13*cos(\t r)+0*14.13*sin(\t r)},{0*14.13*cos(\t r)+1*14.13*sin(\t r)});
\draw [shift={(2.55,-3.05)}] plot[domain=0.25:0.69,variable=\t]({1*9.49*cos(\t r)+0*9.49*sin(\t r)},{0*9.49*cos(\t r)+1*9.49*sin(\t r)});
\draw [shift={(12.95,6.78)},line width=2pt] plot[domain=3.74:4.34,variable=\t]({1*6.4*cos(\t r)+0*6.4*sin(\t r)},{0*6.4*cos(\t r)+1*6.4*sin(\t r)});
\draw [shift={(17.64,4.37)}]  plot[domain=3.32:3.61,variable=\t]({1*7.86*cos(\t r)+0*7.86*sin(\t r)},{0*7.86*cos(\t r)+1*7.86*sin(\t r)});
\draw [shift={(10.68,-11.02)}] plot[domain=1.93:2.45,variable=\t]({1*9.26*cos(\t r)+0*9.26*sin(\t r)},{0*9.26*cos(\t r)+1*9.26*sin(\t r)});
\draw [shift={(5.83,-8.44)}] plot[domain=0.7:1.31,variable=\t]({1*6.31*cos(\t r)+0*6.31*sin(\t r)},{0*6.31*cos(\t r)+1*6.31*sin(\t r)});
\draw [shift={(6.32,25.58)}] plot[domain=4.6:4.84,variable=\t]({1*32.4*cos(\t r)+0*32.4*sin(\t r)},{0*32.4*cos(\t r)+1*32.4*sin(\t r)});
\draw [shift={(9.59,11.48)}] plot[domain=4.1:4.48,variable=\t]({1*8.52*cos(\t r)+0*8.52*sin(\t r)},{0*8.52*cos(\t r)+1*8.52*sin(\t r)});
\draw [shift={(0.2,-5.1)}] plot[domain=0.23:0.7,variable=\t]({1*15.41*cos(\t r)+0*15.41*sin(\t r)},{0*15.41*cos(\t r)+1*15.41*sin(\t r)});
\draw [shift={(3.92,-0.78)}] plot[domain=0.56:1.02,variable=\t]({1*7.05*cos(\t r)+0*7.05*sin(\t r)},{0*7.05*cos(\t r)+1*7.05*sin(\t r)});
\draw [shift={(1.06,4.02)}] plot[domain=-0.13:0.18,variable=\t]({1*6.65*cos(\t r)+0*6.65*sin(\t r)},{0*6.65*cos(\t r)+1*6.65*sin(\t r)});
\draw [shift={(8.62,4.93)}] plot[domain=3.46:3.99,variable=\t]({1*5.97*cos(\t r)+0*5.97*sin(\t r)},{0*5.97*cos(\t r)+1*5.97*sin(\t r)});
\draw [shift={(5.22,5.29)}] plot[domain=4.61:5.13,variable=\t]({1*4.88*cos(\t r)+0*4.88*sin(\t r)},{0*4.88*cos(\t r)+1*4.88*sin(\t r)});
\draw [shift={(9.34,10.11)}] plot[domain=3.98:4.49,variable=\t]({1*9.5*cos(\t r)+0*9.5*sin(\t r)},{0*9.5*cos(\t r)+1*9.5*sin(\t r)});
\draw [shift={(12.13,6.92)}] plot[domain=3.46:4.03,variable=\t]({1*7.81*cos(\t r)+0*7.81*sin(\t r)},{0*7.81*cos(\t r)+1*7.81*sin(\t r)});
\begin{scriptsize}
\fill [color=black] (3.52,-5.14) circle (5pt);
\fill [color=black] (10.64,-4.36) circle (5pt);
\fill [color=black] (11.76,-0.74) circle (5pt);
\fill [color=black] (7.22,0.84) circle (5pt);
\fill [color=black] (4.7,0.44) circle (5pt);
\fill [color=black] (7.44,-2.34) circle (5pt);
\fill [color=black] (10.62,0.82) circle (5pt);
\fill [color=black] (9.9,2.96) circle (5pt);
\fill [color=black] (6.88,4.58) circle (5pt);
\fill [color=black] (2.96,3.06) circle (5pt);
\fill [color=black] (2.7,-1.62) circle (5pt);
\fill [color=black] (1.58,1.08) circle (5pt);
\fill [color=black] (7.66,3.18) circle (5pt);
\fill [color=black] (4.7,4.5) circle (5pt);
\fill [color=black] (2.62,-6.6) circle (5pt);
\fill [color=black] (8.84,-5.5) circle (5pt);
\fill [color=black] (12.58,-3.94) circle (5pt);
\fill [color=black] (13.42,-1.9) circle (5pt);
\fill [color=black] (12.18,0.48) circle (5pt);
\fill [color=black] (11.04,4.7) circle (5pt);
\fill [color=black] (7.6,5.24) circle (5pt);
\fill [color=black] (4.1,5.88) circle (5pt);
\fill [color=qqqqff] (-0.18,5.6) circle (1.5pt);
\fill [color=qqqqff] (0.08,-1.22) circle (1.5pt);
\fill [color=qqqqff] (-1.86,5.2) circle (1.5pt);
\fill [color=qqqqff] (-3.72,8.7) circle (1.5pt);
\fill [color=black] (13.1,0.84) circle (5pt);
\fill [color=black] (11.94,4.88) circle (5pt);
\fill [color=black] (10.36,-6.56) circle (5pt);
\fill [color=qqqqff] (14.78,-5.7) circle (1.5pt);
\fill [color=qqqqff] (15.22,-1.66) circle (1.5pt);
\fill [color=qqqqff] (14.4,5.9) circle (1.5pt);
\fill [color=qqqqff] (7.84,8.88) circle (1.5pt);
\fill [color=black] (9,6.18) circle (5pt);
\fill [color=qqqqff] (6.5,8.88) circle (1.5pt);
\fill [color=qqqqff] (-4.98,-8.3) circle (1.5pt);
\fill [color=qqqqff] (1.98,-9.04) circle (1.5pt);
\fill [color=qqqqff] (12.16,-7.26) circle (1.5pt);
\fill [color=qqqqff] (15.92,-7.82) circle (1.5pt);
\fill [color=qqqqff] (-1.76,-2.94) circle (1.5pt);
\fill [color=qqqqff] (8.98,9.52) circle (1.5pt);
\fill [color=qqqqff] (9.74,11) circle (1.5pt);
\end{scriptsize}
\end{tikzpicture}
%\end{document}
\end{minipage}
	
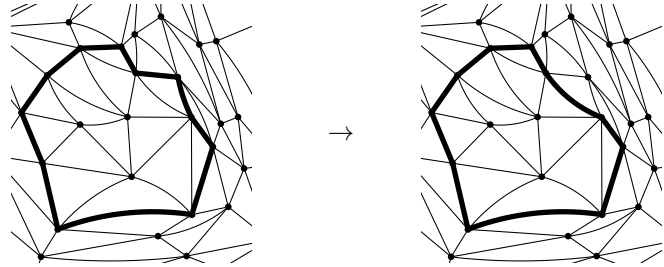
\captionof{figure}{Step type 1}
	\label{curvecontraction1}
\end{center}

1. Take a triangle $v_iv_jv_k$ in $int(c)\cup c$, which has exactly two sides in common with $c$. These two sides trivially form a $2$-path and we can assume without loss of generality that they are $(v_i,v_j)$ and $(v_j,v_k)$. Throw these sides out of $c$ and replace them with $(v_i,v_k)$.

\begin{center}
\begin{minipage}{.4\textwidth}
	\centering
	\usetikzlibrary{arrows}
%\pagestyle{empty}
%\begin{document}
\definecolor{qqqqff}{rgb}{0,0,1}
\begin{tikzpicture}[line cap=round,line join=round,scale=0.25]%,>=triangle 45,x=1.0cm,y=1.0cm]
\clip(1.03,-6.92) rectangle (13.81,6.82);
\draw [line width=2pt] (10.64,-4.36)-- (11.76,-0.74);
\draw [line width=2pt] (11.76,-0.74)-- (10.62,0.82);
\draw (9.9,2.96)-- (7.66,3.18);
\draw [line width=2pt] (7.66,3.18)-- (6.88,4.58);
\draw [line width=2pt] (6.88,4.58)-- (4.7,4.5);
\draw [line width=2pt] (4.7,4.5)-- (2.96,3.06);
\draw [line width=2pt] (2.96,3.06)-- (1.58,1.08);
\draw [line width=2pt] (1.58,1.08)-- (2.7,-1.62);
\draw [line width=2pt] (2.7,-1.62)-- (3.52,-5.14);
\draw (7.44,-2.34)-- (2.7,-1.62);
\draw (7.44,-2.34)-- (10.62,0.82);
\draw (10.64,-4.36)-- (10.62,0.82);
\draw (10.62,0.82)-- (7.22,0.84);
\draw (7.22,0.84)-- (7.44,-2.34);
\draw (7.22,0.84)-- (7.66,3.18);
\draw (4.7,0.44)-- (2.7,-1.62);
\draw (1.58,1.08)-- (4.7,0.44);
\draw (4.7,0.44)-- (7.44,-2.34);
\draw (3.52,-5.14)-- (2.62,-6.6);
\draw (3.52,-5.14)-- (8.84,-5.5);
\draw (3.52,-5.14)-- (0.08,-1.22);
\draw (10.64,-4.36)-- (8.84,-5.5);
\draw (10.64,-4.36)-- (12.58,-3.94);
\draw (12.58,-3.94)-- (13.42,-1.9);
\draw (12.58,-3.94)-- (11.76,-0.74);
\draw (13.42,-1.9)-- (11.76,-0.74);
\draw (11.76,-0.74)-- (12.18,0.48);
\draw (12.18,0.48)-- (9.9,2.96);
\draw (12.18,0.48)-- (11.04,4.7);
\draw (7.6,5.24)-- (6.88,4.58);
\draw (4.1,5.88)-- (4.7,4.5);
\draw (4.1,5.88)-- (-0.18,5.6);
\draw (-0.18,5.6)-- (2.96,3.06);
\draw (-0.18,5.6)-- (1.58,1.08);
\draw (1.58,1.08)-- (0.08,-1.22);
\draw (2.7,-1.62)-- (0.08,-1.22);
\draw (0.08,-1.22)-- (2.62,-6.6);
\draw (2.62,-6.6)-- (8.84,-5.5);
\draw (8.84,-5.5)-- (10.36,-6.56);
\draw (10.36,-6.56)-- (12.58,-3.94);
\draw (10.36,-6.56)-- (14.78,-5.7);
\draw (14.78,-5.7)-- (12.58,-3.94);
\draw (14.78,-5.7)-- (13.42,-1.9);
\draw (13.42,-1.9)-- (13.1,0.84);
\draw (13.1,0.84)-- (12.18,0.48);
\draw (13.42,-1.9)-- (15.22,-1.66);
\draw (15.22,-1.66)-- (13.1,0.84);
\draw (13.1,0.84)-- (11.94,4.88);
\draw (14.4,5.9)-- (11.94,4.88);
\draw (11.94,4.88)-- (11.04,4.7);
\draw (7.6,5.24)-- (9,6.18);
\draw (9.9,2.96)-- (9,6.18);
\draw (11.04,4.7)-- (9,6.18);
\draw (-1.86,5.2)-- (0.08,-1.22);
\draw (-0.18,5.6)-- (-1.86,5.2);
\draw (-1.86,5.2)-- (-3.72,8.7);
\draw (-4.98,-8.3)-- (1.98,-9.04);
\draw (1.98,-9.04)-- (2.62,-6.6);
\draw (1.98,-9.04)-- (10.36,-6.56);
\draw (1.98,-9.04)-- (12.16,-7.26);
\draw (12.16,-7.26)-- (10.36,-6.56);
\draw (12.16,-7.26)-- (14.78,-5.7);
\draw (12.16,-7.26)-- (15.92,-7.82);
\draw (15.92,-7.82)-- (1.98,-9.04);
\draw (15.92,-7.82)-- (14.78,-5.7);
\draw (-3.72,8.7)-- (6.5,8.88);
\draw (-0.18,5.6)-- (6.5,8.88);
\draw (-1.86,5.2)-- (1.58,1.08);
\draw (-1.76,-2.94)-- (1.98,-9.04);
\draw (-1.76,-2.94)-- (2.62,-6.6);
\draw (-1.76,-2.94)-- (0.08,-1.22);
\draw (-1.76,-2.94)-- (-4.98,-8.3);
\draw (-1.76,-2.94)-- (-1.86,5.2);
\draw (-4.98,-8.3)-- (-3.72,8.7);
\draw (-1.86,5.2)-- (6.5,8.88);
\draw (11.04,4.7)-- (13.1,0.84);
\draw (4.7,4.5)-- (-0.18,5.6);
\draw (4.1,5.88)-- (7.84,8.88);
\draw (9,6.18)-- (7.84,8.88);
\draw (7.6,5.24)-- (7.84,8.88);
\draw (6.88,4.58)-- (7.84,8.88);
\draw (6.5,8.88)-- (7.84,8.88);
\draw (-3.72,8.7)-- (-1.76,-2.94);
\draw (11.04,4.7)-- (8.98,9.52);
\draw (7.84,8.88)-- (8.98,9.52);
\draw (9,6.18)-- (8.98,9.52);
\draw (8.98,9.52)-- (6.5,8.88);
\draw (13.42,-1.9)-- (12.18,0.48);
\draw (11.94,4.88)-- (9.74,11);
\draw (11.04,4.7)-- (9.74,11);
\draw (8.98,9.52)-- (-3.72,8.7);
\draw (8.98,9.52)-- (9.74,11);
\draw (-3.72,8.7)-- (9.74,11);
\draw [shift={(-5.41,-5.59)}] plot[domain=0.33:0.69,variable=\t]({1*18.61*cos(\t r)+0*18.61*sin(\t r)},{0*18.61*cos(\t r)+1*18.61*sin(\t r)});
\draw [shift={(3.21,0.36)}] plot[domain=0.86:1.41,variable=\t]({1*5.59*cos(\t r)+0*5.59*sin(\t r)},{0*5.59*cos(\t r)+1*5.59*sin(\t r)});
\draw [shift={(8.09,1.57)}] plot[domain=4.82:5.4,variable=\t]({1*7.11*cos(\t r)+0*7.11*sin(\t r)},{0*7.11*cos(\t r)+1*7.11*sin(\t r)});
\draw [shift={(9.8,3.78)}] plot[domain=2.14:2.79,variable=\t]({1*6.07*cos(\t r)+0*6.07*sin(\t r)},{0*6.07*cos(\t r)+1*6.07*sin(\t r)});
\draw [shift={(8.57,-18.33)},line width=2pt]  plot[domain=1.42:1.94,variable=\t]({1*14.13*cos(\t r)+0*14.13*sin(\t r)},{0*14.13*cos(\t r)+1*14.13*sin(\t r)});
\draw [shift={(2.55,-3.05)}] plot[domain=0.25:0.69,variable=\t]({1*9.49*cos(\t r)+0*9.49*sin(\t r)},{0*9.49*cos(\t r)+1*9.49*sin(\t r)});
\draw [shift={(12.95,6.78)},line width=2pt] plot[domain=3.74:4.34,variable=\t]({1*6.4*cos(\t r)+0*6.4*sin(\t r)},{0*6.4*cos(\t r)+1*6.4*sin(\t r)});
\draw [shift={(17.64,4.37)}]  plot[domain=3.32:3.61,variable=\t]({1*7.86*cos(\t r)+0*7.86*sin(\t r)},{0*7.86*cos(\t r)+1*7.86*sin(\t r)});
\draw [shift={(10.68,-11.02)}] plot[domain=1.93:2.45,variable=\t]({1*9.26*cos(\t r)+0*9.26*sin(\t r)},{0*9.26*cos(\t r)+1*9.26*sin(\t r)});
\draw [shift={(5.83,-8.44)}] plot[domain=0.7:1.31,variable=\t]({1*6.31*cos(\t r)+0*6.31*sin(\t r)},{0*6.31*cos(\t r)+1*6.31*sin(\t r)});
\draw [shift={(6.32,25.58)}] plot[domain=4.6:4.84,variable=\t]({1*32.4*cos(\t r)+0*32.4*sin(\t r)},{0*32.4*cos(\t r)+1*32.4*sin(\t r)});
\draw [shift={(9.59,11.48)}] plot[domain=4.1:4.48,variable=\t]({1*8.52*cos(\t r)+0*8.52*sin(\t r)},{0*8.52*cos(\t r)+1*8.52*sin(\t r)});
\draw [shift={(0.2,-5.1)}] plot[domain=0.23:0.7,variable=\t]({1*15.41*cos(\t r)+0*15.41*sin(\t r)},{0*15.41*cos(\t r)+1*15.41*sin(\t r)});
\draw [shift={(3.92,-0.78)}] plot[domain=0.56:1.02,variable=\t]({1*7.05*cos(\t r)+0*7.05*sin(\t r)},{0*7.05*cos(\t r)+1*7.05*sin(\t r)});
\draw [shift={(1.06,4.02)}] plot[domain=-0.13:0.18,variable=\t]({1*6.65*cos(\t r)+0*6.65*sin(\t r)},{0*6.65*cos(\t r)+1*6.65*sin(\t r)});
\draw [shift={(8.62,4.93)}] plot[domain=3.46:3.99,variable=\t]({1*5.97*cos(\t r)+0*5.97*sin(\t r)},{0*5.97*cos(\t r)+1*5.97*sin(\t r)});
\draw [shift={(5.22,5.29)}] plot[domain=4.61:5.13,variable=\t]({1*4.88*cos(\t r)+0*4.88*sin(\t r)},{0*4.88*cos(\t r)+1*4.88*sin(\t r)});
\draw [shift={(9.34,10.11)}] plot[domain=3.98:4.49,variable=\t]({1*9.5*cos(\t r)+0*9.5*sin(\t r)},{0*9.5*cos(\t r)+1*9.5*sin(\t r)});
\draw [shift={(12.13,6.92)}] plot[domain=3.46:4.03,variable=\t]({1*7.81*cos(\t r)+0*7.81*sin(\t r)},{0*7.81*cos(\t r)+1*7.81*sin(\t r)});
\begin{scriptsize}
\fill [color=black] (3.52,-5.14) circle (5pt);
\fill [color=black] (10.64,-4.36) circle (5pt);
\fill [color=black] (11.76,-0.74) circle (5pt);
\fill [color=black] (7.22,0.84) circle (5pt);
\fill [color=black] (4.7,0.44) circle (5pt);
\fill [color=black] (7.44,-2.34) circle (5pt);
\fill [color=black] (10.62,0.82) circle (5pt);
\fill [color=black] (9.9,2.96) circle (5pt);
\fill [color=black] (6.88,4.58) circle (5pt);
\fill [color=black] (2.96,3.06) circle (5pt);
\fill [color=black] (2.7,-1.62) circle (5pt);
\fill [color=black] (1.58,1.08) circle (5pt);
\fill [color=black] (7.66,3.18) circle (5pt);
\fill [color=black] (4.7,4.5) circle (5pt);
\fill [color=black] (2.62,-6.6) circle (5pt);
\fill [color=black] (8.84,-5.5) circle (5pt);
\fill [color=black] (12.58,-3.94) circle (5pt);
\fill [color=black] (13.42,-1.9) circle (5pt);
\fill [color=black] (12.18,0.48) circle (5pt);
\fill [color=black] (11.04,4.7) circle (5pt);
\fill [color=black] (7.6,5.24) circle (5pt);
\fill [color=black] (4.1,5.88) circle (5pt);
\fill [color=qqqqff] (-0.18,5.6) circle (1.5pt);
\fill [color=qqqqff] (0.08,-1.22) circle (1.5pt);
\fill [color=qqqqff] (-1.86,5.2) circle (1.5pt);
\fill [color=qqqqff] (-3.72,8.7) circle (1.5pt);
\fill [color=black] (13.1,0.84) circle (5pt);
\fill [color=black] (11.94,4.88) circle (5pt);
\fill [color=black] (10.36,-6.56) circle (5pt);
\fill [color=qqqqff] (14.78,-5.7) circle (1.5pt);
\fill [color=qqqqff] (15.22,-1.66) circle (1.5pt);
\fill [color=qqqqff] (14.4,5.9) circle (1.5pt);
\fill [color=qqqqff] (7.84,8.88) circle (1.5pt);
\fill [color=black] (9,6.18) circle (5pt);
\fill [color=qqqqff] (6.5,8.88) circle (1.5pt);
\fill [color=qqqqff] (-4.98,-8.3) circle (1.5pt);
\fill [color=qqqqff] (1.98,-9.04) circle (1.5pt);
\fill [color=qqqqff] (12.16,-7.26) circle (1.5pt);
\fill [color=qqqqff] (15.92,-7.82) circle (1.5pt);
\fill [color=qqqqff] (-1.76,-2.94) circle (1.5pt);
\fill [color=qqqqff] (8.98,9.52) circle (1.5pt);
\fill [color=qqqqff] (9.74,11) circle (1.5pt);
\end{scriptsize}
\end{tikzpicture}
%\end{document}
\end{minipage}
	$\rightarrow$
\begin{minipage}{.4\textwidth}
	\centering
	\usetikzlibrary{arrows}
%\pagestyle{empty}
%\begin{document}
\definecolor{qqqqff}{rgb}{0,0,1}
\begin{tikzpicture}[line cap=round,line join=round,scale=0.25]%,>=triangle 45,x=1.0cm,y=1.0cm]
\clip(1.03,-6.92) rectangle (13.81,6.82);
\draw [line width=2pt] (10.64,-4.36)-- (11.76,-0.74);
\draw [line width=2pt] (11.76,-0.74)-- (10.62,0.82);
\draw (9.9,2.96)-- (7.66,3.18);
\draw [line width=2pt] (7.66,3.18)-- (6.88,4.58);
\draw [line width=2pt] (6.88,4.58)-- (4.7,4.5);
\draw [line width=2pt] (4.7,4.5)-- (2.96,3.06);
\draw [line width=2pt] (2.96,3.06)-- (1.58,1.08);
\draw (1.58,1.08)-- (2.7,-1.62);
\draw [line width=2pt] (2.7,-1.62)-- (3.52,-5.14);
\draw (7.44,-2.34)-- (2.7,-1.62);
\draw (7.44,-2.34)-- (10.62,0.82);
\draw (10.64,-4.36)-- (10.62,0.82);
\draw (10.62,0.82)-- (7.22,0.84);
\draw (7.22,0.84)-- (7.44,-2.34);
\draw (7.22,0.84)-- (7.66,3.18);
\draw [line width=2pt] (4.7,0.44)-- (2.7,-1.62);
\draw [line width=2pt] (1.58,1.08)-- (4.7,0.44);
\draw (4.7,0.44)-- (7.44,-2.34);
\draw (3.52,-5.14)-- (2.62,-6.6);
\draw (3.52,-5.14)-- (8.84,-5.5);
\draw (3.52,-5.14)-- (0.08,-1.22);
\draw (10.64,-4.36)-- (8.84,-5.5);
\draw (10.64,-4.36)-- (12.58,-3.94);
\draw (12.58,-3.94)-- (13.42,-1.9);
\draw (12.58,-3.94)-- (11.76,-0.74);
\draw (13.42,-1.9)-- (11.76,-0.74);
\draw (11.76,-0.74)-- (12.18,0.48);
\draw (12.18,0.48)-- (9.9,2.96);
\draw (12.18,0.48)-- (11.04,4.7);
\draw (7.6,5.24)-- (6.88,4.58);
\draw (4.1,5.88)-- (4.7,4.5);
\draw (4.1,5.88)-- (-0.18,5.6);
\draw (-0.18,5.6)-- (2.96,3.06);
\draw (-0.18,5.6)-- (1.58,1.08);
\draw (1.58,1.08)-- (0.08,-1.22);
\draw (2.7,-1.62)-- (0.08,-1.22);
\draw (0.08,-1.22)-- (2.62,-6.6);
\draw (2.62,-6.6)-- (8.84,-5.5);
\draw (8.84,-5.5)-- (10.36,-6.56);
\draw (10.36,-6.56)-- (12.58,-3.94);
\draw (10.36,-6.56)-- (14.78,-5.7);
\draw (14.78,-5.7)-- (12.58,-3.94);
\draw (14.78,-5.7)-- (13.42,-1.9);
\draw (13.42,-1.9)-- (13.1,0.84);
\draw (13.1,0.84)-- (12.18,0.48);
\draw (13.42,-1.9)-- (15.22,-1.66);
\draw (15.22,-1.66)-- (13.1,0.84);
\draw (13.1,0.84)-- (11.94,4.88);
\draw (14.4,5.9)-- (11.94,4.88);
\draw (11.94,4.88)-- (11.04,4.7);
\draw (7.6,5.24)-- (9,6.18);
\draw (9.9,2.96)-- (9,6.18);
\draw (11.04,4.7)-- (9,6.18);
\draw (-1.86,5.2)-- (0.08,-1.22);
\draw (-0.18,5.6)-- (-1.86,5.2);
\draw (-1.86,5.2)-- (-3.72,8.7);
\draw (-4.98,-8.3)-- (1.98,-9.04);
\draw (1.98,-9.04)-- (2.62,-6.6);
\draw (1.98,-9.04)-- (10.36,-6.56);
\draw (1.98,-9.04)-- (12.16,-7.26);
\draw (12.16,-7.26)-- (10.36,-6.56);
\draw (12.16,-7.26)-- (14.78,-5.7);
\draw (12.16,-7.26)-- (15.92,-7.82);
\draw (15.92,-7.82)-- (1.98,-9.04);
\draw (15.92,-7.82)-- (14.78,-5.7);
\draw (-3.72,8.7)-- (6.5,8.88);
\draw (-0.18,5.6)-- (6.5,8.88);
\draw (-1.86,5.2)-- (1.58,1.08);
\draw (-1.76,-2.94)-- (1.98,-9.04);
\draw (-1.76,-2.94)-- (2.62,-6.6);
\draw (-1.76,-2.94)-- (0.08,-1.22);
\draw (-1.76,-2.94)-- (-4.98,-8.3);
\draw (-1.76,-2.94)-- (-1.86,5.2);
\draw (-4.98,-8.3)-- (-3.72,8.7);
\draw (-1.86,5.2)-- (6.5,8.88);
\draw (11.04,4.7)-- (13.1,0.84);
\draw (4.7,4.5)-- (-0.18,5.6);
\draw (4.1,5.88)-- (7.84,8.88);
\draw (9,6.18)-- (7.84,8.88);
\draw (7.6,5.24)-- (7.84,8.88);
\draw (6.88,4.58)-- (7.84,8.88);
\draw (6.5,8.88)-- (7.84,8.88);
\draw (-3.72,8.7)-- (-1.76,-2.94);
\draw (11.04,4.7)-- (8.98,9.52);
\draw (7.84,8.88)-- (8.98,9.52);
\draw (9,6.18)-- (8.98,9.52);
\draw (8.98,9.52)-- (6.5,8.88);
\draw (13.42,-1.9)-- (12.18,0.48);
\draw (11.94,4.88)-- (9.74,11);
\draw (11.04,4.7)-- (9.74,11);
\draw (8.98,9.52)-- (-3.72,8.7);
\draw (8.98,9.52)-- (9.74,11);
\draw (-3.72,8.7)-- (9.74,11);
\draw [shift={(-5.41,-5.59)}] plot[domain=0.33:0.69,variable=\t]({1*18.61*cos(\t r)+0*18.61*sin(\t r)},{0*18.61*cos(\t r)+1*18.61*sin(\t r)});
\draw [shift={(3.21,0.36)}] plot[domain=0.86:1.41,variable=\t]({1*5.59*cos(\t r)+0*5.59*sin(\t r)},{0*5.59*cos(\t r)+1*5.59*sin(\t r)});
\draw [shift={(8.09,1.57)}] plot[domain=4.82:5.4,variable=\t]({1*7.11*cos(\t r)+0*7.11*sin(\t r)},{0*7.11*cos(\t r)+1*7.11*sin(\t r)});
\draw [shift={(9.8,3.78)}] plot[domain=2.14:2.79,variable=\t]({1*6.07*cos(\t r)+0*6.07*sin(\t r)},{0*6.07*cos(\t r)+1*6.07*sin(\t r)});
\draw [shift={(8.57,-18.33)},line width=2pt]  plot[domain=1.42:1.94,variable=\t]({1*14.13*cos(\t r)+0*14.13*sin(\t r)},{0*14.13*cos(\t r)+1*14.13*sin(\t r)});
\draw [shift={(2.55,-3.05)}] plot[domain=0.25:0.69,variable=\t]({1*9.49*cos(\t r)+0*9.49*sin(\t r)},{0*9.49*cos(\t r)+1*9.49*sin(\t r)});
\draw [shift={(12.95,6.78)},line width=2pt] plot[domain=3.74:4.34,variable=\t]({1*6.4*cos(\t r)+0*6.4*sin(\t r)},{0*6.4*cos(\t r)+1*6.4*sin(\t r)});
\draw [shift={(17.64,4.37)}]  plot[domain=3.32:3.61,variable=\t]({1*7.86*cos(\t r)+0*7.86*sin(\t r)},{0*7.86*cos(\t r)+1*7.86*sin(\t r)});
\draw [shift={(10.68,-11.02)}] plot[domain=1.93:2.45,variable=\t]({1*9.26*cos(\t r)+0*9.26*sin(\t r)},{0*9.26*cos(\t r)+1*9.26*sin(\t r)});
\draw [shift={(5.83,-8.44)}] plot[domain=0.7:1.31,variable=\t]({1*6.31*cos(\t r)+0*6.31*sin(\t r)},{0*6.31*cos(\t r)+1*6.31*sin(\t r)});
\draw [shift={(6.32,25.58)}] plot[domain=4.6:4.84,variable=\t]({1*32.4*cos(\t r)+0*32.4*sin(\t r)},{0*32.4*cos(\t r)+1*32.4*sin(\t r)});
\draw [shift={(9.59,11.48)}] plot[domain=4.1:4.48,variable=\t]({1*8.52*cos(\t r)+0*8.52*sin(\t r)},{0*8.52*cos(\t r)+1*8.52*sin(\t r)});
\draw [shift={(0.2,-5.1)}] plot[domain=0.23:0.7,variable=\t]({1*15.41*cos(\t r)+0*15.41*sin(\t r)},{0*15.41*cos(\t r)+1*15.41*sin(\t r)});
\draw [shift={(3.92,-0.78)}] plot[domain=0.56:1.02,variable=\t]({1*7.05*cos(\t r)+0*7.05*sin(\t r)},{0*7.05*cos(\t r)+1*7.05*sin(\t r)});
\draw [shift={(1.06,4.02)}] plot[domain=-0.13:0.18,variable=\t]({1*6.65*cos(\t r)+0*6.65*sin(\t r)},{0*6.65*cos(\t r)+1*6.65*sin(\t r)});
\draw [shift={(8.62,4.93)}] plot[domain=3.46:3.99,variable=\t]({1*5.97*cos(\t r)+0*5.97*sin(\t r)},{0*5.97*cos(\t r)+1*5.97*sin(\t r)});
\draw [shift={(5.22,5.29)}] plot[domain=4.61:5.13,variable=\t]({1*4.88*cos(\t r)+0*4.88*sin(\t r)},{0*4.88*cos(\t r)+1*4.88*sin(\t r)});
\draw [shift={(9.34,10.11)}] plot[domain=3.98:4.49,variable=\t]({1*9.5*cos(\t r)+0*9.5*sin(\t r)},{0*9.5*cos(\t r)+1*9.5*sin(\t r)});
\draw [shift={(12.13,6.92)}] plot[domain=3.46:4.03,variable=\t]({1*7.81*cos(\t r)+0*7.81*sin(\t r)},{0*7.81*cos(\t r)+1*7.81*sin(\t r)});
\begin{scriptsize}
\fill [color=black] (3.52,-5.14) circle (5pt);
\fill [color=black] (10.64,-4.36) circle (5pt);
\fill [color=black] (11.76,-0.74) circle (5pt);
\fill [color=black] (7.22,0.84) circle (5pt);
\fill [color=black] (4.7,0.44) circle (5pt);
\fill [color=black] (7.44,-2.34) circle (5pt);
\fill [color=black] (10.62,0.82) circle (5pt);
\fill [color=black] (9.9,2.96) circle (5pt);
\fill [color=black] (6.88,4.58) circle (5pt);
\fill [color=black] (2.96,3.06) circle (5pt);
\fill [color=black] (2.7,-1.62) circle (5pt);
\fill [color=black] (1.58,1.08) circle (5pt);
\fill [color=black] (7.66,3.18) circle (5pt);
\fill [color=black] (4.7,4.5) circle (5pt);
\fill [color=black] (2.62,-6.6) circle (5pt);
\fill [color=black] (8.84,-5.5) circle (5pt);
\fill [color=black] (12.58,-3.94) circle (5pt);
\fill [color=black] (13.42,-1.9) circle (5pt);
\fill [color=black] (12.18,0.48) circle (5pt);
\fill [color=black] (11.04,4.7) circle (5pt);
\fill [color=black] (7.6,5.24) circle (5pt);
\fill [color=black] (4.1,5.88) circle (5pt);
\fill [color=qqqqff] (-0.18,5.6) circle (1.5pt);
\fill [color=qqqqff] (0.08,-1.22) circle (1.5pt);
\fill [color=qqqqff] (-1.86,5.2) circle (1.5pt);
\fill [color=qqqqff] (-3.72,8.7) circle (1.5pt);
\fill [color=black] (13.1,0.84) circle (5pt);
\fill [color=black] (11.94,4.88) circle (5pt);
\fill [color=black] (10.36,-6.56) circle (5pt);
\fill [color=qqqqff] (14.78,-5.7) circle (1.5pt);
\fill [color=qqqqff] (15.22,-1.66) circle (1.5pt);
\fill [color=qqqqff] (14.4,5.9) circle (1.5pt);
\fill [color=qqqqff] (7.84,8.88) circle (1.5pt);
\fill [color=black] (9,6.18) circle (5pt);
\fill [color=qqqqff] (6.5,8.88) circle (1.5pt);
\fill [color=qqqqff] (-4.98,-8.3) circle (1.5pt);
\fill [color=qqqqff] (1.98,-9.04) circle (1.5pt);
\fill [color=qqqqff] (12.16,-7.26) circle (1.5pt);
\fill [color=qqqqff] (15.92,-7.82) circle (1.5pt);
\fill [color=qqqqff] (-1.76,-2.94) circle (1.5pt);
\fill [color=qqqqff] (8.98,9.52) circle (1.5pt);
\fill [color=qqqqff] (9.74,11) circle (1.5pt);
\end{scriptsize}
\end{tikzpicture}
%\end{document}
\end{minipage}
	
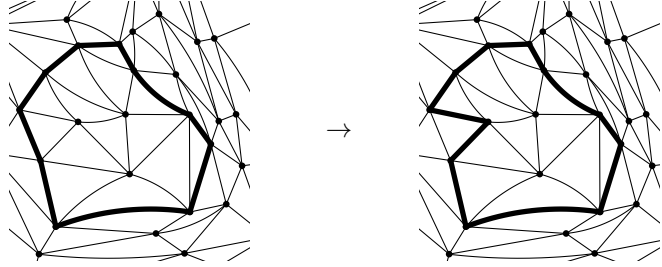
\captionof{figure}{Step type 2}
	\label{curvecontraction2}
\end{center}

2. Take a triangle $v_iv_jv_k$ in $int(c)\cup c$, which has exactly one side in common with $c$ and its third vertex is in $int(c)$. We can assume without loss of generality that this side is $(v_i,v_k)$. Throw this side out of $c$ and replace it with $(v_i,v_j)$ and $(v_j,v_k)$.
\end{lemma}

\begin{proof}
It is enough to prove that at least one of these steps always can be performed if $int(c)\cup c$ is not a triangle, since this always reduces the number of triangles contained in $int(c)\cup c$ by exactly $1$, so in the end, $int(c)\cup c$ will be a triangle.

Since $G$ is a planar graph, there are at most $n-3$ diagonals of $c$ in $G_i(c)$. Thus there exists a vertex $v$ (actually at least $2$ vertices) with no diagonals of $c$ belonging to $G_i(c)$ starting from $v$. This means that if $\left(v,v'\right)$ is an edge of $c$, then the triangle in $int(c)\cup c$ bordering $\left(v,v'\right)$ satisfies the conditions for one of the steps: its third vertex $v''$ is either the other neighbour of $v$ in $c$ or $v''$ is in $int(c)$, otherwise $\left(v,v''\right)$ would be a diagonal of $c$ as the edge connecting them is clearly in $G_i(c)$. And the first possibility means that we can perform step type 1, while the second possibility means that we can perform step type 2.
\end{proof}

\begin{lemma}
\label{lem:curvature}
The curvature of a cycle $c$ is equal to $6$ minus the number of irregular vertices in $int(c)\cup c$ (counted with multiplicity).
\end{lemma}

\begin{proof}
By using Lemma \ref{lem:curvecontraction}, we can transform $c$ into a $3$-cycle, for which $c\cap int(c)$ is a triangle using the two kind of steps there.

Now examine the change of the curvature of $c$, while performing these steps.

If we perform step type 1 and $v_j$ is a regular vertex, the curvature of $c$ does not change as $deg_{G_{ext(c)}}\left(v_i\right)-2$ and $deg_{G_{ext(c)}}\left(v_k\right)-2$ increase by $1$ each, while $deg_{G_{ext(c)}}(v_j)-2$, which was $4-2=2$ before, gets out of the sum.

If we perform step type 1 and $v_j$ is an irregular vertex with multiplicity $m$, then the curvature of $c$ increases by $m$ since $deg_{G_{ext(c)}}\left(v_i\right)-2$ and $deg_{G_{ext(c)}}\left(v_k\right)-2$ increase by $1$ each, while $deg_{G_{ext(c)}}\left(v_j\right)-2$, which was $2-m$ before, gets out of the sum.

If we perform step type 2, the curvature of $c$ does not change, since \\$deg_{G_{ext(c)}}\left(v_i\right)-2$ and $deg_{G_ext(c)}\left(v_j\right)$ increase by $1$ each and the new summand $deg_{G_ext(c)}\left(v_j\right)-2$ starts as $0-2=-2$.

So since in the end, no vertices remain in $int(c)$, the curvature of $c$ has increased by the number of irregular vertices originally in $int(c)$ counted with multiplicity. And the curvature of a triangle is trivially $6$ plus the number of irregular vertices (counted with multiplicity) among its three vertices, so the original $c$ has a curvature of $6$ minus the number of irregular vertices in $int(c)\cup c$ (counted with multiplicity).
\end{proof}

\begin{center}
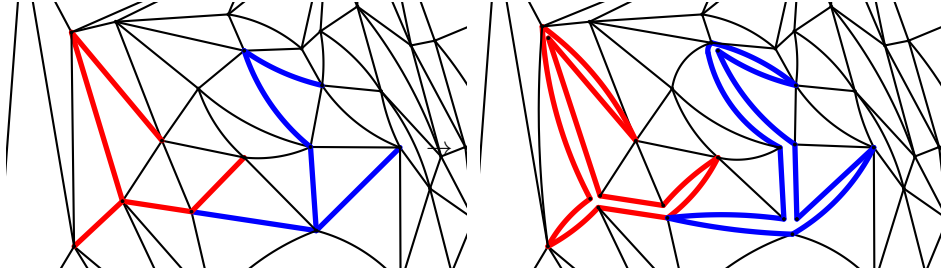

\begin{minipage}{.45\textwidth}
	\centering
	\input{Figures/Gt1t2a}
\end{minipage}
\begin{minipage}{.05\textwidth}
\centering
$\rightarrow$
\end{minipage}
\begin{minipage}{.45\textwidth}
	\centering
	\input{Figures/Gt1t2c}
\end{minipage}
	\captionof{figure}{On the left, a section of $G$ is visible with $\teegy$ highlighted in red and $\teketto$ highlighted in blue, while on the right, the corresponding section of $G\left(\teegy,\teketto\right)$ with $\teegyvesszo$ highlighted in red and $\tekettovesszo$ highlighted in blue.}
	\label{fig:Gt1t2}
\end{center}

\begin{definition}
Take two trees $\teegy$ and $\teketto$ in $G$ such that they do not cross each other, but can have common vertices. Now define $G\left(\teegy,\teketto\right)$ in the following way: take a closed walk around $\teegy$ such that after reaching a vertex $v$ following an edge $e$, always take the closest exit to the right from $e$ as the next step (see Figure \ref{fig:Gt1t2}). We stop immediately before taking the first edge of the walk in the same direction. It is obvious that if the original tree had $k$ vertices and $k-1$ edges, this closed walk has length $2k-2$: all edges are used in both directions. Now move the edges very slightly outwards (creating two edges for all edges of $G$)\footnote{The meaning of "slightly" is analogous to how it was defined in \ref{sec:converting} in the sense that we only require the new edges not to cross or touch each other (at least in their interiors), the distance of their endpoints remains smaller than $d_1$, for all points of the edges and the minimum distance from the respective endpoints remains smaller than $d_2$. We again could say that a virtual movement of the edges is satisfactory.} as seen in Figure \ref{fig:Gt1t2}, which also divides each vertex into as many copies as the degree of said vertex. Call this new cycle $\teegyvesszo$. Also, we redraw the edges starting from the vertices of $\teegy$ that are not edges of $\teegy$ (but their other endpoint might also be part of $\teegy$) such that the triangles and their orientations remain intact and no crossings are made. Now do the same for $\teketto$ (call the arising cycle $\tekettovesszo$). These changes can also be seen in Figure \ref{fig:Gt1t2}). Now define $V\left(G\left(\teegy,\teketto\right)\right)$, $E\left(G\left(\teegy,\teketto\right)\right)$, $\Delta\left(G\left(\teegy,\teketto\right)\right)$, $E_o\left(G\left(\teegy,\teketto\right)\right)$ and $\Delta_o\left(G\left(\teegy,\teketto\right)\right)$ analogously to as we defined them for $G$. Note that $V\left(G\left(\teegy,\teketto\right)\right)$, $E_o\left(G\left(\teegy,\teketto\right)\right)$ and $\Delta_o\left(G\left(\teegy,\teketto\right)\right)$ are also disjoint, but they only cover $S\setminus\left(int\left(\teegyvesszo\right)\cup int\left(\tekettovesszo\right)\right)$.
\end{definition}

\begin{definition}
Now we also can define the local curvature for vertices within cycles within $G\left(\teegy,\teketto\right)$ as $$lc\left(\vtx_i\left(c\right),c\right)=\left\lvert\left\lbrace\delta\in\Delta_o\left(G\left(\teegy,\teketto\right)\right)\vert\vtx_i\left(c\right)\in Cl\delta\right\rbrace\right\rvert-3$$ and the curvature of cycles within $G\left(\teegy,\teketto\right)$ as $$\curv_{G\left(\teegy,\teketto\right)}(c)=\sum\limits_{i=1}^{\lvert E(c)\rvert}{lc\left(\vtx_i\left(c\right)\right)}.$$
\end{definition}

\begin{lemma}
\label{lem:treecurvature}
If $\teegy$ and $\teketto$ are two trees in $G$ containing all irregular vertices and $c$ is a cycle in $G\left(\teegy,\teketto\right)$ for which $int\left(\teegyvesszo\right)\subseteq int(c)$ and $int\left(\tekettovesszo\right)\subseteq ext(c)$, then $\curv_{G\left(\teegy,\teketto\right)}(c)=6-\left\lvert V\left(\teegyvesszo\right)\cap I\right\rvert$.
\end{lemma}

\begin{proof}
For any vertex $v$ of $\teegy$ with degree $k$ within $\teegy$, the sum of the local curvatures of the vertices of $G\left(\teegy,\teketto\right)$ corresponding to $v$ is exactly $6-3k$ if $v$ is regular and $6-m-3k$ if it is irregular with multiplicity $m$. Thus, $\curv_{G\left(\teegy,\teketto\right)}\left(\teegy\right)=6\cdot\left\lvert V\left(\teegy\right)\right\rvert-6\cdot\left\lvert E\left(\teegy\right)\right\rvert-M$, if $M$ is the sum of the multiplicities of the irregular vertices belonging to $\teegy$.

For an arbitrary cycle fulfilling the condition of the lemma, the proof works similarly as for Lemma \ref{lem:curvature}, where we used Lemma \ref{lem:curvecontraction}.
\end{proof}

Now define a graph $H$ as follows (see Figure \ref{3szogracs}). Its vertices are $i_1,i_2,...,i_{12}$ and they are regarded as distinct vertices even if they are not distinct in $G$. Two vertices are connected if they have graph distance at most $3$ in $G$ (this obviously includes pairs of coinciding (non-distinct) irregular vertices).

\begin{center}
\begin{minipage}{.41\textwidth}
	\centering
	\input{Figures/3szogracs1}
\end{minipage}
\begin{minipage}{.16\textwidth}
	\vskip3cm
	\centering
	\usetikzlibrary{arrows}
%\pagestyle{empty}
%\begin{document}
\definecolor{ffzztt}{rgb}{1,0.6,0.2}
\definecolor{zzqqqq}{rgb}{0.6,0,0}
\definecolor{qqzzqq}{rgb}{0,0.6,0}
\definecolor{ffqqff}{rgb}{1,0,1}
\definecolor{qqqqff}{rgb}{0,0,1}
\definecolor{ffqqqq}{rgb}{1,0,0}
\begin{tikzpicture}[line cap=round,line join=round]%,>=triangle 45,x=1.0cm,y=1.0cm]
\clip(-0.6,-0.5) rectangle (1.25,2);
\draw [line width=2pt,color=ffqqqq] (0,0)-- (0,1);
\draw [color=qqqqff] (0,1)-- (-0.5,1.87);
\draw [line width=2pt,color=ffqqff] (-0.5,1.87)-- (0.5,1.87);
\draw [line width=2pt,color=qqzzqq] (0.5,1.87)-- (0,1);
\draw [line width=2pt,color=zzqqqq] (1,0)-- (1,1);
\draw [line width=2pt,color=ffzztt] (1,-1)-- (1,0);
\begin{scriptsize}
\fill [color=black] (0,0) circle (1.5pt);
\draw[color=black] (0.21,0.03) node {$i_4$};
\fill [color=black] (0,1) circle (1.5pt);
\draw[color=black] (0.21,1.03) node {$i_2$};
\fill [color=black] (-0.5,1.87) circle (1.5pt);
\draw[color=black] (-0.24,1.74) node {$i_1$};
\fill [color=black] (0.5,1.87) circle (1.5pt);
\draw[color=black] (0.65,1.74) node {$i_3$};
\fill [color=black] (1,1) circle (1.5pt);
\draw[color=black] (1,1.18) node {$i_5$};
\fill [color=black] (1,0) circle (1.5pt);
\draw[color=black] (0.8,0.03) node {$i_6$};
\fill [color=black] (1,-1) circle (1.5pt);
\end{scriptsize}
\end{tikzpicture}
%\end{document}
\end{minipage}
\begin{minipage}{.41\textwidth}
	\input{Figures/3szogracs3}
\end{minipage}
\captionof{figure}{A part of $G$ with the chosen paths connecting the vertices of $I$ highlighted (left), the part of $H$ corresponding to this part of $G$ (middle) and the corresponding part of $G$ with $\teegy$ and $\teketto$ denoted by bold (right).}
\label{3szogracs}
\end{center}

Now we have two cases:

\vskip0.5cm

{\bf Case 1:}

$H$ either has one connected component or it has two, both having a cardinality of $6$. Call these subcases Case 1a and Case 1b, respectively.

\begin{lemma}
If Case 1a holds, then there exist two subtrees $\teegy$ and $\teketto$ of $G$ that meet in exactly one vertex, they do not cross each other, they together cover all the vertices from $I$ and $max\left(E\left(\teegy\right),E\left(\teketto\right)\right)\le22$.
\end{lemma}

\begin{proof}
Take the smallest subtree $t_0$ of $G$ that contains all the elements of $I$. Suppose that $\left\lvert E\left(t_0\right)\right\rvert>33$. In this case, we could take a spanning tree $t\left(H\right)$ in $H$ and take a shortest path in $G$ between all elements of $I$ that are adjacent in $t\left(H\right)$ (for the irregular vertices of $G$ with multiplicity more than $1$, this path has length $0$, but in all cases it has length at most $3$). The union $U_H$ of these paths is a connected subgraph of $G$ containing all the vertices from $I$ and its edge-count is at most $33$, so if we take an arbitrary spanning tree in $U_H$, it will be shorter than $t_0$, a contradiction. So $\left\lvert E\left(t_0\right)\right\rvert\le33$. Now we will prove that there exist two subtrees of $t_0$ that fulfill the criteria of the lemma. Take all divisions of $t_0$ into two trees such that they cover all of $t_0$, but only meet in one vertex, and choose one in which the maximum of the edge-counts of the two trees is minimal: these will be $\teegy$ and $\teketto$. It is obvious that together they cover all vertices from $I$, so we only have to prove that $max\left(E\left(\teegy\right),E\left(\teketto\right)\right)\le22$. Now let the joint vertex of $\teegy$ and $\teketto$ be $v$. $v$ divides $t_0$ into at most $6$ subtrees (in fact, less as if $v$ would be a regular vertex of $G$ with at least $5$ of its vertices in $t_0$, then $t_0$ could be easily redrawn so that its vertex set would be $V\left(t_0\right)\setminus\left\lbrace v\right\rbrace$, which would contradict the choice of $t_0$), let the set of these trees be $F$. Also, suppose that one tree from $F$ has more than $17$ edges. Then this tree must be one of $\teegy$ and $\teketto$, we can assume without loss of generality that it is $\teegy$. Now we can modify our choice of $\teegy$ and $\teketto$ by removing one edge from $\teegy$ (the only edge of $\teegy$ starting from $v$) and adding it to $\teketto$, and this reduces $max\left(E\left(\teegy\right),E\left(\teketto\right)\right)$, a contradiction. So all of the trees from $F$ have at most $17$ edges. Now take the union $F_1$ of some trees from $F$ that are placed consecutively around $v$ such that their total edge-length does not exceed $17$ and it is maximal in terms of expansion among the subsets of $F$ with these properties. Now create a union $F_2$ with the same properties from $F\setminus F_1$, starting with the counterclockwise next tree after those in $F$, but now we only require $F_2$ to be maximal in the above-described sense within $F\setminus F_1$. Now denote the union of the remaining trees from $F$ by $F_3$. Per definition, $\left\lvert E\left(F_1\right)\cup E\left(F_2\right)\right\rvert\ge18$, so all of $F_1$, $F_2$ and $F_3$ have a maximum edge-count of $17$ and since $ E\left(F_1\right)\cup E\left(F_2\right)\cup E\left(F_3\right)=E(F)$, at least one of them has an edge-count of at least $11$, meaning that in case this $F_i$ is $\teegy$, while the union of the other two $F_i$'s is $\teketto$, the statement holds (it is also possible that the real $\teegy$ and $\teketto$ have an even smaller $max\left(E\left(\teegy\right),E\left(\teketto\right)\right)$).
\end{proof}

\begin{lemma}
If Case 1b holds, then there exist two disjoint subtrees $\teegy$ and $\teketto$ of $G$ that cover six vertices of $I$ each (counted with multiplicity) and \\$max\left(E\left(\teegy\right),E\left(\teketto\right)\right)\le15$.
\end{lemma}

\begin{proof}
Denote the two connected components of $H$ by $H_1$ and $H_2$. Also, take a spanning tree $t\left(H_1\right)$ for $H_1$, then choose an arbitrary shortest path in $G$ for each edge of this tree. Now take a spanning tree of the union of these paths and call the resulting graph $H_1'$. Do the same procedure for $H_2$ and call the resulting graph $H_2'$. Note that per definition, all vertices of $H_1'$ have a graph distance of at most one from at least one vertex of $H_1$, meaning that $H_1'$ and $H_2'$ cannot have any common vertices, as that would mean that some vertex of $H_1$ has graph distance (with respect to $G$) at most two from some vertex of $H_2$, and thus, it would contradict the definition of $H_1$ and $H_2$. And since $H_1'$ is a subgraph of the union of $5$ paths all of length at most $3$, it has at most $15$ edges and the same applies for $H_2'$.
\end{proof}

Now we will use the following lemma:

\begin{center}
	\input{Figures/Case2b3}
	\captionof{figure}{}
	\label{Case2b3}
\end{center}

\begin{lemma}
\label{lem:c2c3}
There exists a series of cycles $c_{1,0},...,c_{1,p}$ for some $p$) in $G\left(\teegy,\teketto\right)$ with $c_{1,0}=\teegyvesszo$ and $c_{1,p}=\tekettovesszo$ satisfying $3$ conditions:

1) All of them have graph length at most $44$, and thus, broken line length less than $44d_1$.

2) For any $i,j$ with $\left\lvert i-j\right\rvert=1$ and any vertex $v$ of $c_{1,i}$, there is a vertex of $c_{1,j}$ neighbouring $v$ in $G\left(\teegy,\teketto\right)$ and thus, having spherical distance less than $d_1$ from it.

3) For any $i,j$ with $\left\lvert i-j\right\rvert=1$ and any edge $e$ of $c_{1,i}$, there is a vertex of $c_{1,j}$ having graph distance at most $1$ from both of the endpoints of $e$ in $G\left(\teegy,\teketto\right)$, and thus, having spherical distance less than $d_1+d_2$ from all the points of $e$.
\end{lemma}
\begin{proof}
We will first describe an algorithm that results in a cycle in $G\left(\teegy,\teketto\right)$ that can be obtained from $\teivesszo$ for both $i=1$ and $i=2$ using the following types of steps (the $j$th iteration of this procedure is called $\teivesszo(j)$, where $\teivesszo=\teivesszo(0)$). We will always use the pre-assumption that if Case 1a holds, all $\teivesszo(j)$ have at least one vertex in common and that if Case 1b holds, the curvature of $\teivesszo(j-1)$ is always $0$, which can be trivially seen from the steps.

1) Choose a $k$ for which $lc\left(\vtx_k\left(\teivesszo(j-1)\right),\teivesszo(j-1)\right)=-2$ and obtain $\teivesszo(j)$ from $\teivesszo(j-1)$ by leaving $\vtx_k\left(\teivesszo(j-1)\right)$ out and instead connect $\vtx_{k-1}\left(\teivesszo(j-1)\right)$ and $\left(\vtx_{k+1}\teivesszo(j-1)\right)$ directly (they are neighbouring in $G\left(\teegy,\teketto\right)$).

2) Choose a $k$ for which $lc\left(\vtx_k\left(\teivesszo(j-1)\right),\teivesszo(j-1)\right)=-1$ and obtain $\teivesszo(j)$ from $\teivesszo(j-1)$ by leaving $\vtx_k\teivesszo(j-1)$ out and replace it with the other common neighbour of $\vtx_{k-1}\left(\teivesszo(j-1)\right)$ and $\vtx_{k+1}\left(\teivesszo(j+1)\right)$.

3) Let $\teivesszo(j)$ be the cycle defined by the neighbours of $\teivesszo(j-1)$ within $ext\left(\teivesszo(j-1)\right)$ (we only use this step if $lc\left(\vtx_k\left(\teivesszo(j-1)\right),\teivesszo(j-1)\right)=0$ for all $k$).

The procedure goes as follows:

If Case 1a holds, we will only use steps type 1 and 2 for $\teegyvesszo$, until it is possible with $int\left(\teegyvesszo(j)\right)$ remaining disjoint from $int\left(\tekettovesszo\right)$, denote the number of steps used by $p_1$. If $\teegyvesszo(p_1)$ and $\tekettovesszo$ are not the same cycle (without regard to their direction), then we will use steps type 1 and 2 for $\tekettovesszo$, until it is possible with $int\left(\tekettovesszo(j)\right)$ remaining disjoint from $int\left(\teegyvesszo\right)$, denote the number of steps used by $p_2$.

Now suppose that $\teegyvesszo\left(p_1\right)$ and $\tekettovesszo\left(p_2\right)$ are different (when regarded as non-directed cycles). This means that $\theta_1\left(p_1\right)$ has a non-negative local curvature in all vertices from $\teegyvesszo\left(p_1\right)\setminus\tekettovesszo\left(p_2\right)$ and so does $\tekettovesszo\left(p_2\right)$ in all vertices from $\left(\tekettovesszo\left(p_2\right)\setminus\teegyvesszo\left(p_1\right)\right)$, otherwise we could have used a step type 1 or 2 for one of the two cycles.

Also note that if for some $k$, both $\vtx_k\left(\teegyvesszo\left(p_1\right)\right)$ and the edges $\left(\vtx_{k-1}\left(\teegyvesszo\left(p_1\right)\right),\right.$\\$\left.\vtx_k\left(\teegyvesszo\left(p_1\right)\right)\right)$ and $\left(\vtx_k\left(\teegyvesszo\left(p_1\right)\right),\vtx_{k+1}\left(\teegyvesszo\left(p_1\right)\right)\right)$ are part of $\tekettovesszo\left(p_2\right)$ (besides being part of $\teegyvesszo\left(p_1\right)$), then $$lc\left(\vtx_k\left(\teegyvesszo\left(p_1\right)\right),\teegyvesszo\left(p_1\right)\right)+lc\left(\vtx_k\left(\teegyvesszo\left(p_1\right)\right),\tekettovesszo\left(p_2\right)\right)=0,$$ while if $\vtx_k\left(\teegyvesszo\left(p_1\right)\right)$ is part of $\tekettovesszo\left(p_2\right)$, but at least one of the edges $\left(\vtx_{k-1}\left(\teegyvesszo\left(p_1\right)\right),\right.$\\$\left.\vtx_k\left(\teegyvesszo\left(p_1\right)\right)\right)$ and $\left(\vtx_k\left(\teegyvesszo\left(p_1\right)\right),\vtx_{k+1}\left(\teegyvesszo\left(p_1\right)\right)\right)$ is not (and such a $k$ obviously exists as $\teegyvesszo\left(p_1\right)$ and $\tekettovesszo\left(p_2\right)$ coincide in some vertices but not in all of them), then $$lc\left(\vtx_k\left(\teegyvesszo\left(p_1\right)\right),\teegyvesszo\left(p_1\right)\right)+lc\left(\vtx_k\left(\teegyvesszo\left(p_1\right)\right),\tekettovesszo\left(p_2\right)\right)>0.$$ But from Lemma \ref{lem:treecurvature}, $\curv\left(\teegyvesszo\left(p_1\right)\right)+\curv\left(\tekettovesszo\left(p_2\right)\right)=0$, which is a contradiction, since above we have just shown that by summing up the local curvatures on the vertices of $\teegyvesszo\left(p_1\right)$ and $\tekettovesszo\left(p_2\right)$, we get a positive number.

If Case 1b holds, we will do a similar procedure with the difference that if we are completely stuck with using steps type 1 and 2 even for $\tekettovesszo$, then we may use step type 3 for the last $\tekettovesszo(j)$. Now suppose that we are stuck with that too, but $\teegyvesszo\left(p_1\right)$ and $\tekettovesszo\left(p_2\right)$ do not fully coincide (again when we regard them as non-directed cycles). If they have at least one joint vertex, then this is a contradiction exactly as written above, but if they do not have, then this must mean $lc\left(\vtx_k\left(\tekettovesszo\left(p_2\right))\right),\tekettovesszo(\left(p_2\right)\right)\ge0$ for all $k$ (otherwise either step 1 or step 2 could be performed for $\tekettovesszo\left(p_2\right)$, but since $\curv\left(\tekettovesszo\left(p_2\right)\right)=0$ (from Lemma \ref{lem:treecurvature}, the above means $lc\left(\vtx_k\left(\teivesszo\left(p_2\right)\right),\teivesszo\left(p_2\right)\right)=0$ for all $k$. Thus, we can apply step type 3, a contradiction.

Now define $p$ as $p_1+p_2$, define $c_{1,j}$ for $0\le j\le p_1$ as $\teegyvesszo(j)$ and define $c_{1,j}$ for $p_1\le p$ as $\tekettovesszo(p-j)$, but in the opposite direction.

It is straightforward from the construction that the above set of cycles fulfill the second and the third condition of the lemma, while the first condition follows from the fact that $\teivesszo(j)$ is shorter than $\teivesszo(j-1)$ for $i=1$ and $i=2$ and $j=1,2,...,p_1$ and $j=1,2,...,p_2$, respectively, which again is straightforward from the construction.
\end{proof}

Note that the above proof is empirically based on the fact that $G\left(\teegy,\teketto\right)$ can be mapped (with a not necessarily bijective mapping) into $G_t(z)$ for some appropriate $z$ with preserving edges, triangles and their orientations.

Now define a function $g(i)=dist'_S\left(-\left(V\left(c_{1,i}\right)\right),c_{1,i}\right)$. The sign of this function depends on whether all of $-\left(V\left(c_{1,i}\right)\right)$ is within $int\left(c_{1,i}\right)$, all of it is within $ext\left(c_{1,i}\right)$ or otherwise. Thus, $g(0)<0$ and $g(p)>0$.

Now take the first $i$, for which $g(i)$ is positive. From Lemma \ref{lem:2pointscycle} and Lemmas \ref{lem:point2cycles}, we can see that either $g(i-1)\le-d_1+\frac{d_2}{2}$ or $g(i)\ge d_1+\frac{d_2}{2}$. Thus, at least one of $c_{1,i-1}$ and $c_{1,i}$ contains two points, whose distance is at least $r\pi-d_1-\frac{d_2}{2}$, while from Lemma \ref{lem:length}, their distance could be less than $22d_1$. Thus, since $r\ge\frac{23d_1+0.5d_2}{\pi}$, this is a contradiction.

{\bf Case 2:} 

Case 1 does not hold, meaning that $H$ either has more than two components or only two, but they do not have a vertex number divisible by $6$.

\begin{lemma}
\label{lem:case3a}
If Case 2 holds, then there exists a cycle $c_2$ in $G$, all of whose vertices have graph distance at least $2$ from all the irregular vertices and which separates them into two groups so that both of the groups has a cardinality not divisible by $6$ (counted with multiplicity).
\end{lemma}

\begin{proof}
\begin{center}
\begin{minipage}{.4\textwidth}
	\centering
	\input{Figures/Case3c}
\end{minipage}
$\rightarrow$
\begin{minipage}{.4\textwidth}
	\centering
	\input{Figures/Case3c2}
\end{minipage}
\captionof{figure}{A part of $G$ (left) and the same part of $G'$ (right)}
\label{Case3c}
\end{center}

Take the graph $G'$ which we get from $G$ by deleting all the vertices that have graph distance at most $1$ (with respect to $G$) from any of the points of $I$ and all the edges that are incident to the deleted vertices.

\begin{lemma}
The connected components of $S\setminus G'$ are exactly those of the following two types:

1) Any open triangle with all vertices having graph distance at least $2$ in $G$ from all points of $I$.

2) For a component $H_a$ of $H$, take the union of the vertices of $G$ with graph distance at most $1$ in $G$ from any of the vertices of $H_a$ and the open edges and open triangles incident to these vertices.
\end{lemma}

\begin{proof}
If a set $C$ is of type 1, then it is is fully inside $S\setminus G$ and since $S\setminus G'\supseteq S\setminus G$, it is fully inside $S\setminus G'$ too. And $C$ is connected, so its points are in the same connected component. And since $\partial{C}\subseteq G'$, $C$ cannot be connected with any other point inside $S\setminus G'$.

If a set $C_a$ is of type 2 (belonging to $H_a$), then all the vertices and open edges inside it are part of $S\setminus G'$ per definition of $G'$ and it is straightforward that they are in one connected component. Also, per definition, no vertices of $C_a$ are on its own border, and also per definition, all open edges of $G$ outside $G'$ have at least one endpoint outside $G'$. So if $C_a$ borders anything from $S\setminus(G'\cup C_a)$, it also borders a vertex from this set. And per definition this vertex has graph distance at most $1$ in $G$ from a vertex of $H\setminus H_a$, which, from the definition of $C_a$, leads to a contradiction with the fact that all vertices of $H$ from different components have graph distance at least $4$ in $G$.

And since all points of $S\setminus G'$ belong to at least one set of type $1$ or type $2$, this proves that indeed, these are the only connected components of $S\setminus G'$.
\end{proof}

Now suppose that $H_a$ is a connected component of $H$ whose vertex count is different from $6$ and let $C_a$ be the connected component of $S\setminus G'$ belonging to $H_a$. Now take the connected components of $S\setminus Cl{C_a}$. Obviously, the number of irregular vertices (counted by multiplicity) contained in them is not divisible by $6$ for all of them, since they contain $12-\left\lvert V(H_a)\right\rvert$ irregular vertices in total (counted by multiplicity), and this number cannot be obtained as the sum of numbers divisible by $6$. And the boundary of each such component is a cycle in $G$, finishing the proof of Lemma \ref{lem:case3a}.
\end{proof}

\begin{definition}
For all applicable $i$, define a function $\psi_i$ from $G_H$ (as labeled in Figure \ref{triangulargrid1bfeliratos2}) to $N_2\left(\vtx_i\left(c_2\right)\right)$ in the following way: let $\psi_i(u)$ be $\vtx_i\left(c_2\right)$, let $\psi_i\left(u_1\right)$ be $\vtx_{i+1}\left(c_2\right)$ (counted modulo $l\left(c_2\right)$) and then let $\psi_i\left(u_2\right),\psi_i\left(u_3\right),...,\psi_i\left(u_6\right)$ be the other neighbours of $\vtx_i\left(c_2\right)$ starting from $\psi_i\left(u_1\right)$ in a positive order. Now for all $j=1,2,...,6$, let $\psi_i\left(u_{j,j+1}\right)$ (with the $j$'s being counted modulo $6$) be the joint neighbour of $\psi_i\left(u_j\right)$ and $\psi_i\left(u_{j+1}\right)$ that is not $\vtx_i\left(c_2\right)$ and let $\psi_i\left(u_{j,j}\right)$ be the neighbour of $\psi_i\left(u_j\right)$ between $\psi_i\left(u_{j-1,j}\right)$ and $\psi_i\left(u_{j,j+1}\right)$ (since both $\vtx_i\left(c_2\right)$ and its neighbours are regular vertices, the above definitions are meaningful).
\end{definition}

It is straightforward to see that $\psi_i\left(G_H\right)$ preserves adjacencies, triangles and the orientation of triangles. Note though that since only the immediate neighbours of $\vtx_i\left(c_2\right)$ are known to be regular, additional adjacencies and even coincidences might occur (see Figure \ref{fig:Case2hexagon6b}), but these can only cause extra restrictions for the colouring, so they do not affect the proof.

\begin{lemma}
\label{lem:case3c}
The curvature of $c_2$ is divisible by $6$.
\end{lemma}

\begin{proof}
From Lemma \ref{lem:Isbell1}, we know that $\sigma\left(\psi_i\left(G_H\right)\right)$ is part of an Isbell colouring $\sigma_I(i)$ for all $i$ ($i=1,...,l\left(c_2\right)$), moreover, since the set $N_2\left(v_i\left(c_2\right)\right)\cap N_2\left(v_{i+1}\left(c_2\right)\right)$ is the image of a $G_H^-$ within $G_H$ both when using the function $\psi_i$ and the function $\psi_{i+1}$ (the indices being counted modulo $l(c)$), this colouring is the same apart from a rotation because of Lemma \ref{lem:Isbell2} due to the fact that $G_h^+$ is a subgraph of $G_H^-$. Now define a function $g:\left\lbrace1,2,...,7\right\rbrace\times\left\lbrace1,2,...,7\right\rbrace\rightarrow\left\lbrace0,1,...,5\right\rbrace$ as follows: $g(i,j)=\frac{3}{\pi}\cdot\angle\left(\vec{uu_1},\vec{vv'}\right)$ where $v$ and $v'$ denote two arbitrary adjacent vertices of $G_H$, whose colours are $i$ and $j$ (this function exists as the ordered pair of colours of two adjacent vertices uniquely determines the corresponding vector in any Isbell colouring). It is straightforward to see that $$lc\left(\vtx_i\left(c_2\right),c_2\right)\equiv g\left(\sigma\left(\vtx_i\left(c_2\right)\right),\sigma\left(\vtx_{i+1}\left(c_2\right)\right)\right)-g\left(\sigma\left(\vtx_{i-1}\left(c_2\right)\right),\sigma\left(\vtx_i\left(c_2\right)\right)\right)\text{\,(mod\,6)}$$ 
(where the indices are counted modulo $l(c)$). Thus, 
\begin{align*}
\curv\left(c_2\right)&\equiv\sum\limits_{i=0}^{i=l\left(c_2\right)-1}{g\left(\sigma\left(\vtx_i\left(c_2\right)\right),\sigma\left(v_{i+1}\left(c_2\right)\right)\right)-g\left(\sigma\left(\vtx_{i-1}\left(c_2\right)\right),\sigma\left(v_i\left(c_2\right)\right)\right)}\\
&=0\pmod6.
\end{align*}
\end{proof}

\begin{center}
\begin{minipage}{.28\textwidth}
	\centering
	\usetikzlibrary{arrows}
%\pagestyle{empty}
%\begin{document}
\begin{tikzpicture}[line cap=round,line join=round,scale=0.75]%,>=triangle 45,x=1.0cm,y=1.0cm]
\clip(-2.25,-2) rectangle (1.35,2.12);
\draw (-1,-1.73)-- (0,-1.73);
\draw (0,1.73)-- (-1,1.73);
\draw (-1,1.73)-- (-1.5,0.87);
\draw (-1.5,0.87)-- (-2,0);
\draw (-2,0)-- (-1.5,-0.87);
\draw (-1.5,-0.87)-- (-1,-1.73);
\draw (0.5,0.87)-- (0,1.73);
\draw (0,1.73)-- (-0.5,0.87);
\draw (-0.5,0.87)-- (-1,1.73);
\draw (-1.5,0.87)-- (-0.5,0.87);
\draw (-1,0)-- (-0.5,0.87);
\draw (-2,0)-- (-1,0);
\draw (-1,0)-- (-1.5,0.87);
\draw (-1.5,-0.87)-- (-1,0);
\draw (-1,0)-- (-0.5,-0.87);
\draw (-0.5,-0.87)-- (-1.5,-0.87);
\draw (-1,-1.73)-- (-0.5,-0.87);
\draw (-0.5,-0.87)-- (0,-1.73);
\draw (0,-1.73)-- (0.5,-0.87);
\draw (0.5,-0.87)-- (-0.5,-0.87);
\draw (0.5,-0.87)-- (1,0);
\draw (0.5,-0.87)-- (0,0);
\draw (1,0)-- (0,0);
\draw (-0.5,-0.87)-- (0,0);
\draw (-1,0)-- (0,0);
\draw (0,0)-- (-0.5,0.87);
\draw (-0.5,0.87)-- (0.5,0.87);
\draw (0,0)-- (0.5,0.87);
\draw (1,0)-- (0.5,0.87);
\begin{scriptsize}
\fill [color=black] (0,0) circle (1.5pt);
\fill [color=black] (1,0) circle (1.5pt);
\fill [color=black] (0.5,0.87) circle (1.5pt);
\fill [color=black] (-0.5,0.87) circle (1.5pt);
\fill [color=black] (-1,0) circle (1.5pt);
\fill [color=black] (-0.5,-0.87) circle (1.5pt);
\fill [color=black] (0.5,-0.87) circle (1.5pt);
\fill [color=black] (0,1.73) circle (1.5pt);
\fill [color=black] (-1,1.73) circle (1.5pt);
\fill [color=black] (-1.5,0.87) circle (1.5pt);
\fill [color=black] (-2,0) circle (1.5pt);
\fill [color=black] (-1.5,-0.87) circle (1.5pt);
\fill [color=black] (-1,-1.73) circle (1.5pt);
\fill [color=black] (0,-1.73) circle (1.5pt);
\end{scriptsize}
\end{tikzpicture}
%\end{document}
	\captionof{figure}{$G_H^-$}
	\label{fig:triangulargrid3}
\end{minipage}
\begin{minipage}{.68\textwidth}
	\centering
	\usetikzlibrary{arrows}
%\pagestyle{empty}
%\begin{document}
\begin{tikzpicture}[line cap=round,line join=round]%,>=triangle 45,x=1.0cm,y=1.0cm]
\clip(-3,-2.25) rectangle (3,2.35);
\draw (-1,-1.73)-- (0,-1.73);
\draw (0,-1.73)-- (1,-1.73);
\draw (1,-1.73)-- (1.5,-0.87);
\draw (1.5,-0.87)-- (2,0);
\draw (2,0)-- (1.5,0.87);
\draw (1.5,0.87)-- (1,1.73);
\draw (1,1.73)-- (0,1.73);
\draw (0,1.73)-- (-1,1.73);
\draw (-1,1.73)-- (-1.5,0.87);
\draw (-1.5,0.87)-- (-2,0);
\draw (-2,0)-- (-1.5,-0.87);
\draw (-1.5,-0.87)-- (-1,-1.73);
\draw (1,-1.73)-- (0.5,-0.87);
\draw (0.5,-0.87)-- (1.5,-0.87);
\draw (1.5,-0.87)-- (1,0);
\draw (1,0)-- (2,0);
\draw (1,0)-- (1.5,0.87);
\draw (1.5,0.87)-- (0.5,0.87);
\draw (0.5,0.87)-- (1,1.73);
\draw (0.5,0.87)-- (0,1.73);
\draw (0,1.73)-- (-0.5,0.87);
\draw (-0.5,0.87)-- (-1,1.73);
\draw (-1.5,0.87)-- (-0.5,0.87);
\draw (-1,0)-- (-0.5,0.87);
\draw (-2,0)-- (-1,0);
\draw (-1,0)-- (-1.5,0.87);
\draw (-1.5,-0.87)-- (-1,0);
\draw (-1,0)-- (-0.5,-0.87);
\draw (-0.5,-0.87)-- (-1.5,-0.87);
\draw (-1,-1.73)-- (-0.5,-0.87);
\draw (-0.5,-0.87)-- (0,-1.73);
\draw (0,-1.73)-- (0.5,-0.87);
\draw (0.5,-0.87)-- (-0.5,-0.87);
\draw (0.5,-0.87)-- (1,0);
\draw (0.5,-0.87)-- (0,0);
\draw (1,0)-- (0,0);
\draw (-0.5,-0.87)-- (0,0);
\draw (-1,0)-- (0,0);
\draw (0,0)-- (-0.5,0.87);
\draw (-0.5,0.87)-- (0.5,0.87);
\draw (0,0)-- (0.5,0.87);
\draw (1,0)-- (0.5,0.87);
\begin{scriptsize}
\fill [color=black] (0,0) circle (2pt);
\fill [color=black] (1,0) circle (2pt);
\fill [color=black] (0.5,0.87) circle (2pt);
\fill [color=black] (-0.5,0.87) circle (2pt);
\fill [color=black] (-1,0) circle (2pt);
\fill [color=black] (-0.5,-0.87) circle (2pt);
\fill [color=black] (0.5,-0.87) circle (2pt);
\fill [color=black] (1.5,0.87) circle (2pt);
\fill [color=black] (1,1.73) circle (2pt);
\fill [color=black] (0,1.73) circle (2pt);
\fill [color=black] (-1,1.73) circle (2pt);
\fill [color=black] (-1.5,0.87) circle (2pt);
\fill [color=black] (-2,0) circle (2pt);
\fill [color=black] (-1.5,-0.87) circle (2pt);
\fill [color=black] (-1,-1.73) circle (2pt);
\fill [color=black] (0,-1.73) circle (2pt);
\fill [color=black] (1,-1.73) circle (2pt);
\fill [color=black] (1.5,-0.87) circle (2pt);
\fill [color=black] (2,0) circle (2pt);
\draw[color=black] (0,0.2) node {\contour{white}{\large{$u$}}};
\draw[color=black] (1,0.2) node {\contour{white}{\large{$u_1$}}};
\draw[color=black] (0.5,1.07) node {\contour{white}{\large{$u_2$}}};
\draw[color=black] (-0.5,1.07) node {\contour{white}{\large{$u_3$}}};
\draw[color=black] (-1,0.2) node {\contour{white}{\large{$u_4$}}};
\draw[color=black] (-0.5,-0.57) node {\contour{white}{\large{$u_5$}}};
\draw[color=black] (0.5,-0.57) node {\contour{white}{\large{$u_6$}}};
\draw[color=black] (2,0) node [anchor=west] {\large{$u_{1,1}$}};
\draw[color=black] (1.5,0.87) node [anchor=west] {\large{$u_{1,2}$}};
\draw[color=black] (1.5,1.83) node [anchor=south] {\large{$u_{2,2}$}};
\draw[color=black] (0.5,1.83) node [anchor=south] {\large{$u_{2,3}$}};
\draw[color=black] (-1,1.83) node [anchor=south] {\large{$u_{3,3}$}};
\draw[color=black] (-1.5,0.97) node [anchor=east] {\large{$u_{3,4}$}};
\draw[color=black] (-2,0.1) node [anchor=east] {\large{$u_{4,4}$}};
\draw[color=black] (-1.5,-0.77) node [anchor=east] {\large{$u_{4,5}$}};
\draw[color=black] (-1.25,-1.73) node [anchor=north] {\large{$u_{5,5}$}};
\draw[color=black] (0.5,-1.73) node [anchor=north] {\large{$u_{5,6}$}};
\draw[color=black] (1.6,-1.63) node [anchor=north] {\large{$u_{6,6}$}};
\draw[color=black] (1.6,-0.87) node [anchor=west] {\large{$u_{6,1}$}};

\end{scriptsize}
\end{tikzpicture}
%\end{document}
	
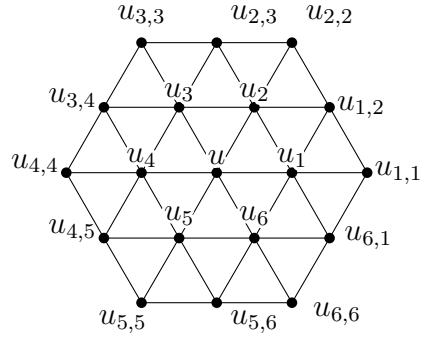
\captionof{figure}{The labeling of $G_H$ we use}
	\label{triangulargrid1bfeliratos2}
\end{minipage}
\end{center}

\begin{center}
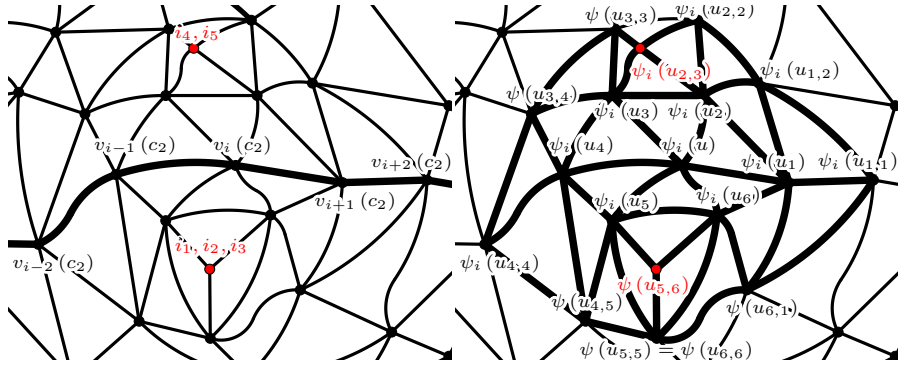

\begin{minipage}{.48\textwidth}
	\centering
	\input{Figures/Case2hexagon5b}
\end{minipage}
\begin{minipage}{.48\textwidth}
	\centering
	\input{Figures/Case2hexagon6b}
\end{minipage}
\captionof{figure}{A portion of $G$ around $\vtx_i\left(c_2\right)$ and the image of $\psi_i$ shown on the same part}
\label{fig:Case2hexagon6b}
\end{center}

So Lemma \ref{lem:case3c} leads to a contradiction as per definition, the number of irregular vertices in $int\left(c_2\right)\cup c_2$ is not divisible by $6$, so according to Lemma \ref{lem:curvature} $curv\left(c_2\right)$ is also not divisible by $6$.

\end{proof}

\section{Concluding remarks}

\begin{center}
\begin{minipage}{.4\textwidth}
	\centering
	\usetikzlibrary{arrows}
%\pagestyle{empty}
%\begin{document}
%\definecolor{uququq}{rgb}{0.25,0.25,0.25}
\begin{tikzpicture}[line cap=round,line join=round]%,>=triangle 45,x=1.0cm,y=1.0cm]
\clip(-0.25,-1.4) rectangle (2.9,10.32);
\draw (0,0)-- (0.76,0.65);
\draw (0.76,0.65)-- (1.7,0.33);
\draw (1.7,0.33)-- (2.65,0);
\draw (2.46,0.98)-- (2.65,0);
\draw (2.46,0.98)-- (1.7,0.33);
\draw (1.7,0.33)-- (1.51,1.31);
\draw (1.51,1.31)-- (2.46,0.98);
\draw (1.51,1.31)-- (0.76,0.65);
\draw (0,4.58)-- (0.19,3.6);
\draw (0.19,3.6)-- (0.38,2.62);
\draw (0.38,2.62)-- (0.57,1.64);
\draw (0.57,1.64)-- (0.76,0.65);
\draw (1.51,1.31)-- (0.57,1.64);
\draw (0.57,1.64)-- (1.32,2.29);
\draw (1.32,2.29)-- (1.51,1.31);
\draw (1.51,1.31)-- (2.27,1.96);
\draw (2.27,1.96)-- (2.46,0.98);
\draw (2.27,1.96)-- (1.32,2.29);
\draw (1.32,2.29)-- (0.38,2.62);
\draw (0.38,2.62)-- (1.13,3.27);
\draw (1.13,3.27)-- (1.32,2.29);
\draw (1.32,2.29)-- (2.08,2.95);
\draw (2.08,2.95)-- (2.27,1.96);
\draw (2.08,2.95)-- (1.13,3.27);
\draw (1.13,3.27)-- (0.19,3.6);
\draw (0.19,3.6)-- (0.94,4.26);
\draw (0.94,4.26)-- (0,4.58);
\draw (0.94,4.26)-- (1.13,3.27);
\draw (0.94,4.26)-- (1.89,3.93);
\draw (1.89,3.93)-- (1.13,3.27);
\draw (2.08,2.95)-- (1.89,3.93);
\draw (2.65,4.58)-- (1.89,3.93);
\draw (1.89,3.93)-- (1.7,4.91);
\draw (1.7,4.91)-- (0.94,4.26);
\draw (0.94,4.26)-- (0.76,5.24);
\draw (0.76,5.24)-- (0,4.58);
\draw (0.76,5.24)-- (1.7,4.91);
\draw (1.7,4.91)-- (2.65,4.58);
\draw (2.46,5.56)-- (2.65,4.58);
\draw (2.46,5.56)-- (1.7,4.91);
\draw (1.7,4.91)-- (1.51,5.89);
\draw (1.51,5.89)-- (2.46,5.56);
\draw (1.51,5.89)-- (0.76,5.24);
\draw (0.76,5.24)-- (0.57,6.22);
\draw (0.57,6.22)-- (1.51,5.89);
\draw (2.46,5.56)-- (2.27,6.55);
\draw (2.27,6.55)-- (1.51,5.89);
\draw (1.51,5.89)-- (1.32,6.87);
\draw (1.32,6.87)-- (0.57,6.22);
\draw (0.57,6.22)-- (0.38,7.2);
\draw (0.38,7.2)-- (1.32,6.87);
\draw (1.32,6.87)-- (2.27,6.55);
\draw (2.27,6.55)-- (2.08,7.53);
\draw (2.08,7.53)-- (1.32,6.87);
\draw (1.32,6.87)-- (1.13,7.86);
\draw (1.13,7.86)-- (0.38,7.2);
\draw (0.19,8.18)-- (0.38,7.2);
\draw (0.19,8.18)-- (1.13,7.86);
\draw (1.13,7.86)-- (2.08,7.53);
\draw (2.08,7.53)-- (1.89,8.51);
\draw (1.89,8.51)-- (1.13,7.86);
\draw (1.13,7.86)-- (0.94,8.84);
\draw (0.94,8.84)-- (0.19,8.18);
\draw (0.94,8.84)-- (1.89,8.51);
\draw (2.65,9.17)-- (1.89,8.51);
\draw (0.76,0.65)-- (0,0.92);
\draw (0,1.15)-- (0.57,1.64);
\draw (0.57,1.64)-- (0,1.83);
\draw (0,2.29)-- (0.38,2.62);
\draw (0.38,2.62)-- (0,2.75);
\draw (0,3.44)-- (0.19,3.6);
\draw (0.19,3.6)-- (0,3.67);
\draw (2.65,0.92)-- (2.46,0.98);
\draw (2.46,0.98)-- (2.65,1.15);
\draw (2.65,1.83)-- (2.27,1.96);
\draw (2.27,1.96)-- (2.65,2.29);
\draw (2.65,2.75)-- (2.08,2.95);
\draw (2.08,2.95)-- (2.65,3.44);
\draw (2.65,3.67)-- (1.89,3.93);
\draw (0.19,8.18)-- (0,9.17);
\draw (0,9.17)-- (0.94,8.84);
\draw (0.76,5.24)-- (0,5.5);
\draw (0,5.73)-- (0.57,6.22);
\draw (0.57,6.22)-- (0,6.42);
\draw (0,6.87)-- (0.38,7.2);
\draw (0.38,7.2)-- (0,7.33);
\draw (0,8.02)-- (0.19,8.18);
\draw (2.65,5.5)-- (2.46,5.56);
\draw (2.46,5.56)-- (2.65,5.73);
\draw (2.65,6.42)-- (2.27,6.55);
\draw (2.27,6.55)-- (2.65,6.87);
\draw (2.65,7.33)-- (2.08,7.53);
\draw (2.08,7.53)-- (2.65,8.02);
\draw (2.65,8.25)-- (1.89,8.51);
\draw (0,0)-- (2.65,0);
\draw (2.65,0)-- (2.65,9.17);
\draw (2.65,9.17)-- (0,9.17);
\draw (0,9.17)-- (0,0);
\draw (0,8.25)-- (0.19,8.18);
\draw(1.32,-0.42) circle (0.42cm);
\draw(1.32,9.59) circle (0.42cm);
\draw (1.89,8.51)-- (1.32,9.17);
\draw (1.32,9.17)-- (0.94,8.84);
\draw (1.32,9.17)-- (1.32,10.01);
\draw (1.7,0.33)-- (1.32,0);
\draw (1.32,0)-- (0.76,0.65);
\draw (1.32,0)-- (1.32,-0.84);
\begin{scriptsize}
\fill [color=black] (0,0) circle (1.5pt);
\draw[color=black] (0.09,0.15) node {1};
\fill [color=black] (2.65,0) circle (1.5pt);
\draw[color=black] (2.74,0.15) node {1};
\fill [color=black] (0.76,0.65) circle (1.5pt);
\draw[color=black] (0.85,0.81) node {2};
\fill [color=black] (1.7,0.33) circle (1.5pt);
\draw[color=black] (1.8,0.48) node {4};
\fill [color=black] (2.46,0.98) circle (1.5pt);
\draw[color=black] (2.56,1.13) node {3};
\fill [color=black] (1.51,1.31) circle (1.5pt);
\draw[color=black] (1.61,1.46) node {6};
\fill [color=black] (2.27,1.96) circle (1.5pt);
\draw[color=black] (2.36,2.11) node {7};
\fill [color=black] (0.57,1.64) circle (1.5pt);
\draw[color=black] (0.67,1.79) node {5};
\fill [color=black] (1.32,2.29) circle (1.5pt);
\draw[color=black] (1.41,2.44) node {1};
\fill [color=black] (0.38,2.62) circle (1.5pt);
\draw[color=black] (0.47,2.77) node {4};
\fill [color=black] (1.13,3.27) circle (1.5pt);
\draw[color=black] (1.23,3.43) node {3};
\fill [color=black] (2.08,2.95) circle (1.5pt);
\draw[color=black] (2.17,3.1) node {2};
\fill [color=black] (1.89,3.93) circle (1.5pt);
\draw[color=black] (1.98,4.08) node {5};
\fill [color=black] (0.19,3.6) circle (1.5pt);
\draw[color=black] (0.28,3.76) node {6};
\fill [color=black] (0.94,4.26) circle (1.5pt);
\draw[color=black] (1.04,4.41) node {7};
\fill [color=black] (0,4.58) circle (1.5pt);
\draw[color=black] (0.09,4.74) node {1};
\fill [color=black] (1.7,4.91) circle (1.5pt);
\draw[color=black] (1.8,5.06) node {4};
\fill [color=black] (0.76,5.24) circle (1.5pt);
\draw[color=black] (0.82,5.39) node {$I$};
\fill [color=black] (1.51,5.89) circle (1.5pt);
\draw[color=black] (1.61,6.04) node {6};
\fill [color=black] (2.65,4.58) circle (1.5pt);
\draw[color=black] (2.74,4.74) node {1};
\fill [color=black] (2.46,5.56) circle (1.5pt);
\draw[color=black] (2.56,5.72) node {3};
\fill [color=black] (0.57,6.22) circle (1.5pt);
\draw[color=black] (0.67,6.37) node {5};
\fill [color=black] (2.27,6.55) circle (1.5pt);
\draw[color=black] (2.36,6.7) node {7};
\fill [color=black] (1.32,6.87) circle (1.5pt);
\draw[color=black] (1.41,7.02) node {1};
\fill [color=black] (2.08,7.53) circle (1.5pt);
\draw[color=black] (2.17,7.68) node {2};
\fill [color=black] (0.38,7.2) circle (1.5pt);
\draw[color=black] (0.47,7.35) node {4};
\fill [color=black] (1.13,7.86) circle (1.5pt);
\draw[color=black] (1.23,8) node {3};
\fill [color=black] (1.89,8.51) circle (1.5pt);
\draw[color=black] (1.98,8.66) node {5};
\fill [color=black] (2.65,9.17) circle (1.5pt);
\draw[color=black] (2.74,9.32) node {1};
\fill [color=black] (0.19,8.18) circle (1.5pt);
\draw[color=black] (0.28,8.33) node {6};
\fill [color=black] (0.94,8.84) circle (1.5pt);
\draw[color=black] (1.04,8.99) node {7};
\fill [color=black] (0,9.17) circle (1.5pt);
\draw[color=black] (0.09,9.32) node {1};
\fill [color=black] (1.32,0) circle (1.5pt);
\draw[color=black] (1.41,0.15) node {7};
\fill [color=black] (1.32,9.17) circle (1.5pt);
\draw[color=black] (1.41,9.32) node {4};
\fill [color=black] (1.32,-0.84) circle (1.5pt);
\draw[color=black] (1.41,-0.69) node {1};
\fill [color=black] (1.32,10.01) circle (1.5pt);
\draw[color=black] (1.41,10.16) node {1};
\end{scriptsize}
\end{tikzpicture}
%\end{document}
	\captionof{figure}{}
	\label{fig:hosszukasgomb14}
\end{minipage}
\begin{minipage}{.4\textwidth}
	\centering
	\usetikzlibrary{arrows}
%\pagestyle{empty}
%\begin{document}
\definecolor{qqffqq}{rgb}{0,1,0}
\definecolor{ffzztt}{rgb}{1,0.6,0.2}
\definecolor{ffqqqq}{rgb}{1,0,0}
\definecolor{zzqqzz}{rgb}{0.6,0,0.6}
\definecolor{ffffqq}{rgb}{1,1,0}
\definecolor{zzqqqq}{rgb}{0.6,0,0}
\definecolor{qqqqff}{rgb}{0,0,1}
\begin{tikzpicture}[line cap=round,line join=round]%,>=triangle 45,x=1.0cm,y=1.0cm]
\clip(-0.25,-1.4) rectangle (2.9,10.32);
\fill[line width=0pt,color=qqqqff,fill=qqqqff,fill opacity=1.0] (0,0) -- (0.53,0) -- (0.57,0.11) -- (0.19,0.55) -- (0,0.51) -- cycle;
\fill[line width=0pt,color=zzqqqq,fill=zzqqqq,fill opacity=1.0] (0.53,0) -- (1.32,0) -- (1.13,0.22) -- (0.57,0.11) -- cycle;
\fill[line width=0pt,color=ffffqq,fill=ffffqq,fill opacity=1.0] (1.32,0) -- (2.12,0) -- (2.27,0.44) -- (1.89,0.87) -- (1.32,0.76) -- (1.13,0.22) -- cycle;
\fill[line width=0pt,color=qqqqff,fill=qqqqff,fill opacity=1.0] (2.12,0) -- (2.65,0) -- (2.65,0.51) -- (2.27,0.44) -- cycle;
\fill[line width=0pt,color=zzqqzz,fill=zzqqzz,fill opacity=1.0] (2.65,0.51) -- (2.65,1.53) -- (2.08,1.42) -- (1.89,0.87) -- (2.27,0.44) -- cycle;
\fill[line width=0pt,color=zzqqqq,fill=zzqqqq,fill opacity=1.0] (2.65,1.53) -- (2.65,2.29) -- (2.46,2.51) -- (1.89,2.4) -- (1.7,1.85) -- (2.08,1.42) -- cycle;
\fill[line width=0pt,color=ffffqq,fill=ffffqq,fill opacity=1.0] (2.65,2.29) -- (2.65,3.06) -- (2.46,2.51) -- cycle;
\fill[line width=0pt,color=ffqqqq,fill=ffqqqq,fill opacity=1.0] (2.65,3.06) -- (2.65,4.07) -- (2.46,4.04) -- (2.27,3.49) -- cycle;
\fill[line width=0pt,color=qqqqff,fill=qqqqff,fill opacity=1.0] (2.65,4.07) -- (2.65,5.09) -- (2.27,5.02) -- (2.08,4.47) -- (2.46,4.04) -- cycle;
\fill[line width=0pt,color=zzqqzz,fill=zzqqzz,fill opacity=1.0] (2.65,5.09) -- (2.65,6.11) -- (2.08,6) -- (1.89,5.46) -- (2.27,5.02) -- cycle;
\fill[line width=0pt,color=zzqqqq,fill=zzqqqq,fill opacity=1.0] (2.65,6.11) -- (2.65,6.87) -- (2.46,7.09) -- (1.89,6.98) -- (1.7,6.44) -- (2.08,6) -- cycle;
\fill[line width=0pt,color=ffffqq,fill=ffffqq,fill opacity=1.0] (2.65,6.87) -- (2.65,7.64) -- (2.46,7.09) -- cycle;
\fill[line width=0pt,color=ffqqqq,fill=ffqqqq,fill opacity=1.0] (2.65,7.64) -- (2.65,8.66) -- (2.46,8.62) -- (2.27,8.07) -- cycle;
\fill[line width=0pt,color=qqqqff,fill=qqqqff,fill opacity=1.0] (2.65,8.66) -- (2.65,9.17) -- (2.12,9.17) -- (2.08,9.06) -- (2.46,8.62) -- cycle;
\fill[line width=0pt,color=ffffqq,fill=ffffqq,fill opacity=1.0] (2.12,9.17) -- (1.32,9.17) -- (1.51,8.95) -- (2.08,9.06) -- cycle;
\fill[line width=0pt,color=zzqqqq,fill=zzqqqq,fill opacity=1.0] (1.32,9.17) -- (0.53,9.17) -- (0.38,8.73) -- (0.76,8.29) -- (1.32,8.4) -- (1.51,8.95) -- cycle;
\fill[line width=0pt,color=qqqqff,fill=qqqqff,fill opacity=1.0] (0.53,9.17) -- (0,9.17) -- (0,8.66) -- (0.38,8.73) -- cycle;
\fill[line width=0pt,color=ffqqqq,fill=ffqqqq,fill opacity=1.0] (0,8.66) -- (0,7.64) -- (0.57,7.75) -- (0.76,8.29) -- (0.38,8.73) -- cycle;
\fill[line width=0pt,color=ffffqq,fill=ffffqq,fill opacity=1.0] (0,7.64) -- (0,6.87) -- (0.19,6.66) -- (0.76,6.76) -- (0.94,7.31) -- (0.57,7.75) -- cycle;
\fill[line width=0pt,color=zzqqqq,fill=zzqqqq,fill opacity=1.0] (0,6.87) -- (0,6.11) -- (0.19,6.66) -- cycle;
\fill[line width=0pt,color=zzqqzz,fill=zzqqzz,fill opacity=1.0] (0,6.11) -- (0,5.09) -- (0.19,5.13) -- (0.38,5.67) -- cycle;
\fill[line width=0pt,color=qqqqff,fill=qqqqff,fill opacity=1.0] (0,5.09) -- (0,4.07) -- (0.38,4.15) -- (0.57,4.69) -- (0.19,5.13) -- cycle;
\fill[line width=0pt,color=ffqqqq,fill=ffqqqq,fill opacity=1.0] (0,4.07) -- (0,3.06) -- (0.57,3.16) -- (0.76,3.71) -- (0.38,4.15) -- cycle;
\fill[line width=0pt,color=ffffqq,fill=ffffqq,fill opacity=1.0] (0,3.06) -- (0,2.29) -- (0.19,2.07) -- (0.76,2.18) -- (0.94,2.73) -- (0.57,3.16) -- cycle;
\fill[line width=0pt,color=zzqqqq,fill=zzqqqq,fill opacity=1.0] (0,2.29) -- (0,1.53) -- (0.19,2.07) -- cycle;
\fill[line width=0pt,color=zzqqzz,fill=zzqqzz,fill opacity=1.0] (0,1.53) -- (0,0.51) -- (0.19,0.55) -- (0.38,1.09) -- cycle;
\fill[line width=0pt,color=ffzztt,fill=ffzztt,fill opacity=1.0] (0.19,0.55) -- (0.57,0.11) -- (1.13,0.22) -- (1.32,0.76) -- (0.94,1.2) -- (0.38,1.09) -- cycle;
\fill[line width=0pt,color=ffqqqq,fill=ffqqqq,fill opacity=1.0] (1.32,0.76) -- (1.89,0.87) -- (2.08,1.42) -- (1.7,1.85) -- (1.13,1.75) -- (0.94,1.2) -- cycle;
\fill[line width=0pt,color=qqffqq,fill=qqffqq,fill opacity=1.0] (0.94,1.2) -- (1.13,1.75) -- (0.76,2.18) -- (0.19,2.07) -- (0,1.53) -- (0.38,1.09) -- cycle;
\fill[line width=0pt,color=qqqqff,fill=qqqqff,fill opacity=1.0] (1.13,1.75) -- (1.7,1.85) -- (1.89,2.4) -- (1.51,2.84) -- (0.94,2.73) -- (0.76,2.18) -- cycle;
\fill[line width=0pt,color=ffzztt,fill=ffzztt,fill opacity=1.0] (1.89,2.4) -- (2.46,2.51) -- (2.65,3.06) -- (2.27,3.49) -- (1.7,3.38) -- (1.51,2.84) -- cycle;
\fill[line width=0pt,color=zzqqzz,fill=zzqqzz,fill opacity=1.0] (1.7,3.38) -- (1.32,3.82) -- (0.76,3.71) -- (0.57,3.16) -- (0.94,2.73) -- (1.51,2.84) -- cycle;
\fill[line width=0pt,color=qqffqq,fill=qqffqq,fill opacity=1.0] (1.7,3.38) -- (2.27,3.49) -- (2.46,4.04) -- (2.08,4.47) -- (1.51,4.36) -- (1.32,3.82) -- cycle;
\fill[line width=0pt,color=zzqqqq,fill=zzqqqq,fill opacity=1.0] (1.32,3.82) -- (1.51,4.36) -- (1.13,4.8) -- (0.57,4.69) -- (0.38,4.15) -- (0.76,3.71) -- cycle;
\fill[line width=0pt,color=ffffqq,fill=ffffqq,fill opacity=1.0] (1.51,4.36) -- (2.08,4.47) -- (2.27,5.02) -- (1.89,5.46) -- (1.32,5.35) -- (1.13,4.8) -- cycle;
\fill[line width=0pt,color=ffzztt,fill=ffzztt,fill opacity=1.0] (1.13,4.8) -- (1.32,5.35) -- (0.94,5.78) -- (0.38,5.67) -- (0.19,5.13) -- (0.57,4.69) -- cycle;
\fill[line width=0pt,color=ffqqqq,fill=ffqqqq,fill opacity=1.0] (1.32,5.35) -- (1.89,5.46) -- (2.08,6) -- (1.7,6.44) -- (1.13,6.33) -- (0.94,5.78) -- cycle;
\fill[line width=0pt,color=qqffqq,fill=qqffqq,fill opacity=1.0] (1.13,6.33) -- (0.76,6.76) -- (0.19,6.66) -- (0,6.11) -- (0.38,5.67) -- (0.94,5.78) -- cycle;
\fill[line width=0pt,color=qqqqff,fill=qqqqff,fill opacity=1.0] (1.13,6.33) -- (1.7,6.44) -- (1.89,6.98) -- (1.51,7.42) -- (0.94,7.31) -- (0.76,6.76) -- cycle;
\fill[line width=0pt,color=ffzztt,fill=ffzztt,fill opacity=1.0] (1.89,6.98) -- (2.46,7.09) -- (2.65,7.64) -- (2.27,8.07) -- (1.7,7.96) -- (1.51,7.42) -- cycle;
\fill[line width=0pt,color=zzqqzz,fill=zzqqzz,fill opacity=1.0] (1.51,7.42) -- (1.7,7.96) -- (1.32,8.4) -- (0.76,8.29) -- (0.57,7.75) -- (0.94,7.31) -- cycle;
\fill[line width=0pt,color=qqffqq,fill=qqffqq,fill opacity=1.0] (1.7,7.96) -- (2.27,8.07) -- (2.46,8.62) -- (2.08,9.06) -- (1.51,8.95) -- (1.32,8.4) -- cycle;
\draw [line width=0.4pt,color=zzqqqq,fill=zzqqqq,fill opacity=1.0] (1.32,-0.42) circle (0.42cm);
\draw [line width=0.4pt,color=ffffqq,fill=ffffqq,fill opacity=1.0] (1.32,9.59) circle (0.42cm);
\draw (0,0)-- (2.65,0);
\draw (2.65,0)-- (2.65,9.17);
\draw (2.65,9.17)-- (0,9.17);
\draw (0,9.17)-- (0,0);
\draw(1.32,-0.42) circle (0.42cm);
\draw(1.32,9.59) circle (0.42cm);
\end{tikzpicture}
%\end{document}
	\captionof{figure}{}
	\label{fig:hosszukasgomb13}
\end{minipage}
\end{center}

In the following, we will discuss issues regarding the optimality of the above results, for example the necessity of the conditions in Theorem \ref{thm:main2}.

\begin{remark}
First, note that the obvious condition that no two vertices with distance at most $1$ can have the same colour (as they represent tiles that have distance at most $1$) was not even included in Theorem \ref{thm:main2}. The only way the condition that tiles of the same colour have distance more than $1$ is transferred into the conditions of Theorem \ref{thm:main2} is the nice colouring condition.
\end{remark}

Now examine the necessity of the conditions of Theorem \ref{thm:main2}.

Note that while in the proof of Theorem \ref{thm:main2}, the second case used a fully combinatorial argument, the first case was more geometric: it did not only use the combinatorial properties of $G$, but it also built on the fact that $S$ is a sphere. So the question logically arises: is it really necessary that $S$ is a sphere or there exist arbitrarily large spaces that are topologically isomorphic to the sphere, but the statement of Theorem \ref{thm:main2} cannot be generalized to them, and if so, can they be nicely $7$-tiled. The following lemma answers this question:

\begin{lemma}
There exists arbitrarily large $L$, such that a fully triangulated crossing-free graph can be drawn on the surface of the cylinder with height $L$ and a disk of radius $\left(1+\varepsilon\right)\cdot\frac{\sqrt{21}}{\pi}$ as a base such that this graph can be nicely coloured. Moreover this cylinder can be nicely $7$-tiled with the aforementioned graph being its adjacency graph.
\end{lemma}

\begin{proof}
Figure \ref{fig:hosszukasgomb14} shows a construction for the graph and Figure \ref{fig:hosszukasgomb13} shows a construction for the tiling.
\end{proof}

Now suppose that we do not require the edges to be short enough.

One could also ask the question for two-dimensional manifolds in general: is it possible to nicely tile such a manifold $M$ using only $7$ colours? First, we make a simple observation for toruses:

\begin{lemma}
There exist arbitrary large toruses in $\mathbb{R}^3$ which can be nicely tiled with $7$ colours.
\end{lemma}

\begin{proof}
The Isbell colouring is periodic both horizontally and vertically if we draw it in the way as in Figure \ref{fig:Isbell}. Also note that (as all nice tilings) it can be slightly scaled or deformed and still remains a nice tiling. Thus, if we take a torus with a large enough minor radius and a major radius that is large enough even compared to its minor radius, the Isbell colouring can be drawn on the surface of the torus.
\end{proof}

But this does not work for other bounded $2$-dimensional manifolds:

\begin{lemma}\label{lem:manifold}
Let $M$ be a bounded $2$-dimensional manifold without any non-contractible curves with diameter less than $1$. Then $M$ cannot be nicely coloured.
\end{lemma}

\begin{proof}
We can construct $G$ in a similar way as in case of a sphere (as described in Section \ref{sec:converting}). Now again we get a triangulation of $M$, thus $\chi_E(M)=\left\lvert V\left(G\right)\right\rvert-\left\lvert E\left(G\right)\right\rvert+\left\lvert\Delta\left(G\right)\right\rvert$ (where $\chi_E(M)$ denotes the Euler characteristic of $M$). And since $\left\lvert\Delta\left(G\right)\right\rvert=\frac{2}{3}\cdot\left\lvert E\left(G\right)\right\rvert$, $\left\lvert E\left(G\right)\right\rvert=3\cdot\left\lvert V\left(G\right)\right\rvert-3\cdot\chi_E(M)$. And since the Euler characteristic of $M$ is negative, this means that at least one vertex of $G$ must have degree more than $6$, thus, it cannot be nicely coloured.
\end{proof}

But it is not straightforward to generalize Theorem \ref{thm:main} from here.

\begin{lemma}
For any non-negative integer numbers $k$ and $n_{min}$ and real numbers $a_{min}$ and $A_{min}$ we can find a $2$-dimensional bounded manifold $M$ with $k$ genuses that is embeddable to $\mathbb{R}^3$ such that it has at least $n_{min}$ tiles, a diameter at least $a_{min}$ and a surface area at least $A_{min}$ and it can be nicely tiled using $4$ colours.
\end{lemma}

\begin{proof}
\begin{center}
	\centering
	%\documentclass[10pt]{article}
%\usepackage[utf8]{inputenc}
%\usepackage{pgfplots}
%\pgfplotsset{compat=1.15}
%\usepackage{mathrsfs}
\usetikzlibrary{arrows}
%\pagestyle{empty}
%\begin{document}
\definecolor{qqwuqq}{rgb}{0.,0.39215686274509803,0.}
\definecolor{qqqqff}{rgb}{0.,0.,1.}
\definecolor{ffqqqq}{rgb}{1.,0.,0.}
\begin{tikzpicture}[line cap=round,line join=round]%,>=triangle 45,x=1.0cm,y=1.0cm]
\clip(-0.1,-0.4) rectangle (9.1,1.35);
\draw [line width=2.pt,color=ffqqqq] (0.,0.)-- (0.9,0.);
\draw [line width=2.pt,color=qqqqff] (0.9,0.)-- (1.8,0.);
\draw [line width=2.pt,color=qqwuqq] (1.8,0.)-- (2.7,0.);
\draw [line width=2.pt,color=ffqqqq] (2.7,0.)-- (3.6,0.);
\draw [line width=2.pt,color=qqqqff] (3.6,0.)-- (4.5,0.);
\draw [line width=2.pt,color=qqwuqq] (4.5,0.)-- (5.4,0.);
\draw [line width=2.pt,color=ffqqqq] (5.4,0.)-- (6.3,0.);
\draw [line width=2.pt,color=qqqqff] (6.3,0.)-- (7.2,0.);
\draw [line width=2.pt,color=qqwuqq] (7.2,0.)-- (8.1,0.);
\draw [line width=2.pt,color=ffqqqq] (8.1,0.)-- (9.,0.);
\draw [line width=2.pt,color=qqqqff] (0.,0.9)-- (0.45,0.9);
\draw [line width=2.pt,color=qqwuqq] (0.45,0.9)-- (1.35,0.9);
\draw [line width=2.pt,color=ffqqqq] (1.35,0.9)-- (2.25,0.9);
\draw [line width=2.pt,color=qqqqff] (2.25,0.9)-- (3.15,0.9);
\draw [line width=2.pt,color=qqwuqq] (3.15,0.9)-- (4.05,0.9);
\draw [line width=2.pt,color=ffqqqq] (4.05,0.9)-- (4.95,0.9);
\draw [line width=2.pt,color=qqqqff] (4.95,0.9)-- (5.85,0.9);
\draw [line width=2.pt,color=qqwuqq] (5.85,0.9)-- (6.75,0.9);
\draw [line width=2.pt,color=ffqqqq] (6.75,0.9)-- (7.65,0.9);
\draw [line width=2.pt,color=qqqqff] (7.65,0.9)-- (8.55,0.9);
\draw [line width=2.pt,color=qqwuqq] (8.55,0.9)-- (9.,0.9);
\draw [line width=2.pt] (0.,0.)-- (0.,0.9);
\draw [line width=2.pt] (9.,0.)-- (9.,0.9);
\draw [line width=2.pt] (2.0192405384730834,0.)-- (2.0192405384730834,0.9);
\draw [line width=2.pt] (4.296167077630869,0.)-- (4.296167077630869,0.9);
\draw [line width=2.pt] (7.748580082126212,0.)-- (7.748580082126211,0.9);
\draw [<->,line width=1.pt] (0.,-0.1) -- (0.9,-0.1);
\draw [<->,line width=1.pt] (0.,1.) -- (0.45,1.);
\draw [<->,line width=1.pt] (0.45,1.) -- (1.35,1.);
\begin{scriptsize}
\draw[color=black] (0.4272431208505662,-0.3280181053614931) node {$0.9$};
\draw[color=black] (0.3,1.25) node {$0.45$};
\draw[color=black] (0.9,1.25) node {$0.9$};
\end{scriptsize}
\end{tikzpicture}
%\end{document}
	\captionof{figure}{}
	\label{2dimmanifold}
\end{center}

Take two parallel line segments in the plane with distance $0.9$ and tile both of them with red, blue and green segments of length $0.9$ periodically such that the difference of the two periods is $1.35$. Then put $k+1$ perpendicular segments between them such that their minimum distance is more than $1$, otherwise let their placement be arbitrary. Now colour these segments with black (see Figure \ref{2dimmanifold}). In this drawing, all tiles of the same colour have distace more than $1$: the minimal distance of the horizontal tiles of the same colour is $\sqrt{(0.45)^2+(0.9)^2}=1.006...$. So if the two segments were taken long enough, then by replacing this drawing by a system of thin enough tubes, we get the desired $M$.
\end{proof}

Such constructions are the reason why Thomassen needed condition 2 in Theorem \ref{thm:thomassen} and we needed a similar condition in Lemma \ref{lem:manifold}.

\section{Acknowledgement}

I would like to thank my supervisor Dömötör Pálvölgyi for suggesting the problem, for his help in preparing this manuscript and all his guidance throughout the last few years. I would also like to thank Gergely Ambrus for his idea on simplifying my proof.

The result was obtained by working on the Polymath 16 project and is related, but not directly connected to it.

\end{document}